\documentclass[11pt]{article}

\usepackage[a4paper,margin=0.7in]{geometry}

\usepackage{amsmath,amssymb,amsthm,amsfonts,mathtools,bm}
\usepackage{graphicx}
\usepackage{multirow,booktabs,array}
\usepackage{enumitem}
\usepackage{algorithm}
\usepackage{algpseudocode}
\usepackage{mathrsfs}
\usepackage{xcolor}
\usepackage{float}
\usepackage{tikz}
\usepackage{subcaption}
\usepackage{epstopdf}
\usepackage{lmodern}
\usepackage{physics}
\usepackage{hyperref}

\hypersetup{
colorlinks=true,
linkcolor=blue,
citecolor=blue,
urlcolor=blue
}

\graphicspath{{figures/}}

\DeclareMathOperator{\sign}{sign}

\newcommand*{\medcap}{\mathbin{\scalebox{1.5}{\ensuremath{\cap}}}}
\newcommand{\Real}{\mathbb R}

\newtheorem{theorem}{Theorem}[section]
\newtheorem{lemma}[theorem]{Lemma}

\newtheorem{definition}[theorem]{Definition}
\newtheorem{remark}[theorem]{Remark}

\providecommand{\keywords}[1]{%
  \par\noindent\textbf{Keywords: }#1\par
}

\providecommand{\MSC}[1]{%
  \par\noindent\textbf{MSC Classification: }#1\par
}

\title{Nonlinear Parareal-Incomplete OSWR Method for the Coupled Reaction-Diffusion System: Convergence Analysis and Computational Strategies}

\author{
Gobinda Garai\thanks{Email: garai@math.cas.cz}\\
Institute of Mathematics, Czech Academy of Sciences, Prague, Czech Republic
\and
Nagaiah Chamakuri\thanks{Email: nagaiah.chamakuri@iisertvm.ac.in}\\
School of Mathematics, Indian Institute of Science Education and Research Thiruvananthapuram,\\ Trivandrum, India
}

\begin{document}

\maketitle

\abstract{This paper presents a development of a nonlinear extension of the Parareal-Incomplete Optimized Schwarz Waveform Relaxation (OSWR) method, aimed at efficiently simulating coupled reaction-diffusion systems, with particular attention to those encountered in cardiac electrophysiology.
The strategy leverages the synergy of spatial and temporal decomposition to tackle the computational challenges of large-scale, nonlinear simulations. 
The proposed approach is accompanied by a comprehensive convergence analysis. Extensive numerical experiments confirm the convergence of the nonlinear Parareal-Incomplete OSWR method. The results showcase its robust convergence behavior across various space-time subdomains, underscoring the method's reliability and effectiveness, especially in complex simulations involving multiple subdomains.}


\keywords{Coupled reaction-diffusion system, Parallel-in-Time (PinT), OSWR, Convergence analysis,  Parareal method.}

\MSC{65M12, 65Y05, 65M15, 65Y20}

\maketitle

\section{Introduction} \label{intro}

The parareal-in-time algorithm has been extensively studied to enable parallel computations by dividing the temporal integration interval into smaller segments \cite{lions2001parareal}. Similarly, optimized Schwarz waveform relaxation (OSWR) algorithms have been explored to address the computational challenges of solving large scale systems that arise from discretizing partial differential equations (PDEs) \cite{gander2007optimized}.
Our primary goal is to develop and refine parallel algorithms that combine the strengths of the parareal-in-time and OSWR methods, working effectively across both spatial and temporal domains for nonlinear PDEs. This effort focuses on solving the complexities of nonlinear coupled reaction-diffusion systems, such as the Rogers-McCulloch (RM) model \cite{rogers1994collocation}.
We denote a bounded connected domain by $\Omega \subset \Real^{d}$, here $d=1 \textrm{ or } 2$, with Lipschitz continuous boundary $\partial \Omega$. We represent the space-time domain and its lateral
boundary are by $Q = \Omega \times (0,T]$ and $\Sigma = \partial \Omega \times (0,T]$, respectively.
The RM model, constituting the focal point of our study, is governed by coupled reaction-diffusion equations given below.
\begin{equation}\label{RM}
\begin{cases}
\frac{\partial u}{\partial t}  = \gamma \Delta u - c_1u(a-u)(1-u) - c_2uv, & (x, t)\in Q\,,\\
\frac{\partial v}{\partial t} = bu - bc_3v , & (x, t)\in Q\,,\\
\frac{\partial u}{\partial \nu}=0, & (x, t)\in \Sigma,\\
u(x,0)=u_0, \text{ and } v(x,0)=v_0, & x\in \Omega,
\end{cases}
\end{equation}
where $\nu$ represents the outward unit normal to $\partial\Omega$. The problem parameters $\gamma, a, b, c_1, c_2$ and $c_3$ appear in \eqref{RM} are positive real numbers \cite{rogers1994collocation}.
The RM model is a phenomenological two-variable model and helps to study cardiac action potential propagation within the myocardium, which is the muscular tissue of the heart. The state variables $u$ and $v$ represent the spatial and temporal evolution of the transmembrane potential and the recovery variable, respectively. These models reduce the complex array of ion currents to two state variables that describe excitation and recovery through the cardiac tissue. 
Such kinds of models hold significance due to their ability to undergo thorough mathematical analysis of excitation and recovery processes, represented graphically in a phase plane where $u$ is plotted against $v$ \cite{keener2009mathematical}. Several modifications of the FitzHugh-Nagumo model have been proposed to mimic the rate sensitivity of action potential duration \cite{Aliev96, vanCapelle_CircRes80, rogers1994collocation, Mitchell_BMB03}. Our investigation specifically centers on the RM model \cite{rogers1994collocation}, notable for its maintenance of a stable resting potential in cardiac myocytes.

For brevity, we denote $f(u)= u(a-u)(1-u)$ and $\widetilde{f}(u)= (a-u)(1-u)$. Through a straightforward computation, we can readily derive the following pair of inequalities:
\begin{equation}\label{inequality1}
\widetilde{f}(u)\geq -\frac{(a-1)^2}{4},\;
f'(u)\geq -\frac{(a+1)^2}{3}.
\end{equation}

The state equations belong to the system of degenerate reaction-diffusion system.
For an isolated heart, their existence and uniqueness results were reported for phenomenological models
to the physiological models in \cite{Bourgault08, Nagaiah_COAP11, 
ColliFranzonea02, Veneroni_NARWA09, KW1}. For the discretization technique of the model \eqref{RM}: a collocation-based finite element method (FEM) is given in \cite{rogers1994collocation}, a nonlinear FEM with a posteriori estimates given in \cite{ratti2019posteriori, NK_APNUM10}. 
The numerical solution of the nonlinear coupled PDE-ODE systems related to cardiac electrophysiology poses a significant challenge due to the complex interplay of temporal and spatial scales inherent in wave excitation phenomena across the membrane \cite{Vigmond08:_solvers, Southern_TBE09, ColliFranzone_book14}. Furthermore, to accurately study wavefront propagation and phenomena like reentry without boundary effects, extensive observation times are required. These timescales need to be large enough to capture the intricate dynamics involved in arrhythmic events, which adds another layer of complexity to the computational modeling process. Previous studies on the monodomain model, including those mentioned above, have utilized time-stepping methods to evolve the equation. This requires sequential computations over prolonged periods to accurately capture long-term behavior in large spatial domains. 

To introduce parallelism in the temporal domain, we employ the Parareal method \cite{lions2001parareal}, originally devised by Lions, Maday, and Turinici. This method is renowned for its efficacy in tackling time-dependent PDEs by leveraging parallel computing techniques. The Parareal approach operates by subdividing the time interval into smaller segments and employing two solvers: a computationally inexpensive yet less precise solver known as the coarse solver and a computationally intensive yet highly accurate solver referred to as the fine solver. Subsequently, a correction step is executed to refine the solution at the coarse time points. The versatility of the Parareal technique is evident in its successful application across various domains, including fluid-structure interaction \cite{farhat2003time}, the Navier-Stokes equation \cite{fischer2005parareal}, parabolic and hyperbolic problems \cite{gander2007analysis}, the Cahn-Hilliard equation \cite{garai2024convergence}, time-periodic problems \cite{gander2013analysis}, molecular dynamics \cite{baffico2002parallel}, and quantum control problems \cite{maday2007monotonic, maday2002parareal, maday2003parallel}.

The Parareal method offers flexibility in selecting suitable coarse and fine solvers, providing researchers the freedom to tailor the method to their specific computational needs. For example, iterative techniques such as the Schwarz waveform relaxation (SWR) method \cite{gander1998space} or its optimized variant (OSWR) \cite{gander2007optimized} can serve as either the fine or coarse solvers or even both, giving rise to what is termed as the Parareal SWR (PSWR) method \cite{gander2019superlinear}. Moreover, recent advancements have explored the coupling of Parareal with Dirichlet-Neumann or Neumann-Neumann waveform relaxation schemes, as proposed in \cite{song2021analysis}.
It's important to note that by employing space-time domain decomposition (DD) methods as solvers, parallelism in space is introduced. This intricate process involves outer Parareal iterations alongside inner iterations corresponding to the space-time DD solvers, thus facilitating a two-level parallelization approach. To expedite computations, a strategy often employed is to prematurely halt the inner iterative solvers after a predefined number of iterations, even before reaching convergence. This anticipation relies on the expectation that overall convergence will be achieved through subsequent Parareal iterations, as discussed in \cite{maday2005parareal}.
Further exploration into the coupling of Parareal with incomplete non-overlapping OSWR iteration has been undertaken, particularly in the context of linear parabolic problems, as discussed and analyzed in \cite{bui2022coupling}.

The current developments underscore the continuous refinement and versatility of the Parareal and OSWR methods, offering promising avenues for accelerating the resolution of time-dependent PDEs across various scientific and engineering domains.
The main objective of our work is to develop and analyze a space-time parallel method for equation \eqref{RM} within the Parareal-incomplete OSWR framework. We begin by presenting the Parareal formulation, followed by an exhaustive investigation of the continuous-level non-overlapping OSWR method tailored for the RM model. Subsequently, we investigate an in-depth analysis of the convergence properties inherent in the Parareal-incomplete OSWR procedure. This endeavor represents a significant extension of the convergence framework established for the Parareal-incomplete OSWR method \cite{bui2022coupling}, particularly within a nonlinear and coupled PDEs context, which has not been investigated in the existing literature. The nonlinear framework we establish for convergence analysis holds promise for application across a spectrum of nonlinear PDEs or systems.


The remainder of this paper is structured as follows: In Section \ref{Section2}, we introduce the notation and explore the existence results for equation \eqref{RM}. Section \ref{Section2-Parareal} presents the Parareal formulation for equation \eqref{RM}. The formulation of the OSWR method for equation \eqref{RM} and its convergence results are discussed in Section \ref{Section3}. In Section \ref{Section4}, we delve into the convergence analysis of the coupled Parareal-incomplete OSWR. To demonstrate the accuracy and robustness of our proposed formulation, numerical results are provided in Section \ref{Section5}.

\section{Formulation of the Parareal Method} \label{Section2}
\subsection{Assumptions}
We first prescribed the assumptions on the nonlinear functions. Similar to \cite{ratti2019posteriori}, we  assume that the nonlinear functions \(\mathfrak {f}(u, v)=c_1u(a-u)(1-u) + c_2uv \) and \(\mathfrak {g }(u, v)= -bu + bc_3v\) satisfies the following  monotonicity condition:
\begin{equation}\label{monotone_fg}
\big(\mathfrak{f}(u_1, w_1) - \mathfrak{f}(u_2, w_2)\big)(u_1 - u_2) + \big(\mathfrak{g}(u_1, w_1) - \mathfrak{g}(u_2, w_2)\big)(w_1 - w_2)
\geq \lambda_{\mathfrak{f}\mathfrak{g}} \left( (u_1 - u_2)^2 + (w_1 - w_2)^2 \right),
\end{equation}
for all $u_1, u_2, w_1, w_2 \in \mathbb{R}$, where \( \lambda_{\mathfrak{f}\mathfrak{g}} > 0 \) is a constant.

We take some assumption on the splitting function $f_s$ which satisfies $f_s(u, u, v)=\mathfrak{f}(u, v)$. We assume monotonicity (or one-sided Lipschitz) in the first variable:
  \begin{equation}\label{osL_fs}
    \big(f_s(u_1, z, w) - f_s(u_2, z, w)\big)(u_1 - u_2) \geq -\lambda_{f_s} \|u_1 - u_2\|^2, \quad \forall u_1, u_2 \in L^2(\Omega),
  \end{equation}
  where $\lambda_{f_s} > 0$ is independent of the frozen variables $z, w \in L^2(\Omega)$. 
 This is a reasonable assumption because:
     $f_s$ is \emph{linear} in $u$,
    and all the nonlinear complexity is in the coefficients, which depend on $z, w$, which are known.

    Next, we assume global Lipschitzness of $f_s$ in the last two variables: 
    \begin{equation}\label{lip_fs}
    \| f_s(u, z_1, w_1) - f_s(u, z_2, w_2) \| \leq L_f \left( \| z_1 - z_2 \| + \| w_1 - w_2 \| \right), \quad \forall z_1, z_2, w_1, w_2 \in L^2(\Omega),
  \end{equation}
  for some constant $L_f > 0$. 

\subsection{Existence Results of the RM Model} 
We denote $H^m(\Omega)$ be the standard Sobolev space with $m=0,1,2$, where $H^0(\Omega)=L^2(\Omega)$. We also denote $\langle \cdot, \cdot \rangle$ be the $L^2$ inner product, and $\parallel \cdot \parallel$ be the $L^2$ norm. 
We denote two non-overlapping decompositions of the domain $\Omega$ by $\Omega_1$ and $\Omega_2$.
We denote $\mathcal{H}_i=H^2(\Omega_i)$ for $i=1, 2$ and $\mathcal{H}=\{ u\in L^2(\Omega): u_{|\Omega_i} \in \mathcal{H}_i\}$.

\subsection{Existence Results of \eqref{RM}}
Below, we list the existence results of the RM Model \eqref{RM} for an initial solution with different regularities.
\begin{theorem}\label{thm1}
Let $u_0, v_0\in L^2(\Omega)$. Then, there exists a unique weak solution to the RM model s.t
\begin{align*}
& u\in L^2(0,T;H^1(\Omega)) \medcap C(0,T;L^2(\Omega)) \medcap L^4(0,T;L^4(\Omega)), \frac{\partial u}{\partial t}\in L^{\frac{4}{3}}(0,T;(H^1(\Omega))^*)\\
\text{ and } & v\in L^{\infty}(0,T; L^2(\Omega)), \frac{\partial v}{\partial t}\in L^{2}(0,T; L^2(\Omega)).
\end{align*}
For $u_0\in H^2(\Omega)$ and $v_0\in L^2(\Omega)$, the weak solution $u, v$ satisfies the following 
\begin{align*}
& u\in L^2(0,T;H^2(\Omega)) \medcap L^{\infty}(0,T;H^1(\Omega))\medcap C(0,T;L^2(\Omega)), \frac{\partial u}{\partial t}\in L^{2}(0,T;L^2(\Omega)) \\
\text{ and } & v\in L^{\infty}(0,T; L^2(\Omega)), \frac{\partial v}{\partial t}\in L^{2}(0,T; L^2(\Omega)).
\end{align*}
\end{theorem}
The existence of a solution with a lower regularity of the initial solution can be found in \cite{Bourgault08, KW1}, while the existence results for initial solutions with higher regularity can be obtained by following a similar procedure, considering higher-order a priori estimates \cite{jackson1990existence}. 

\subsection{ Parareal Method} \label{Section2-Parareal}

Here, we recall the general Parareal method and extend it for the RM model.
To build the Parareal method, we first decompose the time domain $[0, T]$ into $N$ smaller time slices of uniform size, i.e., $[0, T]=\bigcup\limits_{n=1}^{N}[T_{n-1}, T_n]$ with $T_{n}- T_{n-1}=\Delta T=T/N$ is considered. Secondly, each time slice $[T_{n-1}, T_n]$ is divided into $J$ time steps with $\Delta t=\Delta T/J$, then a fine propagator $\mathtt{F}$, and a coarse propagator $\mathtt{G}$ are assigned to compute the solution in fine and coarse grid respectively.
Denoting ${U}_n^{l}$ and ${V}_n^{l}$, the value of the excitation variable and recovery variable at the coarse time points $T_n$'s in the $l$-th Parareal iteration. Let $\mathtt{X}_n^l=[U_n^l, V_n^l]^{\top}$, $\mathtt{F}(T_{n+1}, T_n, \mathtt{X}_n^{l})$ and $\mathtt{G}(T_{n+1}, T_n, \mathtt{X}_n^{l})$ be the fine and coarse operator. Then the Parareal method for the RM model starts with the initial approximation $\mathtt{X}_n^0$ at $T_n$'s, obtained by the coarse operator $\mathtt{G}$ and solves the following prediction-correction scheme for $l=0, 1,...$
\begin{equation}\label{parareal_rm}
    \begin{aligned}
       \mathtt{X}_0^{l+1} & =[ u_0, v_0]^{\top},\\
      \mathtt{X}_{n+1}^{l+1} & =\mathtt{G}(T_{n+1}, T_n, \mathtt{X}_n^{l+1})+\mathtt{F}(T_{n+1}, T_n, \mathtt{X}_n^l)-\mathtt{G}(T_{n+1}, T_n, \mathtt{X}_n^l)\,, \quad n=0,1,\cdots, N-1,
    \end{aligned}       
\end{equation}
where the operator $\mathtt{S}(T_{n+1}, T_{n}, \mathtt{X}_{n}^{l})$ computes the solution at $T_{n+1}$ by taking the initial solution $\mathtt{X}_{n}^{l}$ at $T_{n}$ for $\mathtt{S} = \mathtt{F}$ or $\mathtt{G}$. We use the following semi-implicit scheme to discretize our model 
\begin{equation}\label{discrete_RM}
\begin{aligned}
\frac{u^{n+1} -u^{n}}{\Delta t} & = \gamma \Delta u^{n+1} - c_1u^{n+1}(a-u^n)(1-u^n) - c_2u^{n+1}v^n,\\
\frac{v^{n+1} -v^{n}}{\Delta t} & = bu^{n+1} - bc_3v^{n+1}.
\end{aligned}
\end{equation} 
Below, we discuss the existence of the linear scheme \eqref{discrete_RM} under certain assumptions and given an auxiliary lemma, which we need later to show the convergence of the Parareal-OSWR method. To prove the existence of \eqref{discrete_RM}, it is enough to show the existence of a solution for \eqref{discrete_RM}$_1$, as the equation \eqref{discrete_RM}$_2$ is an algebraic one. We now fix a few notations to prove the existence of the solution of \eqref{discrete_RM}$_1$ via the Galerkin approach. Since the construction of approximate solutions via the Galerkin method and the subsequent passage to the limit follows standard arguments, we restrict ourselves to presenting the necessary a priori estimates. 
\begin{theorem}\label{existence_thm}
    For $\Delta t<\frac{1}{2\lambda_{f_s}}$, there exists a weak solution $(u, v)\in H^1(\Omega)\times L^2(\Omega)$ for the linear scheme \eqref{discrete_RM} with the homogeneous Neumann boundary condition associated with the RM model.
\end{theorem}
\begin{proof}
   ~Let us denote $u^{n+1}:=u, v^{n+1}:=v, u^n:=s_1$ and $v^n:=s_2$ with $s_1, s_2\in L^2(\Omega)$. Observe that the equation \eqref{discrete_RM}$_1$ can be written in the following weak form: 
   \begin{equation}\label{existence_eq1}
       \int_{\Omega} \frac{1}{\Delta t} u \phi + \gamma  \nabla u \cdot \nabla \phi = -\int_{\Omega}  {f}_s(u, s_1, s_2)\phi + \frac{1}{\Delta t} s_1 \phi ,\; \text{for all} \; u, \phi \in H^1(\Omega).
   \end{equation}
   Letting $\phi=u$ in \eqref{existence_eq1} and using \ref{osL_fs} we obtain 
   \begin{equation}\label{apriori}
      \frac{1}{\Delta t} \parallel u\parallel^2 + \gamma \parallel \nabla u\parallel^2 \leq \lambda_{f_s} \parallel u\parallel^2 + \frac{1}{2\Delta t} \left( \parallel u\parallel^2 + \parallel s_1\parallel^2\right).  
   \end{equation}
   From \eqref{apriori} we have 
   \begin{equation}\label{apriori22}
      \left(\frac{1}{2\Delta t} - \lambda_{f_s} \right)\parallel u\parallel^2 + \gamma \parallel \nabla u\parallel^2 \leq  \frac{1}{2\Delta t}  \parallel s_1\parallel^2.  
   \end{equation}
   For $\Delta t<\frac{1}{2\lambda_{f_s}}$ its clear from \eqref{apriori22} the excitation variable $u$ is bounded in $H^1$ norm. So one can extract a weakly convergent subsequence in $H^1$. Consequently strongly $L^2$, which is needed to pass the limit for the function $f_s$, which is linear and continuous with respect to first variable. 
    Now from the algebraic equation \eqref{discrete_RM}$_2$ we have $(1+bc_3\Delta t)v=b\Delta t u +s_2 \in L^2(\Omega)$. Hence, the theorem.
\end{proof}

\begin{lemma}[Lipschitz condition for the coarse operator]\label{Lipschitz_c_op_ct}
The coarse operator in the predictor-corrector relation \eqref{parareal_rm}, given by \eqref{discrete_RM}, satisfies the Lipschitz condition 
\[
\| \mathtt{G}(t_2, t_1, \mathtt{X}) - \mathtt{G}(t_2, t_1, \mathtt{Y}) \| \leq C_{\texttt{Lip}} \| \mathtt{X} - \mathtt{Y} \|, \quad \forall ~ \mathtt{X}, \mathtt{Y} \in \left(L^2(\Omega)\right)^2,
\]
where $\mathtt{X} = [u, v]^{\top}$ and $\mathtt{Y} = [w, z]^{\top}$ for functions $u, v, w, z \in L^2(\Omega)$, and where the coarse time step $\Delta T := t_2 - t_1$ satisfies 
\begin{equation}\label{coarse_time_step}
\Delta T < \min\left\{ \frac{1}{2\lambda_{f_s}}, \frac{1}{b} \right\}.
\end{equation}
\end{lemma}
\begin{proof}~
Since the coarse operator in the Parareal method is the semi-implicit scheme \eqref{discrete_RM}, the functions $u(t), v(t), w(t), z(t)$ satisfy the equations
\begin{equation}\label{li1}
\begin{aligned}
\frac{u(t) -u(t_1)}{t-t_1} & = \gamma \Delta u(t) - {f}_s(u(t), u(t_1), v(t_1)),\; t\in[t_1, t_2],\\
\frac{v(t)-v(t_1)}{t-t_1} & = bu(t) - bc_3v(t), \; t\in[t_1, t_2],
\end{aligned}
\end{equation}
with $u(t_1)=\texttt{u}_0\in L^2(\Omega)$ and $v(t_1)=\texttt{v}_0\in L^2(\Omega)$, and
\begin{equation}\label{li2}
\begin{aligned}
\frac{w(t) -w(t_1)}{t-t_1} & = \gamma \Delta w(t) - {f}_s(w(t), w(t_1), z(t_1)),\; t\in[t_1, t_2],\\
\frac{z(t)-z(t_1)}{t-t_1} & = bw(t) - bc_3z(t), \; t\in[t_1, t_2],
\end{aligned}
\end{equation}
with $w(t_1)=\texttt{w}_0\in L^2(\Omega)$ and $z(t_1)=\texttt{z}_0\in L^2(\Omega)$. With this above setting we have $\mathtt{G}(t, t_1, \mathtt{X})=[u(t), v(t)]^{\top}$ and $\mathtt{G}(t, t_1, \mathtt{Y})=[w(t), z(t)]^{\top}$. Subtracting \eqref{li2}$_1$ from \eqref{li1}$_1$ and letting $\mathtt{u}(t):=u(t)-w(t)$ and $\mathtt{v}(t):=v(t)-z(t)$ we obtain 
\begin{equation}\label{li3}
\begin{aligned}
\frac{\mathtt{u}(t) -\mathtt{u}(t_1)}{t-t_1} & = \gamma \Delta \mathtt{u}(t) - \left( {f}_s(u(t), u(t_1), v(t_1)) - {f}_s(w(t), w(t_1), z(t_1)) \right),
\end{aligned}
\end{equation}
Multiply the equation \eqref{li3} with $\mathtt{u}(t)$ and integrate over $\Omega$, we have 
\begin{multline}\label{li4}
\parallel\mathtt{u}(t) \parallel^2 + (t-t_1) \gamma \parallel\nabla \mathtt{u}(t)\parallel^2= \langle \mathtt{u}(t_1), \mathtt{u}(t)\rangle - (t-t_1) \int_{\Omega} \left( {f}_s(u(t), u(t_1), v(t_1)) - {f}_s(w(t), w(t_1), z(t_1)) \right)  \mathtt{u}(t).
\end{multline}
Rewriting the difference term $\left( {f}_s(u(t), u(t_1), v(t_1)) - {f}_s(w(t), w(t_1), z(t_1)) \right)$ as follows
\begin{equation}\label{diff_fs}
\begin{aligned}
     {f}_s(u(t), u(t_1), v(t_1)) - {f}_s(w(t), w(t_1), z(t_1))  =  {f}_s(u(t), u(t_1), v(t_1)) - {f}_s(w(t), u(t_1), v(t_1)) \\
    + {f}_s(w(t), u(t_1), v(t_1)) - {f}_s(w(t), w(t_1), z(t_1)). 
\end{aligned}
\end{equation}
Using \eqref{diff_fs} in the last integral in \eqref{li4}, and using \eqref{osL_fs} and \eqref{lip_fs} we obtain 
\begin{multline}\label{li45}
\parallel\mathtt{u}(t) \parallel^2 + (t-t_1) \gamma \parallel\nabla \mathtt{u}(t)\parallel^2 \leq  \langle \mathtt{u}(t_1), \mathtt{u}(t)\rangle + (t-t_1)\lambda_{f_s} \parallel \mathtt{u}(t)\parallel^2 + (t-t_1)L_f\left( \parallel \mathtt{u}(t_1)\parallel + \parallel \mathtt{v}(t_1)\parallel\right)\parallel \mathtt{u}(t)\parallel.
\end{multline}
Using Cauchy-Schwarz and Young's inequality in \eqref{li45}, we have 
\begin{equation}\label{li46}
\alpha_t\parallel\mathtt{u}(t) \parallel^2 \leq  \left( 2+ 4(t-t_1)^2L_f^2 \right)\parallel \mathtt{u}(t_1)\parallel^2 + 4(t-t_1)^2L_f^2\parallel \mathtt{v}(t_1)\parallel^2,
\end{equation}
where $\alpha_t=\left(1-2(t-t_1)\lambda_{f_s} \right)$.
Here one has to choose $t\in[t_1, t_2]$ such that $\alpha_t>0$. Now subtracting \eqref{li2}$_2$ from \eqref{li1}$_2$ and taking $L^2$ inner product with 
$\mathtt{v}(t)$ we have 
\begin{equation}\label{li11}
\left\langle \frac{\mathtt{v}(t)-\mathtt{v}(t_1)}{t-t_1}, \mathtt{v}(t) \right\rangle= \langle b\mathtt{u}(t) - bc_3\mathtt{v}(t), \mathtt{v}(t)\rangle.
\end{equation}
Using the Cauchy-Schwarz and Young's inequality in \eqref{li11}, we obtain
\begin{equation}\label{li12}
\beta_t \parallel\mathtt{v}(t) \parallel^2  \leq b(t-t_1)\parallel \mathtt{u}(t)\parallel^2 +  \parallel \mathtt{v}(t_1)\parallel^2,
\end{equation}
where $\beta_t:=1-b(t-t_1)$. Here also one has to choose $t\in[t_1, t_2]$ such that $\beta_t>0$. From \eqref{li46} and \eqref{li12} we have
\begin{equation}\label{lip13}
\underbrace{\begin{bmatrix}
\alpha_t & 0\\
-b(t-t_1) & \beta_{t}
\end{bmatrix}}_{\mathcal{M}}
\begin{bmatrix}
\parallel \mathtt{u}(t) \parallel^2\\
\parallel \mathtt{v}(t) \parallel^2
\end{bmatrix} \leq 
\begin{bmatrix}
 \left( 2+ 4(t-t_1)^2L_f^2 \right) &  4(t-t_1)^2L_f^2 \\
0& 1
\end{bmatrix}
\begin{bmatrix}
\parallel \mathtt{u}(t_1) \parallel^2\\
\parallel \mathtt{v}(t_1) \parallel^2
\end{bmatrix} \leq \left( 2+ 4(t-t_1)^2L_f^2\right)
\begin{bmatrix}
\parallel \mathtt{u}(t_1) \parallel^2\\
\parallel \mathtt{v}(t_1) \parallel^2
\end{bmatrix}.
\end{equation}
Taking $t=t_2$ in \eqref{lip13} in such a way that $\alpha_{t_2}, \beta_{t_2} >0$, i.e., $\Delta T<\min\left\{ \frac{1}{2\lambda_{f_s}}, \frac{1}{b} \right\}$ we obtain
\begin{equation*}
\begin{bmatrix}
\parallel \mathtt{u}(t_2) \parallel^2\\
\parallel \mathtt{v}(t_2) \parallel^2
\end{bmatrix} \leq \left( 2+ 4(\Delta T L_f)^2\right) \rho_{\mathcal{M}^{-1}}
\begin{bmatrix}
\parallel \mathtt{u}(t_1) \parallel^2\\
\parallel \mathtt{v}(t_1) \parallel^2
\end{bmatrix},
\end{equation*}
where $\rho_{\mathcal{M}^{-1}}=\max \{ \frac{1}{\alpha_{t_2}}, \frac{1}{\beta_{t_2}} \}$. Hence, the Lemma is established for $C_{\texttt{Lip}}=\left( 2 + 4(\Delta T L_f)^2\right) \rho_{\mathcal{M}^{-1}}$.
\end{proof}


\section{Non-overlapping OSWR and Its Convergence}\label{Section3}
In this section, we introduce the OSWR method for the RM model \eqref{RM}. Let the spatial domain $\Omega = \mathbb{R}$ (or $\mathbb{R}^2$) is divided into two non-overlapping subdomains $\Omega_{1} = (-\infty, 0), \Omega_{2} = (0, \infty)$ in 1D and  $\Omega_{1} = (-\infty, 0)\times\mathbb{R}, \Omega_{2} = (0, \infty)\times\mathbb{R}$ in 2D. Set $\Omega_{j,T}:=\Omega_{j}\times(0, T], \Gamma_{T}:=\Gamma\times(0, T]$, where $\Gamma=\{ 0\}$ in 1D (or $\Gamma=\{ 0\}\times\mathbb{R}$ in 2D). The OSWR method for the RM model \eqref{RM} starts with initial guesses $u_i^0(0, t)$ for $t\in(0, T]$ along $\Gamma_{T}$ and then simultaneously solve for $k=1, 2, \cdots$
\begin{equation}\label{oswr1}
\begin{aligned}
&\left\{
\begin{aligned}
\frac{\partial u_i^k}{\partial t}  = \gamma \Delta u_i^k - c_1u_i^k(a-u_i^k)(1-u_i^k) - c_2u_i^kv_i^k, \,\ \text{in} \,\ \Omega_{i,T},\\
\frac{\partial v_i^k}{\partial t}  = bu_i^k - bc_3v_i^k ,\,\ \text{in} \,\ \Omega_{i,T},  \\
u_i^k(x, 0)
 = u_0,\,\ \text{in}\,\ \Omega_{i},\\
\mathtt{B}_i u_i^k
 = 
\mathtt{B}_i u_{3-i}^{k-1},
\,\ \mbox{on}\,\ \Gamma_{T},\\
\end{aligned}\right.
\end{aligned}
\end{equation}
where the interface operator $\mathtt{B}_i$ is given by 
\begin{equation}
\mathtt{B}_i = \gamma\frac{\partial}{\partial n_i} + p,
\end{equation}
with $n_i$'s as outward normal to space-time interface $\Gamma_T$ for $i=1, 2$, and $p\in\mathbb{R}^{+}$ is the Robin parameter which will accelerate the convergence of the OSWR method. We rewrite the transmission condition \eqref{oswr1}$_4$ in the following way
\begin{equation}
\mathtt{B}_i u_i^{k}=g_i^{k}\; \text{on}\; \Gamma_{T},  \text{where}\; g_i^{k}= \mathtt{B}_i u_{3-i}^{k-1}\; \text{on}\; \Gamma_{T}, 
\end{equation}
which is needed to write the OSWR method compactly. So it is easy to observe from \eqref{oswr1} that OSWR takes the initial solution $u_0, v_0$ and initial robin traces $\boldsymbol{g}^{0}=(g_1^{0}, g_2^{0})^{\top}$ as inputs and then it solves the subdomain problem simultaneously. Before going to the next iteration, one traces out the Robin data and updates by $\boldsymbol{g}^{1}$.  The procedure in compact form can be written as 
\begin{equation}
\left( [u^{K}, v^{K}]^{\top}, \boldsymbol{g}^{K} \right)= \text{OSWR}_{K} \left( u_0,v_0, \boldsymbol{g}^{0} \right), 
\end{equation}
where $K\in\mathbb{N}$ be the fixed iteration number, $u^{K}, v^{K}$ are the composite solution such that $u^{K}_{|\Omega_i}=u_i^{K}, v^{K}_{|\Omega_i}=v_i^{K}$.  

\subsection{Existence of Solution of Subproblems}
It suffices to investigate the existence of a solution for the first subdomain problem, after which the same approach can be extended to other subdomains. The first subdomain model problem is defined as follows.
\begin{equation}\label{RM_subproblem}
\begin{cases}
\frac{\partial u}{\partial t}  = \gamma \Delta u - c_1u(a-u)(1-u) - c_2uv, & (x,t)\in\Omega_1\times(0,T],\\
\frac{\partial v}{\partial t} = bu - bc_3v , & (x,t)\in\Omega_1\times(0,T],\\
\frac{\partial u}{\partial n}=0, & (x,t)\in\left(\partial\Omega_1 - \Gamma\right)\times(0,T],\\
\gamma\frac{\partial u}{\partial n} + pu
 = g,  & (x,t)\in\Gamma_T\\
u(x,0)=u_0,  \quad v(x,0)=v_0, & x\in \Omega_1.
\end{cases}
\end{equation}

To show the existence of a weak solution of the subproblem \eqref{RM_subproblem}, we adopt the framework given in \cite{caetano2010schwarz}. First, study the existence and uniqueness of the linear non-homogeneous problem corresponding to the nonlinear subproblem \eqref{RM_subproblem}. We then obtained the solution to the nonlinear problem using the Picard fixed point technique. In \eqref{linear_RM_subproblem}, we consider the linear problem associated with \eqref{RM_subproblem} 
\begin{equation}\label{linear_RM_subproblem}
\begin{cases}
\frac{\partial w}{\partial t}  = \gamma \Delta w + \mathtt{f}, & (x,t)\in\Omega_1\times(0,T],\\
\frac{\partial z}{\partial t} = bw - bc_3z , & (x,t)\in\Omega_1\times(0,T],\\
\frac{\partial w}{\partial n}=0, & (x,t)\in\left(\partial\Omega_1 - \Gamma\right)\times(0,T],\\
\gamma\frac{\partial w}{\partial n} + pw
 = g,  & (x,t)\in\Gamma_T\\
w(x,0)=u_0,  \quad z(x,0)=v_0, &  x\in \Omega_1 \,.
\end{cases}
\end{equation}
We 
define the solution space $\mathcal{S}_{w}(T):=W^{1,\infty}(0,T; L^2(\Omega))\medcap L^{\infty}(0, T; H^2(\Omega))\medcap H^1(0, T; H^1(\Omega))$ for the variable $w$, and 
$\mathcal{S}_{z}(T):=W^{1,\infty}(0,T; L^2(\Omega))\medcap H^1(0, T; L^2(\Omega)) \medcap L^{\infty}(0, T; H^1(\Omega))$ for the variable $z$.
\begin{theorem}[a priori estimates]\label{thm_existence_linear}
Let $u_0\in H^2(\Omega_1), v_0\in H^1(\Omega_1), \mathtt{f}\in H^1(0, T; L^2(\Omega_1))$ and $g\in H^1(0, T;L^2(\Gamma))\medcap L^{\infty}(0, T;H^{\frac{1}{2}}(\Gamma))$, then the problem \eqref{linear_RM_subproblem} has unique solution in $\mathcal{S}_{w}(T) \times \mathcal{S}_{z}(T)$ and satisfies the following estimate 
\begin{equation}\label{estimate_galerkin}
\begin{aligned}
\parallel (w, z)\parallel_{\mathcal{S}_{w}(T) \times \mathcal{S}_{z}(T)}^2 \leq \mathfrak{C}e^{\theta T} \left( \parallel u_0\parallel^2_{H^2(\Omega_1)} + \parallel v_0\parallel^2_{L^2(\Omega_1)}  + \parallel g\parallel^2_{H^1(0, T;L^2(\Gamma))\medcap L^{\infty}(0, T;H^{\frac{1}{2}}(\Gamma))} \right. \\
\left.+ \parallel \mathtt{f}\parallel_{H^{1}(0,T; L^2(\Omega_1))}^2 \right),
\end{aligned}
\end{equation}
where $\mathfrak{C}$ is depends on $b, c_3, \gamma, p$ and $\theta_1$, and $\theta=\max \left \{1+\frac{b}{4c_3}, 2bc_3 \right \}$.
\end{theorem}
\begin{proof}~
\textit{First, we provide a priori estimate for $w\in L^{\infty}(0,T; L^2(\Omega_1))\medcap L^{2}(0,T; H^1(\Omega_1))$ and $z\in L^{\infty}(0,T; L^2(\Omega_1))$.} By computing the $L^2$ inner product of \eqref{linear_RM_subproblem}$_{1}$ with $w$ and \eqref{linear_RM_subproblem}$_{2}$ with $z$, and performing integration by parts across the subdomain $\Omega_1$, followed by the application of Cauchy-Schwarz and Young's inequality, and finally integrating over the temporal domain, we arrive at:
\begin{equation}\label{apriori1}
\begin{aligned}
\parallel w\parallel^2 + 2\gamma\int_{0}^t \parallel\nabla w \parallel^2 + p \int_0^t \parallel w \parallel^2_{\Gamma}  & \leq \parallel u_0\parallel^2 + \frac{1}{p}\int_0^t \parallel g \parallel^2_{\Gamma}
+ \int_0^t \parallel \mathtt{f} \parallel^2 + \int_0^t \parallel w \parallel^2,
\end{aligned}
\end{equation}
 
\begin{equation}\label{apriori2}
\begin{aligned}
\text{ and } \parallel z\parallel^2 & \leq \parallel v_0\parallel^2 + 
 2bc_3 \int_0^t \parallel z \parallel^2 + \frac{b}{4c_3}\int_0^t \parallel w \parallel^2.
\end{aligned}
\end{equation}
By adding \eqref{apriori1} and \eqref{apriori2} and using Gr\"{o}nwall's inequality we obtain 
\begin{equation}\label{apriori3}
\begin{aligned}
\parallel w\parallel_{L^{\infty}(0,T; L^2(\Omega_1))}^2+\parallel z\parallel_{L^{\infty}(0,T; L^2(\Omega_1))}^2+ 2\gamma\parallel \nabla w\parallel_{L^{2}(0,T; L^2(\Omega_1))}^2 + p\parallel w\parallel_{L^{2}(0,T; L^2(\Gamma))}^2 & \leq \\
e^{\theta T} \left( \parallel u_0\parallel^2 + \parallel v_0\parallel^2 + \frac{1}{p}\parallel g\parallel^2_{L^{2}(0,T; L^2(\Gamma))} + \parallel \mathtt{f}\parallel^2_{L^{2}(0,T; L^2(\Omega_1))} \right).
\end{aligned}
\end{equation}
\textit{Second, we provide a priori estimate for $\partial_t w\in L^{\infty}(0,T; L^2(\Omega_1))\medcap L^{2}(0,T; H^1(\Omega_1))$ and $\partial_t z\in L^{\infty}(0,T; L^2(\Omega_1))$.} To get that we differentiate the system \eqref{linear_RM_subproblem} with respect to $t$, and set $\bar{w}:=\partial_t w$ and $\bar{z}:=\partial_t z$. Then we have the following system
\begin{equation}\label{linear_RM_subproblem_diff_wrt_t}
\begin{cases}
\frac{\partial \bar{w}}{\partial t}  = \gamma \Delta \bar{w} + \partial_t\mathtt{f}, & (x,t)\in\Omega_1\times(0,T],\\
\frac{\partial \bar{z}}{\partial t} = b\bar{w} - bc_3\bar{z} , & (x,t)\in\Omega_1\times(0,T],\\
\frac{\partial \bar{w}}{\partial n}=0, & (x,t)\in\left(\partial\Omega - \Gamma\right)\times(0,T],\\
\gamma\frac{\partial \bar{w}}{\partial n} + p\bar{w}
 = \partial_t g,  & (x,t)\in\Gamma_T.
\end{cases}
\end{equation}
Multiplying \eqref{linear_RM_subproblem_diff_wrt_t}$_1$ by $\bar{w}$ and \eqref{linear_RM_subproblem_diff_wrt_t}$_2$ by $\bar{z}$, integrating by parts and invoking Cauchy-Schwarz and Young's inequality, and finally integrating with respect to time, we have  
\begin{equation}\label{apriori4}
\begin{aligned}
\parallel \partial_t w\parallel^2 + 2\gamma\int_{0}^t \parallel\nabla  \partial_t w \parallel^2 + p \int_0^t \parallel \partial_t w \parallel^2_{\Gamma}  & \leq \parallel \partial_t w_0\parallel^2 + \frac{1}{p}\int_0^t \parallel \partial_t g \parallel^2_{\Gamma}
+ \int_0^t \parallel \partial_t\mathtt{f} \parallel^2 + \int_0^t \parallel \partial_t w \parallel^2,
\end{aligned}
\end{equation}
\begin{equation}\label{apriori5}
\begin{aligned}
\text{ and } \parallel \partial_t z\parallel^2 & \leq \parallel \partial_t z_0\parallel^2 + 
 2bc_3 \int_0^t \parallel \partial_t z \parallel^2 + \frac{b}{4c_3}\int_0^t \parallel \partial_t w \parallel^2.
\end{aligned}
\end{equation}
By adding \eqref{apriori4} and \eqref{apriori5} and using Gr\"{o}nwall's inequality we obtain 
\begin{equation}\label{apriori6}
\begin{aligned}
\parallel \partial_t w\parallel_{L^{\infty}(0,T; L^2(\Omega_1))}^2+\parallel \partial_t z\parallel_{L^{\infty}(0,T; L^2(\Omega_1))}^2+ 2\gamma\parallel \nabla \partial_t w\parallel^2_{L^{2}(0,T; L^2(\Omega_1))} + p\parallel \partial_t w\parallel^2_{L^{2}(0,T; L^2(\Gamma))} &  \\
\leq e^{\theta T} \left( \parallel \partial_t w_0\parallel^2 + \parallel \partial_t z_0\parallel^2 + \frac{1}{p}\parallel \partial_t g\parallel^2_{L^{2}(0,T; L^2(\Gamma))} + \parallel \partial_t \mathtt{f}\parallel^2_{L^{2}(0,T; L^2(\Omega_1))} \right).
\end{aligned}
\end{equation}
Now to estimate $\parallel \partial_t z_0\parallel^2$ in \eqref{apriori6}, we multiply $\partial_t z$ to \eqref{linear_RM_subproblem}$_{2}$, integrate over the subdomain and evaluate at $t=0$, we obtain 
\begin{equation}\label{apriori7}
\parallel \partial_t z_0\parallel^2 \leq b^2 \parallel u_0\parallel^2.
\end{equation}
Similarly for the estimate $\parallel \partial_t w_0\parallel^2$ in \eqref{apriori6}, we obtain 
\begin{equation}\label{apriori8}
\parallel \partial_t w_0\parallel^2 \leq 2 \left( \gamma^2 \parallel \Delta u_0\parallel^2 + \parallel\mathtt{f}(\cdot,0)\parallel^2 \right).
\end{equation}
Using \eqref{apriori7} and \eqref{apriori8} in \eqref{apriori6} we get
\begin{equation}\label{apriori9}
\begin{aligned}
\parallel \partial_t w\parallel_{L^{\infty}(0,T; L^2(\Omega_1))}^2+\parallel \partial_t z\parallel_{L^{\infty}(0,T; L^2(\Omega_1))}^2+ 2\gamma\parallel \nabla \partial_t w\parallel^2_{L^{2}(0,T; L^2(\Omega_1))} + p\parallel \partial_t w\parallel^2_{L^{2}(0,T; L^2(\Gamma))} &  \\
\leq e^{\theta T} \left( b^2 \parallel u_0\parallel^2 + 2\gamma^2 \parallel  u_0\parallel_{H^2(\Omega_1)}^2 + \frac{1}{p}\parallel g\parallel^2_{H^{1}(0,T; L^2(\Gamma))} + 3\parallel \mathtt{f}\parallel^2_{H^{1}(0,T; L^2(\Omega_1))} \right).
\end{aligned}
\end{equation}
\textit{Third a priori estimate for $\nabla w\in L^{\infty}(0,T; L^2(\Omega_1))$.}
By multiplying \eqref{linear_RM_subproblem}$_{1}$ with $\partial_tw$, performing an integration by parts, employing Cauchy-Schwarz and Young's inequality, and subsequently integrating over the time domain, we arrive at
 \begin{equation}\label{apriori10}
\begin{aligned}
\parallel \partial_t w\parallel^2_{L^{2}(0,T; L^2(\Omega_1))} + \gamma\parallel \nabla w\parallel_{L^{\infty}(0,T; L^2(\Omega_1))}^2 + p\parallel w\parallel_{L^{\infty}(0,T; L^2(\Gamma))}^2 &  \leq \\
 \left(  p\parallel u_0\parallel^2_{\Gamma}+ \parallel \nabla  u_0\parallel^2 + \frac{1}{p}\parallel g\parallel^2_{L^{2}(0,T; L^2(\Gamma))} + \parallel \mathtt{f}\parallel^2_{L^{2}(0,T; L^2(\Omega_1))} + p\parallel \partial_t w\parallel^2_{L^{2}(0,T; L^2(\Gamma))}\right).
\end{aligned}
\end{equation}
\textit{Fourth a priori estimate for $\partial_tz\in L^{2}(0,T; L^2(\Omega_1))$.} We multiply $\partial_t z$ to \eqref{linear_RM_subproblem}$_{2}$, integrate over space and time, and using Cauchy-Schwarz and Young's inequality, we obtain 
\begin{equation}\label{apriori11}
\parallel \partial_t z\parallel_{L^{2}(0,T; L^2(\Omega_1))}^2 +bc_3 \parallel z\parallel_{L^{\infty}(0,T; L^2(\Omega_1))}^2 \leq bc_3\parallel v_0\parallel^2 + b^2 \parallel w\parallel_{L^{2}(0,T; L^2(\Omega_1))}^2.
\end{equation}
\textit{Fifth a priori estimate for $z\in L^{\infty}(0,T; H^1(\Omega_1))$.} To achieve this, we begin by differentiating \eqref{linear_RM_subproblem}$_{2}$ with respect to \( x \). Then, we multiply by \( \nabla v \), integrate over both space and time, and apply Cauchy-Schwarz and Young's inequalities to obtain
\begin{equation}
    \parallel \nabla v\parallel^2 + bc_3 \parallel v \parallel^2_{L^2(0, T; H^1(\Omega_1))} \leq \parallel \nabla v_0\parallel^2 + \frac{b}{c_3} \parallel u \parallel^2_{L^2(0, T; H^1(\Omega_1))}.
\end{equation}
\textit{Sixth a priori estimate for $w\in L^{\infty}(0,T; H^2(\Omega_1))$.} To obtain estimate in $H^2(\Omega_1)$ we pose the equation \eqref{linear_RM_subproblem}$_{1}$ as follow 
\begin{equation}\label{elliptic_problem}
\begin{cases}
 \Delta w =\frac{1}{\gamma}\left( \partial_t w- \mathtt{f}\right) \in L^{\infty}(0,T; L^2(\Omega_1)),\\
\frac{\partial w}{\partial n} 
 = \frac{1}{\gamma}\left( g-pw \right)  \in L^{\infty}(0,T; H^{\frac{1}{2}}(\Gamma)),
\end{cases}
\end{equation}
and then using the classical elliptic regularity result \cite{lions2012non} we obtain the following estimate 
\begin{equation}\label{apriori12}
\begin{split}
\parallel w\parallel_{L^{\infty}(0,T; H^2(\Omega_1))}   \leq 
 \theta_1\left( \parallel \mathtt{f} \parallel_{L^{\infty}(0,T; L^2(\Omega_1))} + \parallel g\parallel_{L^{\infty}(0,T; H^{\frac{1}{2}}(\Gamma))} + \parallel \partial_t w \parallel_{L^{\infty}(0,T; L^2(\Omega_1))} \right.\\ 
 \left. +\parallel w \parallel_{L^{\infty}(0,T; H^1(\Omega_1))}\right),
\end{split}
\end{equation}
where $\theta_1$ is some positive constant. So by the Galerkin procedure, we have the existence and uniqueness of the weak solution of the system \eqref{linear_RM_subproblem}, and it satisfies the estimate \eqref{estimate_galerkin}.
\end{proof}

Observe that for our problem $\mathtt{f}$ is given by $\mathtt{f} = \mathtt{f}_1 +\mathtt{f}_2$, where $\mathtt{f}_1(u)=-c_1u(a-u)(1-u)$ and $\mathtt{f}_2(u, v)=-c_2uv$. Next we show that for $(u, v)\in \mathcal{S}_w(T) \times \mathcal{S}_z(T)$, $\mathtt{f} \in H^1(0,T; L^2(\Omega_1))$.
\begin{lemma}[See \cite{caetano2010schwarz}]\label{lemma_ex}
Let $\Omega \subset \mathbb{R}^2$ be the regular domain and $\mathcal{F}\in \mathcal{C}^1(\mathbb{R})$. There exists a continuous positive increasing function $\psi$ such that for any $u, v\in H^2(\Omega)$,
\[
\parallel \mathcal{F}(u) -\mathcal{F}(v)\parallel \leq \psi(\max(\parallel u\parallel_{\infty}, \parallel v\parallel_{\infty})) \parallel u -v\parallel,
\]
where the function $\psi$ is given by $\psi(y) = \sup\limits_{\vert\xi\vert\in(0, y)} \vert \mathcal{F}'(\xi)\vert$.  
\end{lemma}

\begin{lemma}\label{fh_h1}
For $(u, v)\in \mathcal{S}_w(T) \times \mathcal{S}_z(T)$, $\mathtt{f}\in H^1(0,T; L^2(\Omega_1))$ and satisfies the following 
\begin{equation}\label{est_fh}
\begin{aligned}
\parallel\mathtt{f}(u)\parallel_{ H^1(0,T; L^2(\Omega_1))}^2 \leq T \psi \left( \parallel u\parallel_{L^{\infty}(0, T; L^{\infty}(\Omega_1))} \right)^2 \parallel u\parallel_{W^{1, \infty}(0, T; L^{2}(\Omega_1))}^2 \\ + 2T\parallel u\parallel_{L^{\infty}(0,T; L^{\infty}(\Omega_1))}^2  \parallel v\parallel_{W^{1, \infty}(0,T; L^{2}(\Omega_1))}^2 + C_{\texttt{emd}}T\parallel v\parallel_{L^{\infty}(0,T; H^{1}(\Omega_1))}^2  \parallel  u\parallel_{W^{1,\infty}(0,T; H^1(\Omega_1))}^2.
\end{aligned}
\end{equation}
\end{lemma}
\begin{proof}~
Using Lemma \ref{lemma_ex} and following Lemma 3.3 in \cite{caetano2010schwarz}, we have 
\begin{equation}\label{fh1}
\parallel\mathtt{f}_1(u)\parallel_{ H^1(0,T; L^2(\Omega_1))} \leq \sqrt{T} \psi \left( \parallel u\parallel_{L^{\infty}(0, T; L^{\infty}(\Omega_1))} \right) \parallel u\parallel_{W^{1, \infty}(0, T; L^{2}(\Omega_1))}. 
\end{equation}
 Now we show $\parallel\mathtt{f}_2\parallel_{ H^1(0,T; L^2(\Omega_1))}$ is finite. As $u\in L^{\infty}(0,T; H^2(\Omega_1))$ then by embedding result, we have $u\in L^{\infty}(0,T; L^{\infty}(\Omega_1))$, therefore we obtain 
\begin{equation}\label{fh2}
\begin{aligned}
\parallel\mathtt{f}_2\parallel_{ L^2(0,T; L^2(\Omega_1))} &\leq \sqrt{T} \parallel u\parallel_{L^{\infty}(0,T; L^{\infty}(\Omega_1))}  \parallel v\parallel_{L^{\infty}(0,T; L^{2}(\Omega_1))} \\
&\leq \sqrt{T} \parallel u\parallel_{L^{\infty}(0,T; L^{\infty}(\Omega_1))}  \parallel v\parallel_{W^{1,\infty}(0,T; L^{2}(\Omega_1))}. 
\end{aligned}
\end{equation}
Next, we consider $\parallel\mathtt{f}_2'\parallel_{ L^2(0,T; L^2(\Omega_1))}$, and for that we have to show that $\parallel u \partial_t v\parallel_{ L^2(0,T; L^2(\Omega_1))}$ and $\parallel v\partial_t u\parallel_{ L^2(0,T; L^2(\Omega_1))}$ are finite. It's clear that  
\begin{equation}\label{fh3}
\begin{aligned}
\parallel u \partial_t v\parallel_{ L^2(0,T; L^2(\Omega_1))}^2  \leq T\parallel u \parallel_{ L^{\infty}(0,T; L^{\infty}(\Omega_1))}^2
\parallel \partial_t v\parallel_{ L^{\infty}(0,T; L^{2}(\Omega_1))}^2 \\ \leq T\parallel u \parallel_{ L^{\infty}(0,T; L^{\infty}(\Omega_1))}^2
\parallel v\parallel_{ W^{1,\infty}(0,T; L^2(\Omega_1))}^2. 
\end{aligned}
\end{equation}
As $v\in L^{\infty}(0,T; H^1(\Omega_1))$ and $\partial_t u\in L^{2}(0,T; H^{1}(\Omega_1))$, for the term $\parallel v \partial_t u \parallel^2_{L^2(\Omega_1)}$ we obtain 
\begin{equation*}
\int_{\Omega_1} v^2 (\partial_t u)^2 \leq \parallel v\parallel_{L^{4}(\Omega_1)}  \parallel \partial_t u\parallel_{L^4(\Omega_1)} \leq C_{\texttt{emd}} \parallel v\parallel_{H^{1}(\Omega_1)}  \parallel \partial_t u\parallel_{H^1(\Omega_1)}.
\end{equation*}
Therefore we have 
\begin{equation}\label{fh4}
\begin{aligned}
\parallel v\partial_t u\parallel_{ L^2(0,T; L^2(\Omega_1))}^2\leq C_{\texttt{emd}}T\parallel v\parallel_{L^{\infty}(0,T; H^{1}(\Omega_1))}^2  \parallel \partial_t u\parallel_{L^{\infty}(0,T; H^1(\Omega_1))}^2.
\end{aligned}
\end{equation} 
Finally, using \eqref{fh1}, \eqref{fh2}, \eqref{fh3} and \eqref{fh4} we obtain the estimate in \eqref{est_fh}.
\end{proof}

Next, we define a mapping, which is in the following form:
\begin{definition}
Let $u_0\in H^2(\Omega_1), v_0\in H^1(\Omega_1)$ and $g\in H^1(0, T;L^2(\Gamma))\medcap L^{\infty}(0, T;H^{\frac{1}{2}}(\Gamma))$. For $(u, v)\in \mathcal{S}_w(T) \times \mathcal{S}_z(T)$ the linear problem \eqref{fixed_point_map}
\begin{equation}\label{fixed_point_map}
\begin{cases}
\frac{\partial w}{\partial t}  = \gamma \Delta w + \mathtt{f}_1(u) + \mathtt{f}_2(u, v), & (x,t)\in\Omega_1\times(0,T],\\
\frac{\partial z}{\partial t} = bw - bc_3z , & (x,t)\in\Omega_1\times(0,T],\\
\frac{\partial w}{\partial n}=0, & (x,t)\in\left(\partial\Omega_1 - \Gamma\right)\times(0,T],\\
\gamma\frac{\partial w}{\partial n} + pw
 = g,  & (x,t)\in\Gamma_T\\
w(x,0)=u_0, \quad 
z(x,0)=v_0, & x\in \Omega_1.
\end{cases}
\end{equation}
has unique solution $(w, z)\in \mathcal{S}_w(T) \times \mathcal{S}_z(T)$ using Theorem \eqref{thm_existence_linear} and Lemma \eqref{fh_h1}, and naturally defines a map $(w, z) = \mathcal{A}(u, v)$ in $\mathcal{S}_w(T) \times \mathcal{S}_z(T)$.
\end{definition}

Next, we set $Q$ such that 
\begin{equation}\label{bound_Q}
Q^2 \geq 4 \mathfrak{C}\left( \parallel u_0\parallel^2_{H^2(\Omega_1)} + \parallel v_0\parallel^2_{H^1(\Omega_1)}  + \parallel g\parallel^2_{H^1(0, T;L^2(\Gamma))\medcap L^{\infty}(0, T;H^{\frac{1}{2}}(\Gamma))} \right),
\end{equation}
and define the time as 
\begin{equation}\label{def_time}
\widehat{T}(Q)=\sup \left\{ \widetilde{T}\leq T: \max\left( \frac{e^{\theta \widetilde{T}}}{2}, 2\widetilde{T}\mathfrak{C}e^{\theta \widetilde{T}} \left(  \psi(Q)^2 +2Q^2 + C_{\texttt{emd}}Q^2 \right), 2e^{\frac{\Theta \widetilde{T}}{2}}\sqrt{\widetilde{T}\theta_2} \right) \leq 1
\right\},
\end{equation}
where $\Theta = 1+b+ Q$ and $\theta_2 = \max \left\{ \psi(Q)^2 + \frac{c_2^2 Q^2C_{\texttt{emd}}^2}{\gamma}, Q \right\}$.

\begin{lemma}
For the set $\mathfrak{B}_Q:=\left\{ (w, z)\in \mathcal{S}_w(\widehat{T})\times\mathcal{S}_z(\widehat{T}):~ \parallel (w, z)\parallel_{\mathcal{S}_w(\widehat{T})\times\mathcal{S}_z(\widehat{T})} \leq Q \right\}$, the map $\mathcal{A}$ satisfies
$\mathcal{A}(\mathfrak{B}_Q)\subseteq \mathfrak{B}_Q$.
\end{lemma}
\begin{proof}~
For $(u, v)\in\mathfrak{B}_Q$ using Theorem \ref{thm_existence_linear} and Lemma \ref{fh_h1}, we deduce that  
\begin{multline}\label{estimate_in_BQ}
\parallel\mathcal{A}(u, v)\parallel_{\mathcal{S}_w(\widehat{T})\times\mathcal{S}_z(\widehat{T})}^2 \leq \mathfrak{C}e^{\theta \widehat{T}} \big( \parallel u_0\parallel^2_{H^2(\Omega_1)} + \parallel v_0\parallel^2_{H^1(\Omega_1)}  + \parallel g\parallel^2_{H^1(0, \widehat{T};L^2(\Gamma))\medcap L^{\infty}(0, \widehat{T};H^{\frac{1}{2}}(\Gamma))}  \\
+ \widehat{T} \psi \left( \parallel u\parallel_{L^{\infty}(0, \widehat{T}; L^{\infty}(\Omega_1))} \right)^2 \parallel u\parallel_{W^{1, \infty}(0, \widehat{T}; L^{2}(\Omega_1))}^2  + 2\widehat{T}\parallel u\parallel_{L^{\infty}(0,\widehat{T}; L^{\infty}(\Omega_1))}^2  \parallel v\parallel_{W^{1, \infty}(0,\widehat{T}; L^{2}(\Omega_1))}^2 \\+ C_{\texttt{emd}}\widehat{T}\parallel v\parallel_{L^{\infty}(0,\widehat{T}; H^{1}(\Omega_1))}^2  \parallel  u\parallel_{W^{1,\infty}(0,\widehat{T}; H^1(\Omega_1))}^2 \big).
\end{multline}
Using \eqref{bound_Q} and observing $\psi$ is an increasing function we obtain the following from \eqref{estimate_in_BQ}
\begin{equation}\label{estimate_in_BQ2}
\parallel\mathcal{A}(u, v)\parallel_{\mathcal{S}_w(\widehat{T})\times\mathcal{S}_z(\widehat{T})}^2 \leq e^{\theta \widehat{T}} \left(    \frac{Q^2}{4} +  \mathfrak{C}\widehat{T}Q^2\left(  \psi(Q)^2 +2Q^2 + C_{\texttt{emd}}Q^2 \right)\right).
\end{equation}
Using \eqref{def_time} in \eqref{estimate_in_BQ2} we have 
$\parallel\mathcal{A}(u, v)\parallel_{\mathcal{S}_w(\widehat{T})\times\mathcal{S}_z(\widehat{T})}^2 \leq Q^2$, hence the Lemma.
\end{proof}

\begin{theorem}\label{contraction_thm}
The set $\mathfrak{B}_Q$ is a closed subspace of $L^{\infty}(0,\widehat{T}; L^2(\Omega_1))\times L^{\infty}(0,\widehat{T}; L^2(\Omega_1))$ and the map $\mathcal{A}$ is contraction in $\mathfrak{B}_Q$.
\end{theorem}
\begin{proof}~To see the closedness of $\mathfrak{B}_Q$, consider the sequence $(w_n, z_n)_{n\geq 1}\in\mathfrak{B}_Q$ converging to $(w,z)\in L^{\infty}(0,\widehat{T}; L^2(\Omega_1))\times L^{\infty}(0,\widehat{T}; L^2(\Omega_1))$. But $\mathfrak{B}_Q$ is weakly compact in $\mathcal{S}_w(\widehat{T})\times\mathcal{S}_z(\widehat{T})$, there exists a weakly converging subsequence $(w_{n_k}, z_{n_k})_{k\geq 1}$ converging to $(\hat{w}, \hat{z})\in\mathfrak{B}_Q$ in $\mathcal{S}_w(\widehat{T})\times\mathcal{S}_z(\widehat{T})$. So by the uniqueness of the weak limit, we have $(w,z)\in\mathfrak{B}_Q$.

To show the map $\mathcal{A}$ is contraction, let $u_i, v_i \in \mathfrak{B}_Q$, we set $\mathcal{A}(u_i, v_i)=(w_i, z_i)$ for $i=1, 2$. Observe that the differences $\mathtt{w}:=w_1 - w_2, \mathtt{z}:=z_1 - z_2$ satisfies the following system
\begin{equation}\label{contraction_map}
\begin{cases}
\frac{\partial \mathtt{w}}{\partial t}  = \gamma \Delta \mathtt{w} + \mathtt{f}_1(u_1) -\mathtt{f}_1(u_2) + \mathtt{f}_2(u_1, v_1)-\mathtt{f}_2(u_2, v_2), & (x,t)\in\Omega_1\times(0,T],\\
\frac{\partial \mathtt{z}}{\partial t} = b\mathtt{w} - bc_3\mathtt{z} , & (x,t)\in\Omega_1\times(0,T],\\
\frac{\partial \mathtt{w}}{\partial n}=0, & (x,t)\in\left(\partial\Omega_1 - \Gamma\right)\times(0,T],\\
\gamma\frac{\partial \mathtt{w}}{\partial n} + p\mathtt{w}
 = 0,  & (x,t)\in\Gamma_T\\
\mathtt{w}(x,0)=0, \quad 
\mathtt{z}(x,0)=0, & x\in \Omega_1.
\end{cases}
\end{equation}
Multiplying \eqref{contraction_map}$_1$ by $\mathtt{w}$ and integrate over $\Omega_1$ we have 
\begin{equation}\label{contr_apriori1}
\begin{aligned}
\frac{1}{2}\frac{\partial}{\partial t}\parallel \mathtt{w}\parallel^2 + \gamma \parallel\nabla \mathtt{w} \parallel^2 + p \parallel \mathtt{w} \parallel^2_{\Gamma}  =  
\int_{\Omega_1} \left( \mathtt{f}_1(u_1) -\mathtt{f}_1(u_2)\right) \mathtt{w}  + \int_{\Omega_1} \left( \mathtt{f}_2(u_1, v_1) -\mathtt{f}_2(u_2, v_2) \right) \mathtt{w}.
\end{aligned}
\end{equation}
Note that $\mathtt{f}_2(u_1, v_1) -\mathtt{f}_2(u_2, v_2) =-c_2( u_1v_1 - u_2v_2)=-c_2(\mathtt{u}v_1 + \mathtt{v}u_2)$, where we denote $\mathtt{u}=u_1-u_2$ and $\mathtt{v}=v_1-v_2$. Using Sobolev embedding, Cauchy-Schwarz, and Young's inequality on the right-hand side of \eqref{contr_apriori1}, we have 
\begin{multline}\label{contr_apriori2}
\frac{1}{2}\frac{\partial}{\partial t}\parallel \mathtt{w}\parallel^2 + \gamma \parallel\nabla \mathtt{w} \parallel^2 + p \parallel \mathtt{w} \parallel^2_{\Gamma}  \leq   
\frac{1}{2}\left( \parallel \mathtt{f}_1(u_1) -\mathtt{f}_1(u_2)\parallel^2 + \parallel \mathtt{w}\parallel^2\right) +\frac{\gamma}{2} \parallel\nabla \mathtt{w} \parallel^2\\ + \frac{c_2^2 C_{\texttt{emd}}^2}{2\gamma}\parallel v_1\parallel_{L^{4}}^2  \parallel \mathtt{u}\parallel^2 +\parallel u_2\parallel_{L^{\infty}}\frac{1}{2} \left( \parallel \mathtt{v}\parallel^2 +\parallel \mathtt{w}\parallel^2 \right).
\end{multline}
Multiplying \eqref{contraction_map}$_2$ by $\mathtt{z}$, using Cauchy-Schwarz and Young's inequality and integrate over $\Omega_1$, we obtain
\begin{equation}\label{contr_apriori3}
\frac{1}{2}\frac{\partial}{\partial t}\parallel \mathtt{z}\parallel^2   \leq \frac{b}{2} \left( \parallel \mathtt{z}\parallel^2 + \parallel \mathtt{w}\parallel^2 \right).
\end{equation} 
Adding \eqref{contr_apriori1} and \eqref{contr_apriori3} and integrate over the temporal domain $[0, t]$ where $t\leq \widehat{T}$, and using Gr\"{o}nwall's inequality we obtain 
\begin{multline}\label{contr_apriori4}
\parallel \mathtt{w}\parallel^2 + \parallel \mathtt{z}\parallel^2 \leq e^{\Theta t} \big[\int_0^t \parallel \mathtt{f}_1(u_1) -\mathtt{f}_1(u_2)\parallel^2  + \frac{c_2^2 C_{\texttt{emd}}^2}{2\gamma}\parallel v_1\parallel_{L^{\infty}(0,t;L^{4}(\Omega_1))}^2 \int_0^t \parallel \mathtt{u}\parallel^2 \\ + \parallel u_2\parallel_{L^{\infty}(0,t;L^{\infty}(\Omega_1))} \int_0^t \parallel \mathtt{v}\parallel^2\big].
\end{multline}
Using Lemma \ref{lemma_ex} in \eqref{contr_apriori4}, we have 
\begin{equation}\label{contr_apriori5}
\parallel \mathtt{w}\parallel^2_{L^{\infty}(0,\widehat{T}; L^2(\Omega_1))} + \parallel \mathtt{z}\parallel^2_{L^{\infty}(0,\widehat{T}; L^2(\Omega_1))} \leq 
e^{\Theta \widehat{T}} \widehat{T} \theta_2 \left( \parallel \mathtt{u}\parallel^2_{L^{\infty}(0,\widehat{T}; L^2(\Omega_1))} + \parallel \mathtt{v}\parallel^2_{L^{\infty}(0,\widehat{T}; L^2(\Omega_1))} \right).
\end{equation}
Thus by using \eqref{def_time}, we have 
\begin{equation}
\parallel (w, z)\parallel_{L^{\infty}(0,\widehat{T};L^{2}(\Omega_1))} \leq \frac{1}{2} \parallel (u, v)\parallel_{L^{\infty}(0,\widehat{T};L^{2}(\Omega_1))},
\end{equation}
and hence the map $\mathcal{A}$ is a contraction. Thus by the Picard's theorem, the map $\mathcal{A}$ has a unique fixed point in $\mathfrak{B}_Q$, which proves
the existence and uniqueness of a solution $(w, z)\in\mathfrak{B}_Q$ to the nonlinear problem.
\end{proof}

To obtain the solution over the entire time interval, we first examine the existence of the solution in $[\widehat{T}, \widehat{\widehat{T}}]$. Given the regularity of the solution $(w, z) \in \mathfrak{B}_Q$, it is clear that $w(\widehat{T})$ and $z(\widehat{T})$ are well-defined, with $w(\widehat{T}) \in H^2(\Omega_1)$ and $z(\widehat{T}) \in H^1(\Omega_1)$. Consequently, we can apply a similar analysis in the interval $[\widehat{T}, \widehat{\widehat{T}}]$ with $Q$ in \eqref{bound_Q} being replaced by the following $Q_n$
\begin{equation*}
Q_n^2 \geq 4 \mathfrak{C}\left( \parallel u(\widehat{T})\parallel^2_{H^2(\Omega_1)} + \parallel v(\widehat{T})\parallel^2_{H^1(\Omega_1)}  + \parallel g\parallel^2_{H^1(0, T;L^2(\Gamma))\medcap L^{\infty}(0, T;H^{\frac{1}{2}}(\Gamma))} \right),
\end{equation*}
and $\widehat{\widehat{T}}(Q_n)$ is defined in the same way as in \eqref{def_time}, with the replacement of $Q$ by $Q_n$ in the expression given in \eqref{def_time}.
By repeating the argument in finitely many steps, we can extend the solution over the entire time interval $[0, T]$.
From the above theorem, we can see that the sequences $u_i^{k}, v_i^{k}$ generated through OSWR method are well-posed and satisfy $\left( u_i^{k}, v_i^{k} \right) \in \mathcal{S}_w({T})\times\mathcal{S}_z({T})$ as long as the Robin trace belongs to $H^1(0, {T};L^2(\Gamma))\medcap L^{\infty}(0, {T};H^{\frac{1}{2}}(\Gamma))$. Rewriting the transmission condition in \eqref{oswr1}, it is easy to observe that Robin trace belongs to the above-mentioned space.  Next, we show the convergence of the OSWR method.

\begin{remark}
The assumption \eqref{monotone_fg} can alternatively be employed to establish Theorem \ref{contraction_thm}, albeit with different constants appearing in \eqref{def_time}.
\end{remark}

\subsection{Convergence of the OSWR}

Let us denote the subdomain errors as $e_i^{k}:=u-u_i^{k}$ and $\mathfrak{e}_i^{k}:=v-v_i^{k}$ for $i=1, 2$. We denote $e^{k}$ and $\mathfrak{e}^{k}$ as functions in $L^2(\Omega\times (0, T))$ such that $e^{k}_{|\Omega_i}=e_i^{k}$ and $\mathfrak{e}^{k}_{|\Omega_i}=\mathfrak{e}_i^{k}$ for $i=1, 2$. Define the interface error as $\eta_i^{k}:=g_{i}-g_i^{k}$, where $g_i=\mathtt{B}_iu_{3-i}$ on $\Gamma_T$ for $i=1,2$, and $\boldsymbol{\eta}^{k}:=\left( \eta_1^{k}, \eta_2^{k} \right)^{\top}$.

\begin{theorem}\label{OSWR_thm_conv}
Let $u_0\in H^2(\Omega), v_0\in L^2(\Omega)$, $g_i^{0}\in H^1(0, T;L^2(\Gamma))\medcap L^{\infty}(0, T;H^{\frac{1}{2}}(\Gamma))$, and $\eta^K,  \eta^0 \in L^2(0, T; L^2(\Gamma))$. Suppose ${\mathtt{B}_iu_0}_{|\Gamma}={g_i^{0}}_{|\Gamma}$. Then the OSWR iterates $u_i^{k}, v_i^{k}$ converges to the exact solution $u, v$ in their respective subdomains in $L^{\infty}(0,T;L^2(\Omega))\bigcap L^2(0,T;H^1(\Omega))\times L^{\infty}(0,T;L^2(\Omega))$, and satisfies the following estimate for $0<t\leq T$
\begin{multline}\label{OSWR_est}
 \sum_{k=1}^K  \left( \frac{1}{2}\parallel {e}^{k}(\cdot,t)\parallel^2 + \frac{1}{2}\parallel \mathfrak{e}^{k}(\cdot, t)\parallel^2 +\gamma\sum_{i=1}^2 \parallel {e}_i^{k}\parallel_{L^2(0,t;H^1(\Omega_i))}^2\right) +\frac{\gamma}{4p} \parallel\boldsymbol{\eta}^{K}(t)\parallel^2 \leq \\ \frac{\gamma}{4p} \parallel\boldsymbol{\eta}^{0}(t)\parallel^2 + C_5\sum_{k=1}^K \int_0^t  \left( \frac{1}{2}\parallel {e}^{k}(\cdot,s)\parallel^2 + \frac{1}{2}\parallel \mathfrak{e}^{k}(\cdot, s)\parallel^2 \right), 
\end{multline}
where $C_5$ is some positive constant.
\end{theorem}
\begin{proof}~
The subdomain error equations for $i=1, 2$ are given by
\begin{equation}\label{oswr1_err}
\begin{aligned}
&\left\{
\begin{aligned}
\frac{\partial e_i^{k}}{\partial t}  = \gamma \Delta e_i^{k} - c_1\left[ f(u)-f(u_i^{k})\right] - c_2\left( uv - u_i^{k}v_i^{k}\right), \,\ \text{in} \,\ \Omega_{i,T},\\
\frac{\partial \mathfrak{e}_i^{k}}{\partial t}  = be_i^{k} - bc_3\mathfrak{e}_i^{k} ,\,\ \text{in} \,\ \Omega_{i,T},  \\
\mathtt{B}_i e_i^{k}
 = 
\mathtt{B}_i e_{3-i}^{k-1},
\,\ \mbox{on}\,\ \Gamma_{T}.\\
\end{aligned}\right.
\end{aligned}
\end{equation}
Multiplying the equation \eqref{oswr1_err}$_1$ by $e_i^{k}$ and integrate by parts over the spatial domain we obtain 
\begin{equation}\label{err_est1}
\begin{aligned}
\frac{1}{2}\frac{d}{d t} \parallel e_i^{k}\parallel^2 + \gamma\parallel\nabla e_i^{k} \parallel^2 -\gamma\int_{\Gamma} \frac{\partial e_i^{k}}{\partial n_i} e_i^{k} & =- c_1 \int_{\Omega} f'(\xi_i) \left( e_i^{k}\right)^2 -c_2\int_{\Omega} \left( uv - u_i^{k}v_i^{k}\right)e_i^{k},\\
& \leq c_1\frac{(a+1)^2}{3} \parallel e_i^{k}\parallel^2 -c_2\int_{\Omega} \left( uv - u_i^{k}v_i^{k}\right)e_i^{k},
\end{aligned}
\end{equation}
where on the second inequality, we make use of the inequality \eqref{inequality1} for $u(x,t)< \xi_i(x,t)<u_i^{k}(x,t)$. Next we rewrite the term $\left( uv - u_i^{k}v_i^{k}\right)$ as
\begin{equation}\label{iden2}
\begin{aligned}
uv - u_i^{k}v_i^{k} & = uv- u_i^{k}v +u_i^{k}v -u_i^{k}v_i^{k} 
=  v {e}_i^{k} + u_i^{k} \mathfrak{e}_i^{k} .
\end{aligned}
\end{equation}
Using the identity \eqref{iden2}, the Sobolev embedding result for $u, u_i^{k}\in L^{\infty}(0, T; H^2(\Omega))$, the Cauchy-Schwarz and the Young's inequality we obtain 
\begin{equation}\label{err_int}
\begin{aligned}
-c_2\int_{\Omega} \left( uv - u_i^{k}v_i^{k}\right)e_i^{k} &\leq c_2\left( \frac{1}{2} \parallel u_i^{k}\parallel_{L^{\infty}(0,T;L^{\infty}(\Omega))} \left( \parallel \mathfrak{e}_i^{k}\parallel^2 + \parallel e_i^{k}\parallel^2 \right) + c_g \parallel e_i^{k}\parallel^2 \right),\\
&\leq \frac{C_1}{2} \parallel \mathfrak{e}_i^{k}\parallel^2 + \left( \frac{C_1}{2} + C_2\right) \parallel e_i^{k}\parallel^2.
\end{aligned}
\end{equation}
Using \eqref{err_int} in \eqref{err_est1} and integrate with respect to time $t$ we have
\begin{equation}\label{err_est2}
\begin{aligned}
\frac{1}{2}\parallel e_i^{k}(t)\parallel^2 + \gamma\int_0^t\parallel\nabla e_i^{k} \parallel^2 -\gamma\int_0^t\int_{\Gamma} \frac{\partial e_i^{k}}{\partial n_i} e_i^{k} 
& \leq C_3 \int_0^t\parallel e_i^{k}\parallel^2 + \frac{C_1}{2} \int_0^t\parallel \mathfrak{e}_i^{k}\parallel^2,
\end{aligned}
\end{equation} 
where $C_3=\left( c_1\frac{(a+1)^2}{3}+\frac{C_1}{2} + C_2 \right)$. For obtaining the Robin parameter $p$, employing the identity $AB=\frac{1}{4p}\left[ (A+pB)^2 - (A-pB)^2\right]$ for the boundary integral term in \eqref{err_est2} for first subdomain, we have the following inequality. 
\begin{equation}\label{err_est3}
\begin{aligned}
\frac{1}{2}\parallel e_1^{k}(t)\parallel^2 + \gamma\int_0^t\parallel\nabla e_1^{k} \parallel^2 +\frac{\gamma}{4p}\int_0^t\parallel\mathtt{B}_2 e_1^{k}\parallel_{\Gamma}^2
& \leq C_3 \int_0^t\parallel e_1^{k}\parallel^2 + \frac{C_1}{2} \int_0^t\parallel \mathfrak{e}_1^{k}\parallel^2
 + \frac{\gamma}{4p}\int_0^t\parallel\mathtt{B}_1 e_1^{k}\parallel_{\Gamma}^2,\\
& \leq C_3 \int_0^t\parallel e_1^{k}\parallel^2 + \frac{C_1}{2} \int_0^t\parallel \mathfrak{e}_1^{k}\parallel^2
 + \frac{\gamma}{4p}\int_0^t\parallel\mathtt{B}_1 e_2^{k-1}\parallel_{\Gamma}^2,
\end{aligned}
\end{equation} 
where on the second inequality we use the transmission condition \eqref{oswr1_err}$_3$ for $i=1$. Similarly, for the second subdomain, we have 
\begin{equation}\label{err_est4}
\begin{aligned}
\frac{1}{2}\parallel e_2^{k}(t)\parallel^2 + \gamma\int_0^t\parallel\nabla e_2^{k} \parallel^2 +\frac{\gamma}{4p}\int_0^t\parallel\mathtt{B}_1 e_2^{k}\parallel_{\Gamma}^2
& \leq C_3 \int_0^t\parallel e_2^{k}\parallel^2 + \frac{C_1}{2} \int_0^t\parallel \mathfrak{e}_2^{k}\parallel^2
 + \frac{\gamma}{4p}\int_0^t\parallel\mathtt{B}_2 e_1^{k-1}\parallel_{\Gamma}^2.
\end{aligned}
\end{equation} 
Adding \eqref{err_est3} and \eqref{err_est4} we get
\begin{equation}\label{err_est5}
\begin{aligned}
\frac{1}{2}\sum_{i=1}^2\parallel e_i^{k}(t)\parallel^2 + \gamma \sum_{i=1}^2\int_0^t\parallel\nabla e_i^{k} \parallel^2 +\frac{\gamma}{4p} \sum_{i=1}^2\int_0^t\parallel\mathtt{B}_i e_{3-i}^{k}\parallel_{\Gamma}^2
 \leq C_3 \sum_{i=1}^2\int_0^t\parallel e_i^{k}\parallel^2 \\
 + \frac{C_1}{2} \sum_{i=1}^2\int_0^t\parallel \mathfrak{e}_i^{k}\parallel^2
+\frac{\gamma}{4p}\sum_{i=1}^2\int_0^t\parallel\mathtt{B}_i e_{3-i}^{k-1}\parallel_{\Gamma}^2.
\end{aligned}
\end{equation} 
Summing over $k$ in \eqref{err_est5} for $k=1,2,\cdots, K$ we have 
\begin{equation}\label{err_est6}
\begin{aligned}
\frac{1}{2}\sum_{k=1}^K\sum_{i=1}^2\parallel e_i^{k}(t)\parallel^2 + \gamma\sum_{k=1}^K\sum_{i=1}^2\int_0^t \parallel\nabla e_i^{k} \parallel^2 +\frac{\gamma}{4p} \sum_{i=1}^2\int_0^t\parallel\mathtt{B}_i e_{3-i}^{K}\parallel_{\Gamma}^2
 \leq C_3 \sum_{k=1}^K\sum_{i=1}^2\int_0^t\parallel e_i^{k}\parallel^2 \\
 + \frac{C_1}{2} \sum_{k=1}^K\sum_{i=1}^2\int_0^t\parallel \mathfrak{e}_i^{k}\parallel^2
+\frac{\gamma}{4p}\sum_{i=1}^2\int_0^t\parallel\mathtt{B}_i e_{3-i}^{0}\parallel_{\Gamma}^2.
\end{aligned}
\end{equation} 
Multiplying the equation \eqref{oswr1_err}$_2$ by $\mathfrak{e}_i^{k}$ and integrate over the spatial and time domain we obtain 
\begin{equation}\label{err_est7}
\begin{aligned}
\frac{1}{2}\parallel \mathfrak{e}_i^{k}(t)\parallel^2 
& \leq C_4 \int_0^t\parallel e_i^{k}\parallel^2 + C_4 \int_0^t\parallel \mathfrak{e}_i^{k}\parallel^2.
\end{aligned}
\end{equation} 
Adding the subdomain solution corresponding to the variable $\mathfrak{e}_i^{k}$ using \eqref{err_est7} and summing over $k$ for $k=1, 2, \cdots, K$ we have 
\begin{equation}\label{err_est8}
\begin{aligned}
\frac{1}{2}\sum_{k=1}^K\sum_{i=1}^2\parallel \mathfrak{e}_i^{k}(t)\parallel^2 
& \leq C_4 \sum_{k=1}^K\sum_{i=1}^2\int_0^t\parallel e_i^{k}\parallel^2 + C_4 \sum_{k=1}^K\sum_{i=1}^2\int_0^t\parallel \mathfrak{e}_i^{k}\parallel^2.
\end{aligned}
\end{equation} 
Let us define 
\begin{equation}
\mathtt{E}^{K}(t):= \frac{1}{2}\sum_{k=1}^K\sum_{i=1}^2\parallel {e}_i^{k}(t)\parallel^2 + \frac{1}{2}\sum_{k=1}^K\sum_{i=1}^2\parallel \mathfrak{e}_i^{k}(t)\parallel^2, \; \mathtt{F}^{K}(t):= \gamma\sum_{k=1}^K\sum_{i=1}^2\int_0^t \parallel\nabla e_i^{k} \parallel^2.
\end{equation}
Adding \eqref{err_est6} and \eqref{err_est8} we obtain 
\begin{equation}\label{err_est9}
\begin{aligned}
\mathtt{E}^{K}(t)+ \mathtt{F}^{K}(t) +\mathtt{H}^{K}(t)
 \leq \mathtt{H}^{0}(t)+ C_5\int_0^t\mathtt{E}^{K}(s),
\end{aligned}
\end{equation} 
where $\mathtt{H}^{j}(t):=\frac{\gamma}{4p} \parallel\boldsymbol{\eta}^{j}(t)\parallel^2$ for $j=0, K$. 
Using Gr\"{o}nwall's inequality in \eqref{err_est9} we have
\begin{equation}\label{err_est10}
\begin{aligned}
\mathtt{E}^{K}(t) \leq \underbrace{\mathtt{H}^{0}(t) + C_5\int_0^t \mathtt{H}^{0}(s) e^{C_5 (t-s)}}_{\mathtt{P}(t)} <\infty, \\
\mathtt{F}^{K}(t) \leq \mathtt{H}^{0}(t) +C_5\int_0^t \mathtt{P}(s)<\infty.
\end{aligned}
\end{equation}
As RHS of the general term $\mathtt{E}^{K}(t)$ and $\mathtt{F}^{K}(t)$ in \eqref{err_est10} are independent of the iteration number $k$, 
and are uniformly bounded in $L^{\infty}(0,T)$. Hence $u_i^{k}$ converges to $u_{|\Omega_i}$ in 
$L^{\infty}(0,T;L^2(\Omega_i))\bigcap L^2(0,T;H^1(\Omega_i))$ and $v_i^{k}$ converges to $v_{|\Omega_i}$ in 
$L^{\infty}(0,T;L^2(\Omega_i))$. From \eqref{err_est9} we have the estimate \eqref{OSWR_est}.
\end{proof}

\begin{remark}
The proof of Theorem \ref{OSWR_thm_conv} can also be obtained by utilizing the assumption \eqref{monotone_fg}, leading to the OSWR error estimate given in \eqref{OSWR_est}.
\end{remark}

As we consider incomplete OSWR iteration, the composite solution at the endpoint of each time slice only has the regularity in $L^2$ for the excitation variable. Thus, the overall regularity of the solution generated through the Parareal-OSWR process belongs to $\mathcal{H}$. In this regard, we define the initial error $e_0=\overline{u}_0-u_0, \mathfrak{e}_0= \overline{v}_0-v_0$, where $\overline{u}_0, \overline{v}_0$ are initial solution for incomplete-OSWR iteration and $u_0, v_0$ are exact initial solution.

\begin{theorem}\label{incomplete_OSWR_thm_conv}
Let $u_0\in \mathcal{H}, v_0\in L^2(\Omega)$ and $g_i^{0}\in H^1(0, T;L^2(\Gamma))\medcap L^{\infty}(0, T;H^{\frac{1}{2}}(\Gamma))$. Suppose ${\mathtt{B}_iu_0}_{|\Gamma}={g_i^{0}}_{|\Gamma}$. Then the OSWR iterates $u_i^{k}, v_i^{k}$ satisfies the following estimate for $0<t\leq T$
\begin{multline}\label{OSWR_est_inc}
 \sum_{k=1}^K  \left( \frac{1}{2}\parallel {e}^{k}(\cdot,t)\parallel^2 + \frac{1}{2}\parallel \mathfrak{e}^{k}(\cdot, t)\parallel^2 +\gamma\sum_{i=1}^2 \parallel {e}_i^{k}\parallel_{L^2(0,t;H^1(\Omega_i))}^2\right) +\frac{\gamma}{4p} \parallel\boldsymbol{\eta}^{K}(t)\parallel^2 \leq \\ \frac{K}{2} \left( \parallel e_0\parallel^2 + \parallel \mathfrak{e}_0\parallel^2 \right) +\frac{\gamma}{4p} \parallel\boldsymbol{\eta}^{0}(t)\parallel^2 + \boldsymbol{C}\sum_{k=1}^K \int_0^t \left( \frac{1}{2}\parallel {e}^{k}(\cdot,s)\parallel^2 + \frac{1}{2}\parallel \mathfrak{e}^{k}(\cdot, s)\parallel^2 \right), 
\end{multline}
where $\boldsymbol{C}$ is some positive constant.
\end{theorem}
The proof of Theorem \ref{incomplete_OSWR_thm_conv} proceeds along similar lines to that of Theorem \ref{OSWR_thm_conv}.

\subsection{Best Choice of the Robin Parameter $p$}
To predict the Robin parameter $p$, we linearize the RM model \eqref{RM} with respect to the constant initial solution. We denote $u_l, v_l$ as the linearized variables corresponding to the variables $u, v$. The linearized RM model is given by 
\begin{equation}\label{RM_linear}
\begin{cases}
\frac{\partial u_l}{\partial t}  = \gamma \Delta u_l +\kappa u_l -c_2v_l -c_2, & (x,t)\in\Omega\times(0,T),\\
\frac{\partial v_l}{\partial t} = bu_l - bc_3v_l + b-bc_3 , & (x,t)\in\Omega\times(0,T),
\end{cases}
\end{equation}
where $\kappa=c_1(a-1)-c_2$. Now, the OSWR method for the linearized RM model \eqref{RM_linear} is given by 
\begin{equation}\label{oswr1_linear}
\begin{aligned}
&\left\{
\begin{aligned}
\frac{\partial {u_l}_i^{k}}{\partial t}  = \gamma \Delta {u_l}_i^{k} +\kappa{u_l}_i^{k} -c_2{v_l}_i^{k} -c_2, \,\ \text{in} \,\ \Omega_{i,T},\\
\frac{\partial {v_l}_i^{k}}{\partial t}  = b{u_l}_i^{k} - bc_3{v_l}_i^{k} +b-bc_3 ,\,\ \text{in} \,\ \Omega_{i,T},  \\
\mathtt{B}_i {u_l}_i^{k}
 = 
\mathtt{B}_i {u_l}_{3-i}^{k-1},
\,\ \mbox{on}\,\ \Gamma_{T}.\\
\end{aligned}\right.
\end{aligned}
\end{equation}
For the subdomain error of the linearized RM model \eqref{RM_linear}, we denote ${e_l}_i^{k}=u_l-{u_l}_i^{k}$ and ${\mathfrak{e}_l}_i^{k}=v_l-{v_l}_i^{k}$. Subtracting \eqref{oswr1_linear} from \eqref{RM_linear}, we have the following subdomain error equations for $i=1, 2$ 
\begin{equation}\label{oswr1_linear_err}
\begin{aligned}
&\left\{
\begin{aligned}
\frac{\partial {e_l}_i^{k}}{\partial t}  = \gamma \Delta {e_l}_i^{k} +\kappa{e_l}_i^{k} -c_2{\mathfrak{e}_l}_i^{k}, \,\ \text{in} \,\ \Omega_{i,T},\\
\frac{\partial {\mathfrak{e}_l}_i^{k}}{\partial t}  = b{e_l}_i^{k} - bc_3{\mathfrak{e}_l}_i^{k} ,\,\ \text{in} \,\ \Omega_{i,T},  \\
\mathtt{B}_i {e_l}_i^{k}
 = 
\mathtt{B}_i {e_l}_{3-i}^{k-1},
\,\ \mbox{on}\,\ \Gamma_{T}.\\
\end{aligned}\right.
\end{aligned}
\end{equation} 
Applying the Fourier transform in time for the error equations in \eqref{oswr1_linear_err}, we have the following differential-algebraic equation
\begin{equation}\label{Fourier_err}
\begin{cases}
\mathrm{i}\omega {\widehat{e_l}}_i^{k} = \gamma \partial_{xx} {\widehat{e_l}}_i^{k} +\kappa {\widehat{e_l}}_i^{k} -c_2{\widehat{\mathfrak{e}_l}}_i^{k},\\
\mathrm{i}\omega {\widehat{\mathfrak{e}_l}}_i^{k}  = b{\widehat{e_l}}_i^{k} - bc_3{\widehat{\mathfrak{e}_l}}_i^{k},
\end{cases}
\end{equation}
where $\mathrm{i}=\sqrt{-1}$ and $\omega$ is Fourier variable. From \eqref{Fourier_err}$_2$ we obtain ${\widehat{\mathfrak{e}_l}}_i^{k} = \frac{b}{\mathrm{i}\omega + bc_3}{\widehat{e_l}}_i^{k}$, and using this relation in \eqref{Fourier_err}$_1$ we have to solve the following ODE in each subdomain
\begin{equation}
\frac{d^2 {\widehat{e_l}}_i^{k}}{d x^2} - \underbrace{\left( \frac{\mathrm{i}\omega}{\gamma} -\frac{\kappa}{\gamma} + \frac{c_2b}{\gamma(\mathrm{i}\omega + bc_3)} \right)}_{z_{\omega}}{\widehat{e_l}}_i^{k}=0,
\end{equation}
with the characteristic roots $\lambda_1=\sqrt{z_{\omega}}, \lambda_2=-\lambda_1$. We rewrite $\sqrt{z_{\omega}}$ as
\begin{equation}\label{sqrt_zw}
\sqrt{z_{\omega}}=\sqrt{\frac{\sqrt{R_p(\omega)^2 + I_p(\omega)^2} + R_p(\omega)}{2}} + \mathrm{i}\sign(\omega) \sqrt{\frac{\sqrt{R_p(\omega)^2 + I_p(\omega)^2} - R_p(\omega)}{2}}
\end{equation}
where $R_p(\omega)=\left( -\frac{\kappa}{\gamma} + \frac{c_2c_3b^2}{\gamma(b^2c_3^2+\omega^2)}\right)$ and $I_p(\omega)= \left( \frac{\omega}{\gamma} - \frac{\omega bc_2}{\gamma(b^2c_3^2 + \omega^2)}\right)$. Clearly $\Re(\lambda_1)>0$ and $\Re(\lambda_2)<0$. Therefore, we have the following subdomain solution 
\begin{equation}\label{sub_err}
\begin{cases}
{\widehat{e_l}}_1^{k}(x, \omega) = \chi_1^{k}(\omega) e^{\lambda_1 x},  & (x,\omega)\in(-\infty, 0)\times\mathbb{R},\\
{\widehat{e_l}}_2^{k}(x, \omega) = \chi_2^{k}(\omega) e^{-\lambda_1 x},  & (x,\omega)\in(0, \infty)\times\mathbb{R},
\end{cases}
\end{equation}
where $\chi_1^{k}(\omega)$ and $\chi_2^{k}(\omega)$ will be computed using the interface condition \eqref{oswr1_linear_err}$_3$. So using \eqref{sub_err} in \eqref{oswr1_linear_err}$_3$ we have  
\[ \chi_1^{k}(\omega) = \frac{p-\gamma\lambda_1}{p+\gamma\lambda_1}\chi_2^{k-1}(\omega),\; \chi_2^{k}(\omega) = \frac{p-\gamma\lambda_1}{p+\gamma\lambda_1}\chi_1^{k-1}(\omega).\]
Therefore, we have the following recurrence relation for the subdomain error as 
\begin{equation}\label{rec_relation}
{\widehat{e_l}}_i^{k}(x, \omega) = \underbrace{\left( \frac{p-\gamma\lambda_1}{p+\gamma\lambda_1} \right)^2}_{\rho(p, \omega)} {\widehat{e_l}}_i^{k-2}(x, \omega).
\end{equation}
From \eqref{rec_relation}, we have the convergence factor of the OSWR method as $\rho(p, \omega)$. For any given frequency, we can see that the convergence factor vanishes for $p=\gamma\lambda_1$. But one has to choose $p$ in such a way that it minimizes the maximum error for any given frequency, that is to say, find $p=p_t$ such that the following hold
\begin{equation}\label{minmax}
p_t:=\min_{p>0}\max_{[\omega_{\min}, \omega_{\max}]} \vert \rho(p, \omega)\vert,
\end{equation} 
where $\omega_{\min}=\frac{\pi}{T}, \omega_{\max}=\frac{\pi}{\Delta t}$ are the lowest and highest frequencies on a numerical grid. In practice, \eqref{minmax} can be solved numerically as a min-max problem, and we denote that by $p_{\texttt{num}}$. Here, we also analyze the min-max problem analytically. Let $r=\Re(\lambda_1)$, then the convergence factor $\vert\rho(p, \omega)\vert$ can be rewritten as 
\begin{equation}\label{minmax2}
\rho_s(\widetilde{p}, r)=\frac{(r - \widetilde{p})^2 + r^2 - s}{(r + \widetilde{p})^2 + r^2 - s},
\end{equation}
where $\widetilde{p}=\frac{p}{\gamma}$, and $s\in [R_p(\omega_{\min}), R_p(\omega_{\max})]$. It is easy to observe that for fixed $\widetilde{p}$ and $r$, the function $\rho_s$ in \eqref{minmax2} is monotonically decreasing in $s$, thus we have  
\begin{equation}\label{minmax3}
\rho_s(\widetilde{p}, r)\leq \rho(\widetilde{p}, r)=\frac{(r - \widetilde{p})^2 + r^2 - R_p(\omega_{\min})}{(r + \widetilde{p})^2 + r^2 - R_p(\omega_{\min})}.
\end{equation}
Therefore, it is enough to solve the following min-max problem  
\begin{equation}\label{minmax4}
p_a:=\min_{\widetilde{p}>0}\max_{[\sqrt{R_p(\omega_{\min})}, r_{\max}]} \rho(\widetilde{p}, r),
\end{equation} 
where $r_{\max}=\Re(\lambda_1(\omega_{\max}))$. The solution of the above min-max problem \eqref{minmax4} is summarized in the following theorem. 

\begin{theorem}\label{thm_robinpara}
The best choice of the Robin parameter $p$ is given by 
\begin{equation*}
p_a= \begin{cases}
\gamma\sqrt{\sqrt{R_p(\omega_{\min})}\left(2r_{\max} + \sqrt{R_p(\omega_{\min})} \right)} &\text{for}\; r_{\max}\geq \frac{1+\sqrt{5}}{2}\sqrt{R_p(\omega_{\min})},\\
\gamma \sqrt{2r_{\max}^2 - R_p(\omega_{\min})} &\text{for}\; r_{\max}< \frac{1+\sqrt{5}}{2}\sqrt{R_p(\omega_{\min})},
\end{cases}
\end{equation*}
where $p_a$ is the solution of the min-max problem \eqref{minmax4}.
\end{theorem}
The proof of Theorem \ref{thm_robinpara} follows from Theorem 5.17 in \cite{gander2007optimized}. Note that for certain values of $\kappa$, $R_p$ may be negative. But for the practical use of the parameters involved in the expression of $R_p$, we always have $R_p$ to be positive.

To find the Robin parameter in 2D, we apply the Fourier transform in time and spatial variable along $y$-direction for the error equations in \eqref{oswr1_linear_err}, then we have the following differential-algebraic equation
\begin{equation}\label{Fourier_err_2d}
\begin{cases}
\mathrm{i}\omega {\widehat{\widehat{e_l}}}_i^{k} = \gamma \partial_{xx} {\widehat{\widehat{e_l}}}_i^{k}-\gamma \zeta^2 {\widehat{\widehat{e_l}}}_i^{k} +\kappa {\widehat{\widehat{e_l}}}_i^{k} -c_2{\widehat{\widehat{\mathfrak{e}_l}}}_i^{k},\\
\mathrm{i}\omega {\widehat{\widehat{\mathfrak{e}_l}}}_i^{k}  = b{\widehat{\widehat{e_l}}}_i^{k} - bc_3{\widehat{\widehat{\mathfrak{e}_l}}}_i^{k},
\end{cases}
\end{equation}
where $\zeta$ is the Fourier variable corresponding to $y$-direction. From \eqref{Fourier_err_2d}$_2$ we obtain ${\widehat{\widehat{\mathfrak{e}_l}}}_i^{k} = \frac{b}{\mathrm{i}\omega + bc_3}{\widehat{\widehat{e_l}}}_i^{k}$, and using this relation in \eqref{Fourier_err_2d}$_1$ we solve the following ODE in each subdomain
\begin{equation}
\frac{d^2 {\widehat{\widehat{e_l}}}_i^{k}} {d x^2} - \underbrace{\left(z_{\omega}+\zeta^2\right)}_{z_{\omega, \zeta}}{\widehat{\widehat{e_l}}}_i^{k}=0,
\end{equation}
with the characteristic roots $\lambda_1^{2d}=\sqrt{z_{\omega, \zeta}}, \lambda_2^{2d}=-\lambda_1^{2d}$. The expression of $\lambda_1^{2d}$ has the same form as in \eqref{sqrt_zw} with a replacement of $R_p(\omega)$ by $R_p(\omega, \zeta)$ in \eqref{sqrt_zw}, where $R_p(\omega, \zeta)=R_p(\omega)+\zeta^2$. Now, similar to one-dimensional analysis, we obtain the following subdomain error recurrence relation
\begin{equation}\label{rec_relation_2d}
{\widehat{\widehat{e_l}}}_i^{k}(x, \omega, \zeta) = \underbrace{\left( \frac{p-\gamma\lambda_1^{2d}}{p+\gamma\lambda_1^{2d}} \right)^2}_{\rho(p, \omega, \zeta)} {\widehat{\widehat{e_l}}}_i^{k-2}(x, \omega, \zeta),
\end{equation}
where $\rho(p, \omega, \zeta)$ is the convergence factor. 
So to find $p$ we have to solve the following min-max problem
\begin{equation}\label{minmax-2d}
p_t:=\min_{p>0}\max_{[\omega_{\min}, \omega_{\max}]} \max_{[\zeta_{\min}, \zeta_{\max}]}\vert \rho(p, \omega, \zeta)\vert,
\end{equation} 
where $\zeta_{\min}, \zeta_{\max}$ are the lowest and highest frequencies on a numerical grid.
\subsubsection{Reduced dynamics and prediction of $p$}
Here, we discuss another way to predict the Robin parameter $p$. We notice that the recovery variable $v$ is small in magnitude, and also, the coefficient $c_2$ is of $\mathcal{O}(10^{-1})$ or smaller in practice. Thus, by ignoring the behavior of the recovery variable, we have from \eqref{RM_linear}
\begin{equation}\label{RM_reduced}
\frac{\partial u_l}{\partial t}  = \gamma \Delta u_l +\kappa u_l,  (x,t)\in\Omega\times(0,T).
\end{equation}
Following \cite{gander2010optimized}, we find the Robin parameter $p=p_{\texttt{red}}$ for the linear reaction-diffusion equation \eqref{RM_reduced} as
\begin{equation}
p_{\texttt{red}}=\sqrt{\pi\gamma}\left( 2\sqrt{(4\gamma\kappa + 4\gamma^2\zeta_{\min}^2)^2 +16\gamma^2\omega_{\min}} + 4\gamma\kappa + 4\gamma^2\zeta_{\min}^2 \right)^{\frac{1}{4}} \left( \Delta t \right)^{-\frac{1}{2}}.
\end{equation}

\section{Convergence of Parareal-OSWR method}\label{Section4}
The Parareal-OSWR method for the RM model is to use \eqref{discrete_RM} as coarse operator and OSWR method \eqref{oswr1} with incomplete interaction as fine operator in the predictor-corrector scheme \eqref{parareal_rm}. This procedure is explained in the following Algorithm~\ref{alg:parareal}.
\begin{algorithm}
\caption{Parareal - incomplete OSWR Method}
\label{alg:parareal}
\begin{algorithmic}[1]
    \State Initialize the Parareal iteration by solving the coarse operator, i.e., for $\mathtt{X}_{0}^0=[u_0, v_0]^{\top}$ obtain $\mathtt{X}_n^{0}=\mathtt{G}(T_{n}, T_{n-1}, \mathtt{X}_{n-1}^0)$ for $n=1,2,\cdots, N$.
    \State Initialize the Robin data by taking $\boldsymbol{g}_{n}^{0,0}=\left({g}_{1,n}^{0,0}, {g}_{2,n}^{0,0} \right)$ on the time slice $[T_{n-1}, T_n]$ for $n=1,2,\cdots, N$. A possible choice of ${g}_{i,n}^{0,0} = {\mathtt{B}_i {U}_n^{0}}_{|\Gamma_T}$ for $i=1, 2$. 
    \For{$l = 0, 1, \cdots $ \textit{(Parareal loop)}}
         \State Compute $K$ incomplete OSWR iteration by 
         \begin{equation}
         \left( [u_n^{l,K}, v_n^{l,K}]^{\top}, \boldsymbol{g}_n^{l,K} \right)= \text{OSWR}_{K} \left( \mathtt{X}_n^{l}, \boldsymbol{g}_n^{l,0} \right)
         \end{equation}
         in parallel on the time slices $[T_{n-1}, T_n]$ for $n=1,2,\cdots, N$.
        \State Do the following correction:
\begin{equation}\label{parareal_inc_oswr}
    \begin{aligned}
       \mathtt{X}_0^{l+1} & =[ u_0, v_0]^{\top},\\
      \mathtt{X}_{n+1}^{l+1} & =\mathtt{G}(T_{n+1}, T_n, \mathtt{X}_n^{l+1})+[u_n^{l,K}(\cdot, T_{n+1}), v_n^{l,K}(\cdot, T_{n+1})]^{\top}-\mathtt{G}(T_{n+1}, T_n, \mathtt{X}_n^l),
    \end{aligned}       
\end{equation}
        \State Update the space-time interface trace by \begin{equation}\label{parareal_inc_oswr2}
    \boldsymbol{g}_n^{l+1,0}= \boldsymbol{g}_n^{l,K}      
\end{equation}
    \EndFor
\end{algorithmic}
\end{algorithm}

To prove the convergence of the Parareal-incomplete OSWR method, we define the following error quantities
\begin{itemize}
 \item  error at coarse time points as $\mathcal{E}_{n}^l:=[U_n^l-u(\cdot, T_n), V_n^l-v(\cdot, T_n)]^{\top}$ for $n=0,1,\cdots, N$, 
 \item  error in the space-time slice $\Omega\times[T_{n-1}, T_n]$ as $e_n^{l,k}:=u_n^{l,k}-u, \mathfrak{e}_n^{l,k}:=v_n^{l,k}-v$ such that $e_{i,n}^{l,k}:= {e_n^{l,k}}_{|\Omega_i}, \mathfrak{e}_{i,n}^{l,k}:= {\mathfrak{e}_n^{l,k}}_{|\Omega_i}$ for $n=0,1,\cdots, N$,
 \item error at the space-time inteface as $ \boldsymbol{\eta}_n^{l,k}=  \boldsymbol{g}_n^{l,k}-  \boldsymbol{g}_n$, where $ \boldsymbol{g}_n=[{\mathtt{B}_1 u_n^{0}}_{|\Gamma\times [T_{n-1}, T_n]}, {\mathtt{B}_2 u_n^{0}}_{|\Gamma\times [T_{n-1}, T_n]}]$.
\end{itemize}
\begin{lemma}\label{PA_lemma1}
For $n=0,1, \cdots, N-1$ and coarse step $\Delta T$ satisfying \eqref{coarse_time_step}, the inequality holds:
\begin{equation}\label{PA_err}
\sum_{l=0}^{L} \parallel \mathcal{E}_{n+1}^{l+1}\parallel^2 \leq 8C_{\texttt{Lip}}^2 \sum_{l=0}^{L+1}  \parallel \mathcal{E}_{n}^{l}\parallel^2  + 2 \sum_{l=0}^L \left( \parallel e_n^{l,k}(\cdot, T_{n+1}) \parallel^2 + \parallel \mathfrak{e}_n^{l,k}(\cdot, T_{n+1}) \parallel^2 \right).
\end{equation}
\end{lemma}
\begin{proof}~
Using the Parareal scheme \eqref{parareal_rm} in the definition of the error term $\mathcal{E}_{n+1}^{l+1}$ we obtain 
\begin{equation}\label{PA_err1}
\begin{aligned}
\mathcal{E}_{n+1}^{l+1} &= \mathtt{G}(T_{n+1}, T_n, \mathtt{X}_n^{l+1})+\mathtt{F}(T_{n+1}, T_n, \mathtt{X}_n^l)-\mathtt{G}(T_{n+1}, T_n, \mathtt{X}_n^l)- [u(\cdot, T_{n+1}), v(\cdot, T_{n+1})]^{\top}\\
&=\mathtt{G}(T_{n+1}, T_n, \mathtt{X}_n^{l+1}) - \mathtt{G}(T_{n+1}, T_n, \mathtt{X}_n) +\mathtt{F}(T_{n+1}, T_n, \mathtt{X}_n^l)- [u(\cdot, T_{n+1}), v(\cdot, T_{n+1})]^{\top}\\& + \mathtt{G}(T_{n+1}, T_n, \mathtt{X}_n) -\mathtt{G}(T_{n+1}, T_n, \mathtt{X}_n^l).
\end{aligned}
\end{equation}
Taking norm and using triangle inequality in \eqref{PA_err1}, and then using the Lemma \ref{Lipschitz_c_op_ct} we have 
\begin{equation}\label{PA_err2}
\begin{aligned}
\parallel\mathcal{E}_{n+1}^{l+1}\parallel^2 &\leq  2 \parallel [e_n^{l,k}(\cdot, T_{n+1}), \mathfrak{e}_n^{l,k}(\cdot, T_{n+1})]^{\top}] \parallel^2 + 2C_{\texttt{Lip}}^2 \left( \parallel\mathcal{E}_{n}^{l+1}\parallel  + \parallel\mathcal{E}_{n}^{l}\parallel \right)^2\\
&\leq  2 \left( \parallel e_n^{l,k}(\cdot, T_{n+1}) \parallel^2 + \parallel \mathfrak{e}_n^{l,k}(\cdot, T_{n+1}) \parallel^2 \right) + 4C_{\texttt{Lip}}^2 \left( \parallel\mathcal{E}_{n}^{l+1}\parallel^2  + \parallel\mathcal{E}_{n}^{l}\parallel^2 \right).
\end{aligned}
\end{equation}
Taking sum in \eqref{PA_err2} for $l=0,1,\cdots, L$ we get
\begin{equation}\label{PA_err3}
\sum_{l=0}^{L} \parallel \mathcal{E}_{n+1}^{l+1}\parallel^2 \leq 4C_{\texttt{Lip}}^2 \sum_{l=0}^{L} \left( \parallel \mathcal{E}_{n}^{l+1}\parallel^2 +\parallel \mathcal{E}_{n}^{l}\parallel^2\right) + 2 \sum_{l=0}^L \left( \parallel e_n^{l,k}(\cdot, T_{n+1}) \parallel^2 + \parallel \mathfrak{e}_n^{l,k}(\cdot, T_{n+1}) \parallel^2 \right).
\end{equation}
From \eqref{PA_err3} we derive \eqref{PA_err}.
\end{proof}

\begin{lemma}\label{PA_lemma2}
For $n=0,1, \cdots, N-1$, the following inequality holds:
\begin{equation}\label{PA_err0}
\sum_{l=0}^L \left( \parallel e_n^{l,k}(\cdot, T_{n+1}) \parallel^2 + \parallel \mathfrak{e}_n^{l,k}(\cdot, T_{n+1}) \parallel^2 \right) \leq \frac{K\sigma}{2} \sum_{l=0}^L \parallel \mathcal{E}_{n}^{l}\parallel^2  +\frac{\gamma\sigma}{4p} \parallel\boldsymbol{\eta}_n^{0,0}\parallel^2,
\end{equation}
where $\sigma=\left( 1+\frac{e^{\boldsymbol{C} \Delta T}-1}{\boldsymbol{C}} \right)$.
\end{lemma}
\begin{proof}~
Adapting the error estimate in Theorem \ref{incomplete_OSWR_thm_conv} for incomplete OSWR in the time slice $[T_{n}, T_{n+1}]$ with $\mathcal{E}_{n}^{l}\in\mathcal{H}$ as initial error, $\boldsymbol{\eta}_n^{l,k}$ as interface error and ${e}_n^{l,k}, \mathfrak{e}_n^{l,k}$ as subdomain error we obtain
\begin{multline}\label{OSWR_est_parareal}
 \sum_{k=1}^K  \left( \frac{1}{2}\parallel {e}_n^{l,k}(\cdot,t)\parallel^2 + \frac{1}{2}\parallel \mathfrak{e}_n^{l,k}(\cdot, t)\parallel^2 +\gamma\sum_{i=1}^2 \parallel {e}_{n,i}^{l,k}\parallel_{L^2(0,t;H^1(\Omega_i))}^2\right) +\frac{\gamma}{4p} \parallel\boldsymbol{\eta}_n^{l,K}(t)\parallel^2 \leq \\ \frac{K}{2} \parallel \mathcal{E}_{n}^{l}\parallel^2  +\frac{\gamma}{4p} \parallel\boldsymbol{\eta}_n^{l,0}(t)\parallel^2 + \boldsymbol{C}\sum_{k=1}^K\int_{T_{n}}^t \left( \frac{1}{2}\parallel {e}_n^{l, k}(\cdot,s)\parallel^2 + \frac{1}{2}\parallel \mathfrak{e}_n^{l,k}(\cdot, s)\parallel^2 \right).
\end{multline}
Observe that $\boldsymbol{\eta}_n^{l+1,0}= \boldsymbol{\eta}_n^{l,K}$ hold at the space-time interface, and using this in \eqref{OSWR_est_parareal} we have  
\begin{multline}\label{OSWR_est_parareal2}
 \sum_{k=1}^K  \left( \frac{1}{2}\parallel {e}_n^{l,k}(\cdot,t)\parallel^2 + \frac{1}{2}\parallel \mathfrak{e}_n^{l,k}(\cdot, t)\parallel^2 +\gamma\sum_{i=1}^2 \parallel {e}_{n,i}^{l,k}\parallel_{L^2(0,t;H^1(\Omega_i))}^2\right) +\frac{\gamma}{4p} \parallel\boldsymbol{\eta}_n^{l+1,0}(t)\parallel^2 \leq \\ \frac{K}{2} \parallel \mathcal{E}_{n}^{l}\parallel^2  +\frac{\gamma}{4p} \parallel\boldsymbol{\eta}_n^{l,0}(t)\parallel^2 + \boldsymbol{C}\sum_{k=1}^K\int_{T_{n}}^t \left( \frac{1}{2}\parallel {e}_n^{l, k}(\cdot,s)\parallel^2 + \frac{1}{2}\parallel \mathfrak{e}_n^{l,k}(\cdot, s)\parallel^2 \right).
\end{multline}
Taking sum in \eqref{OSWR_est_parareal2} for $l=0, 1,\cdots, L$ we have 
\begin{multline}\label{OSWR_est_parareal3}
\sum_{l=0}^L  \sum_{k=1}^K  \left( \frac{1}{2}\parallel {e}_n^{l,k}(\cdot,t)\parallel^2 + \frac{1}{2}\parallel \mathfrak{e}_n^{l,k}(\cdot, t)\parallel^2 +\gamma\sum_{i=1}^2 \parallel {e}_{n,i}^{l,k}\parallel_{L^2(0,t;H^1(\Omega_i))}^2\right) +\frac{\gamma}{4p} \parallel\boldsymbol{\eta}_n^{L+1,0}(t)\parallel^2 \leq \\ \frac{K}{2} \sum_{l=0}^L \parallel \mathcal{E}_{n}^{l}\parallel^2  +\frac{\gamma}{4p} \parallel\boldsymbol{\eta}_n^{0,0}(t)\parallel^2 + \boldsymbol{C}\sum_{l=0}^L \sum_{k=1}^K\int_{T_{n}}^t \left( \frac{1}{2}\parallel {e}_n^{l, k}(\cdot,s)\parallel^2 + \frac{1}{2}\parallel \mathfrak{e}_n^{l,k}(\cdot, s)\parallel^2 \right).
\end{multline}
Define $\mathcal{G}^{L,K}(t):=\sum_{l=0}^L  \sum_{k=1}^K  \left( \frac{1}{2}\parallel {e}_n^{l,k}(\cdot,t)\parallel^2 + \frac{1}{2}\parallel \mathfrak{e}_n^{l,k}(\cdot, t)\parallel^2 \right)$, then from \eqref{OSWR_est_parareal3} we have 
\begin{equation}\label{OSWR_est_parareal4}
\mathcal{G}^{L,K}(t) \leq \frac{K}{2} \sum_{l=0}^L \parallel \mathcal{E}_{n}^{l}\parallel^2  +\frac{\gamma}{4p} \parallel\boldsymbol{\eta}_n^{0,0}(t)\parallel^2 + \boldsymbol{C} \int_{T_{n}}^t \mathcal{G}^{L,K}(s). 
\end{equation}
Using Gr\"{o}nwall's inequality in \eqref{OSWR_est_parareal4} and then evaluated at $t=T_{n+1}$ we obtain 
\begin{equation*}
\mathcal{G}^{L,K}(T_{n+1}) \leq  \frac{K\sigma}{2} \sum_{l=0}^L \parallel \mathcal{E}_{n}^{l}\parallel^2  +\frac{\gamma\sigma}{4p} \parallel\boldsymbol{\eta}_n^{0,0}(T_{n+1})\parallel^2.
\end{equation*}
 Hence, the estimate \eqref{PA_err0}.
\end{proof}

\begin{theorem}
For $\Delta T$ satisfying \eqref{coarse_time_step} the Parareal iterates $U_n^l\rightarrow u(\cdot, T_n)$ and $ V_n^l \rightarrow v(\cdot, T_n)$ in $L^2(\Omega)$ as $l\rightarrow\infty$.
\end{theorem}
\begin{proof}~
Using Lemma \ref{PA_lemma2} in Lemma \ref{PA_lemma1} we have
\begin{equation}\label{PA_err_conv}
\sum_{l=0}^{L} \parallel \mathcal{E}_{n+1}^{l+1}\parallel^2 \leq 8C_{\texttt{Lip}}^2 \sum_{l=0}^{L+1}  \parallel \mathcal{E}_{n}^{l}\parallel^2  + K\sigma \sum_{l=0}^L \parallel \mathcal{E}_{n}^{l}\parallel^2  +\frac{\gamma\sigma}{2p} \parallel\boldsymbol{\eta}_n^{0,0}\parallel^2.
\end{equation} 
Rewriting the above inequality as follows
\begin{equation}\label{PA_err_conv2}
\sum_{l=0}^{L+1} \parallel \mathcal{E}_{n+1}^{l}\parallel^2 \leq C_{\sigma} \sum_{l=0}^{L+1}  \parallel \mathcal{E}_{n}^{l}\parallel^2 + \mathfrak{R}_n, \text{for}\; n=0,1,\cdots, N-1,
\end{equation}
where $C_{\sigma}=8C_{\texttt{Lip}}^2 + K\sigma$ and $\mathfrak{R}_n = \parallel \mathcal{E}_{n+1}^{0}\parallel^2+\frac{\gamma\sigma}{2p} \parallel\boldsymbol{\eta}_n^{0,0}\parallel^2$.
After applying induction on \eqref{PA_err_conv2}, we derive
\begin{equation}\label{PA_err_conv3}
\sum_{l=0}^{L+1} \parallel \mathcal{E}_{n+1}^{l}\parallel^2 \leq C_{\sigma}^{n+1} \sum_{l=0}^{L+1}  \parallel \mathcal{E}_{0}^{l}\parallel^2 + \sum_{j=0}^n C_{\sigma}^{j} \mathfrak{R}_{n-j}, \text{for}\; n=0,1,\cdots, N-1.
\end{equation}
As $\mathcal{E}_{0}^{l}=0$ for every $l\geq 0$, from \eqref{PA_err_conv3} we have 
\begin{equation}\label{PA_err_conv4}
\sum_{l=0}^{L+1} \parallel \mathcal{E}_{n+1}^{l}\parallel^2 \leq \sum_{j=0}^n C_{\sigma}^{j} \mathfrak{R}_{n-j}, \text{for}\; n=0,1,\cdots, N-1.
\end{equation}
Since the RHS of \eqref{PA_err_conv4} remains independent of the parareal iteration $L$, the partial sum $\sum_{l=0}^{L} \parallel \mathcal{E}_{n}^{l}\parallel^2$ is uniformly bounded. Consequently, $\mathcal{E}_{n}^{l}$ tends toward zero in the $L^2(\Omega)$-norm as $l\rightarrow\infty$. Hence, the theorem.
\end{proof}
\begin{remark}
     Here, we just want to comment on the overall cost of the Parareal incomplete OSWR method. The computational cost to reach a desired accuracy is the number of inner OSWR iterations times the number of outer Parareal iterations. For cost analysis in the linear case, see \cite{bui2022coupling}.
\end{remark}

\section{Numerical Illustration}\label{Section5}
In this section, we present the numerical experiments conducted using the Parareal-incomplete OSWR method for the RM model \eqref{RM}. We discretize the RM model using centered finite differences in space and the backward Euler method in time, employing the linearization described in \eqref{discrete_RM}. The problem parameters $a$, $b$, $c_1$, $c_2$, $c_3$, and $\gamma$ are set to $a = 0.13$, $b = 0.013$, $c_1 = 0.26$, $c_2 = 0.1$, $c_3 = 1.0$, $\gamma = 0.001$, unless otherwise specified. We consider the following smooth initial data in our numerical experiments unless otherwise specified. 
\begin{eqnarray}
 \text{ For 1D computations, } u_0 &=e^{-4x^2}, v_0=0\; \text{ where }\; \Omega=(0, 1), \text{and}  \\
  \text{ For 2D computations, } u_0 &=e^{-4(x^2+y^2)}, v_0=0\; \text{ where }\; \Omega=(0, 1)^2, \label{initial-2d}
\end{eqnarray}
\subsection{Numerical results of OSWR method and varying Robin parameters} 
Here, we analyze the numerical convergence of the OSWR method for different choices of the Robin parameter $p$. The iteration process starts with a zero initial guess at the midpoint interface and continues until the error  $\parallel u-u^k\parallel_{L^{\infty}(0,T;L^2(\Omega))}$ and $\parallel v-v^k\parallel_{L^{\infty}(0,T;L^2(\Omega))}$ reaches a tolerance of $10^{-6}$. Here, 
$u, v$ represents the discrete sequential solution using the linearize scheme \eqref{discrete_RM} and 
$u^k, v^k$ represents the discrete OSWR solution at the 
$k$-th iteration.
\subsubsection{Experiments in 1D}
Figure~\ref{oswr_diff_T} illustrates the convergence trajectories of the OSWR method for variables $u$ and $v$, considering different time windows and varying mesh parameters.
\begin{figure}[h!]
    \centering
    \subfloat{{\includegraphics[height=4cm,width=7cm]{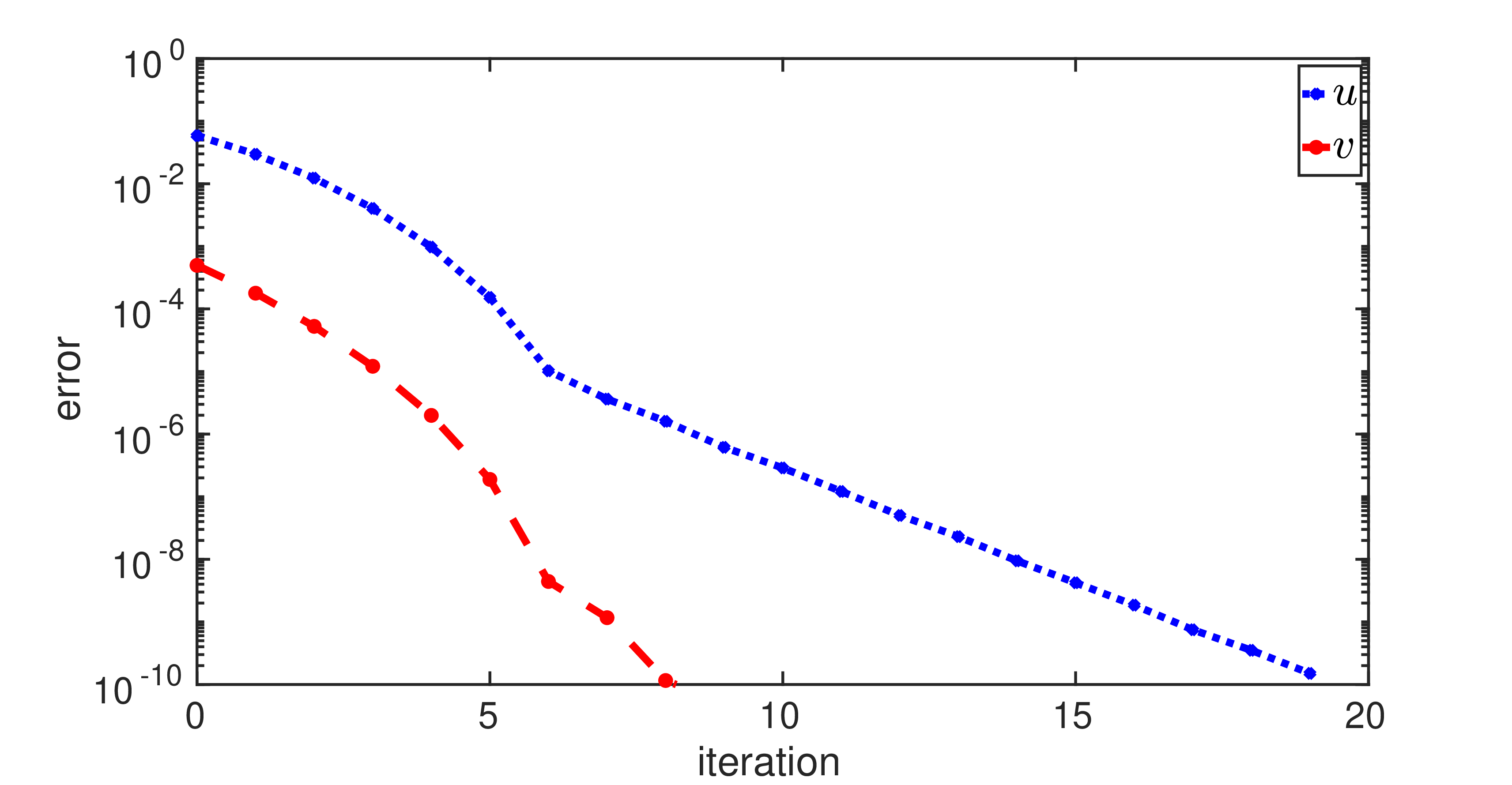} }}
    \subfloat{{\includegraphics[height=4cm,width=7cm]{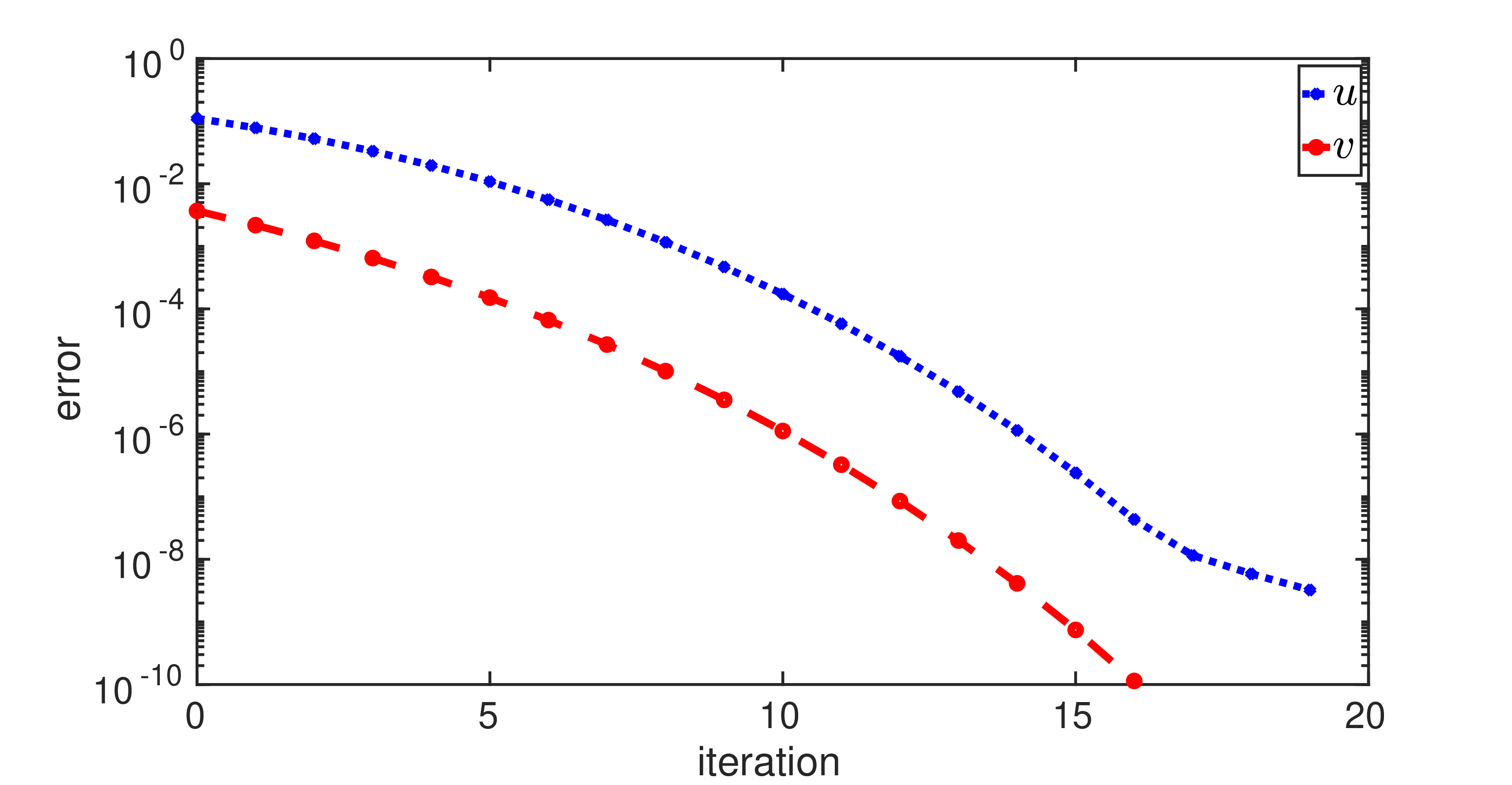} }}
    \caption{On the left: convergence of OSWR with $h=1/128, \Delta t=0.01, p_{\texttt{num}}=0.18$ and $T=1$; On the right: convergence of OSWR with $h=1/256, \Delta t=0.005, p_{\texttt{num}}=0.16$ and $T=4$.}
    \label{oswr_diff_T}
\end{figure}

We observe that the recovery variable $v$ converges faster compared to the excitation variable $u$ in both cases. Additionally, it's evident that the convergence trends extend as the time frame broadens. Subsequently, we compare the convergence behavior of the excitation variable $u$ for $p=p_a$ and $p=p_{\texttt{num}}$. As depicted in Figure \ref{oswr_diff_p}, we observe that for $p = p_a$, the error diminishes faster compared to $p = p_{\texttt{num}}$ up to a certain number of iterations in the small time window. However, after that point, the error contracts faster for $p = p_{\texttt{num}}$. In the long-time scenario, error contracts a bit faster for $p_a$. The error curves for $p_a$ and $p_{\texttt{num}}$ closely align.
\begin{figure}[h!]
    \centering
    \subfloat{{\includegraphics[height=4cm,width=7cm]{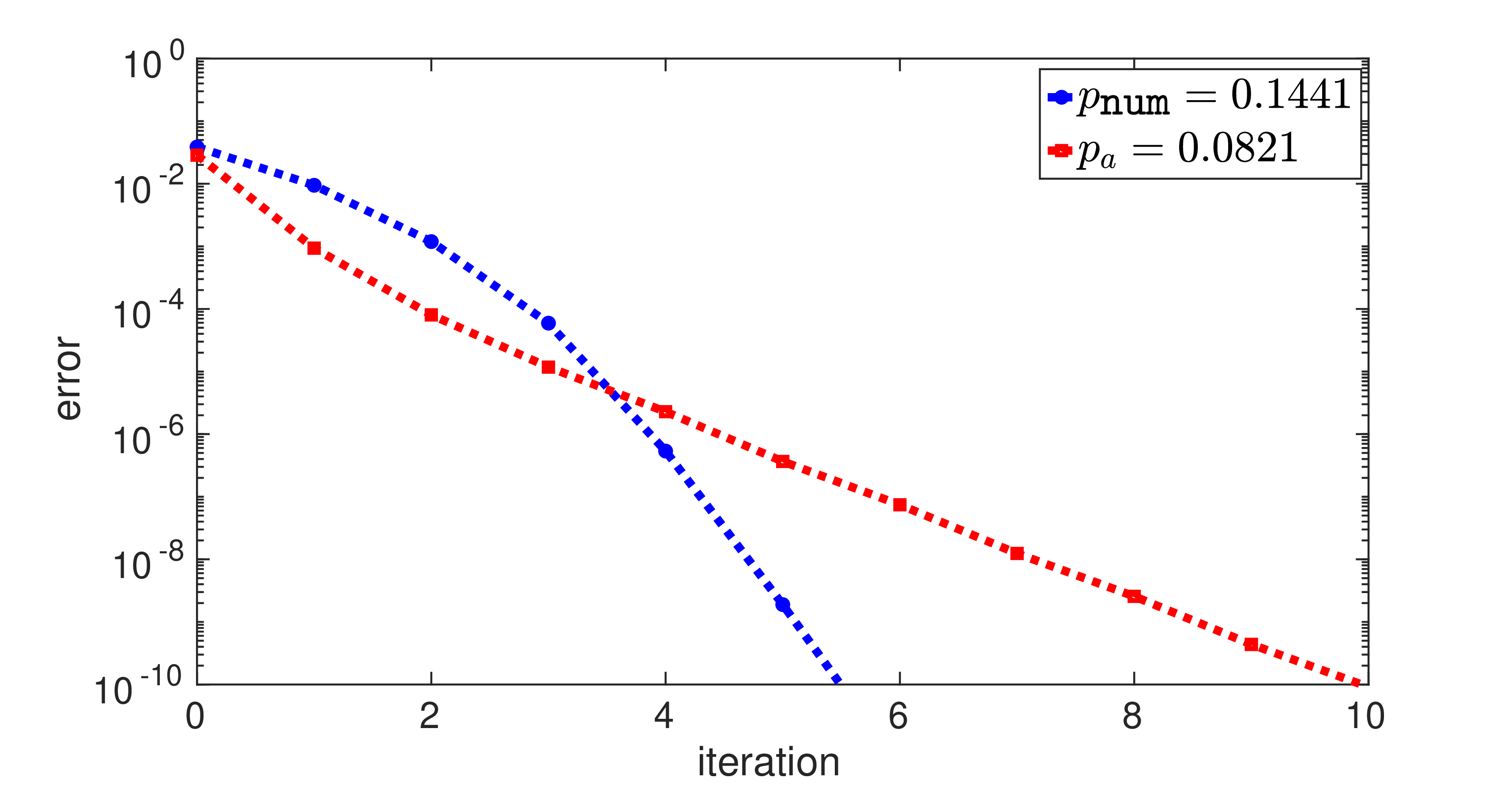} }}
    \subfloat{{\includegraphics[height=4cm,width=7cm]{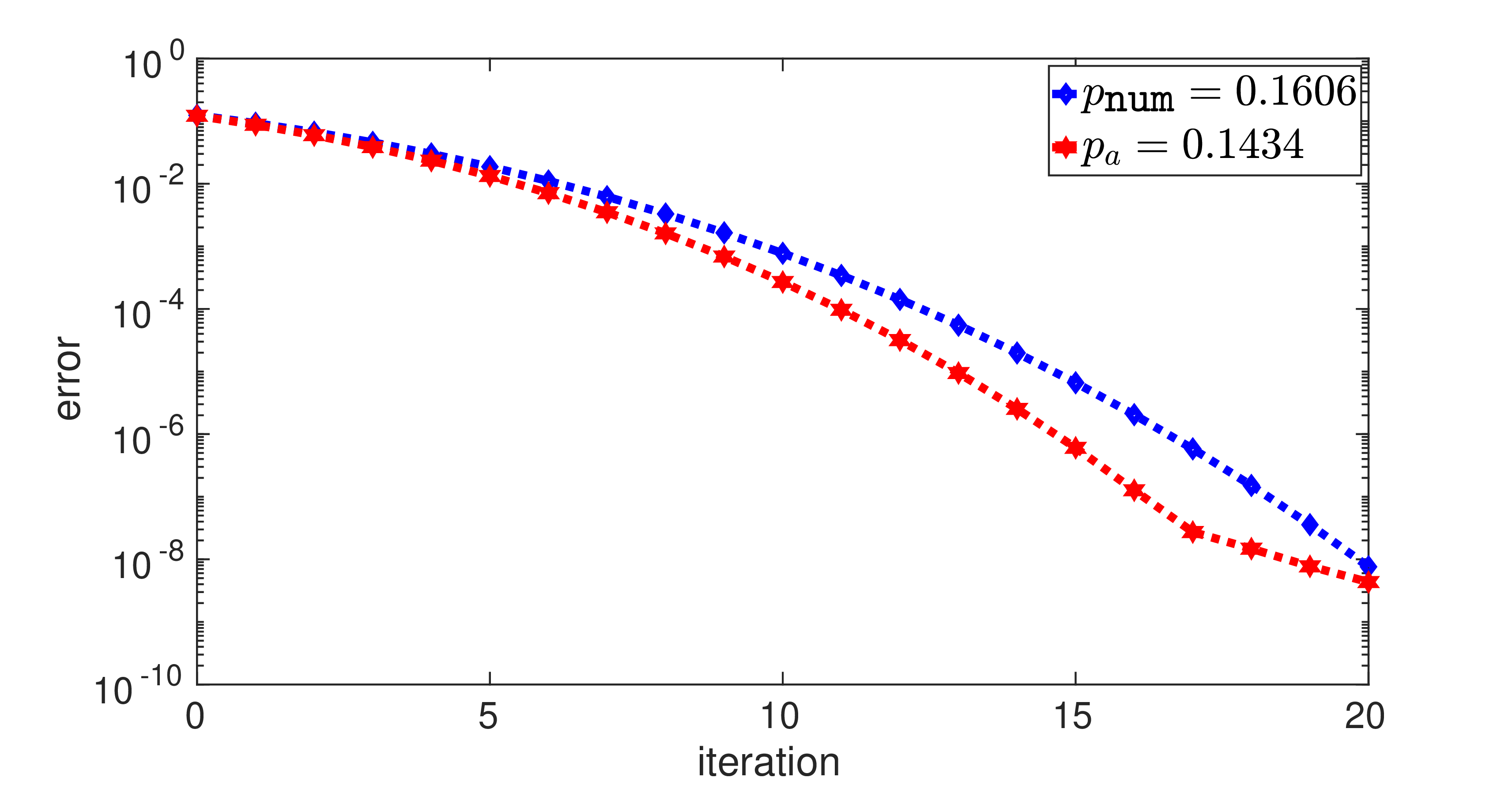} }}
    \caption{On the left: convergence of OSWR with $h=1/256, \Delta t=0.05$ and $T=0.5$; On the right: convergence of OSWR with $h=1/256, \Delta t=0.005$ and $T=5$.}
    \label{oswr_diff_p}
\end{figure}

The importance of finding the best choice of the Robin parameter can be observed from Figure \ref{num_ana_p},  which showcases the error after six OSWR iterations for the variable $u$. In the left plot, it is apparent that the most effective values of $p$ are situated within the interval $(0, 0.15)$, aligning with both our analytical and numerical determinations of $p$ depicted in the right plot.
\begin{figure}[h!]
    \centering
    \subfloat{{\includegraphics[height=4cm,width=7cm]{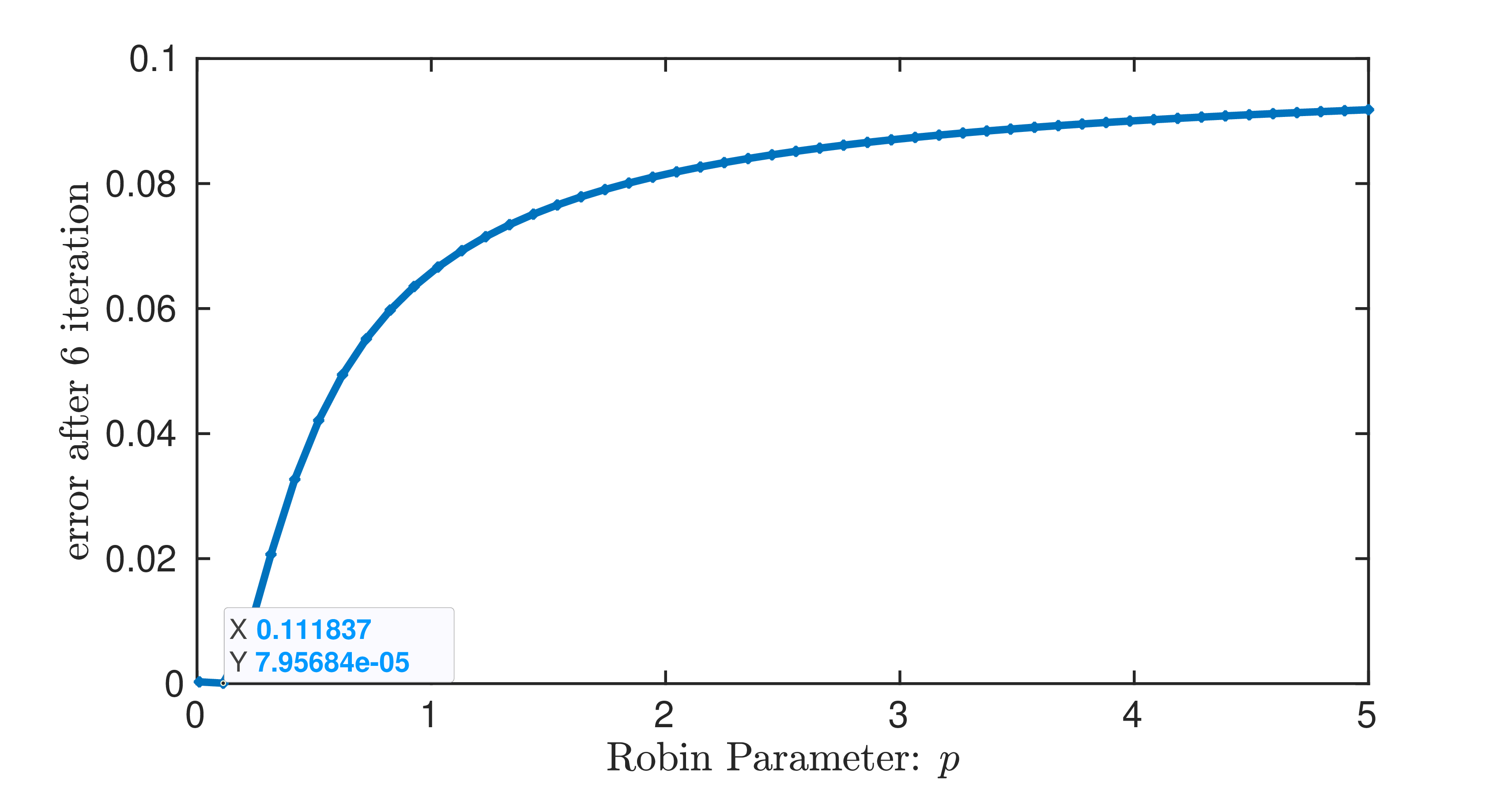} }}
    \subfloat{{\includegraphics[height=4cm,width=7cm]{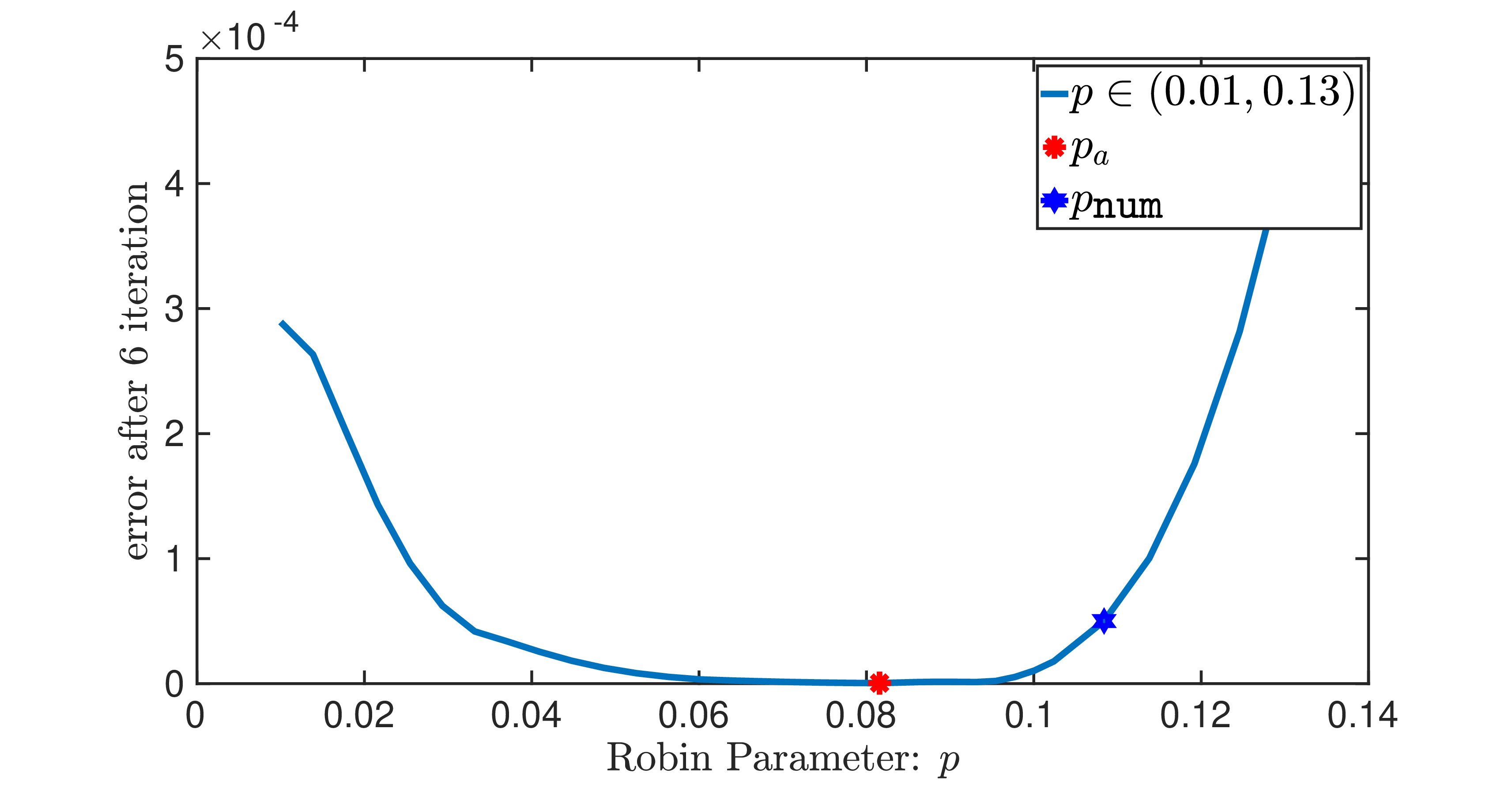} }}
    \caption{Error of OSWR after $6$ iteration with $h=1/256, \Delta t=0.05$ and $T=2$. On the left: $p\in(0.01,5)$; On the right:  $p\in(0.01,0.13)$.}
    \label{num_ana_p}
\end{figure}

\subsubsection{Experiments in 2D}
To conduct numerical experiments in two dimensions, we utilize a uniform mesh size of $h=1/64$. In the left plot of Figure \ref{oswr_p_2d}, we have illustrated the error curves for the variables $u$ and $v$. Upon examining these curves, it is evident that the convergence behavior in this 2D context closely mirrors what we observed in the one-dimensional case. This suggests a consistent performance and reliability of our method across different dimensional settings.
\begin{figure}[h!]
    \centering
    \subfloat{{\includegraphics[height=4cm,width=7cm]{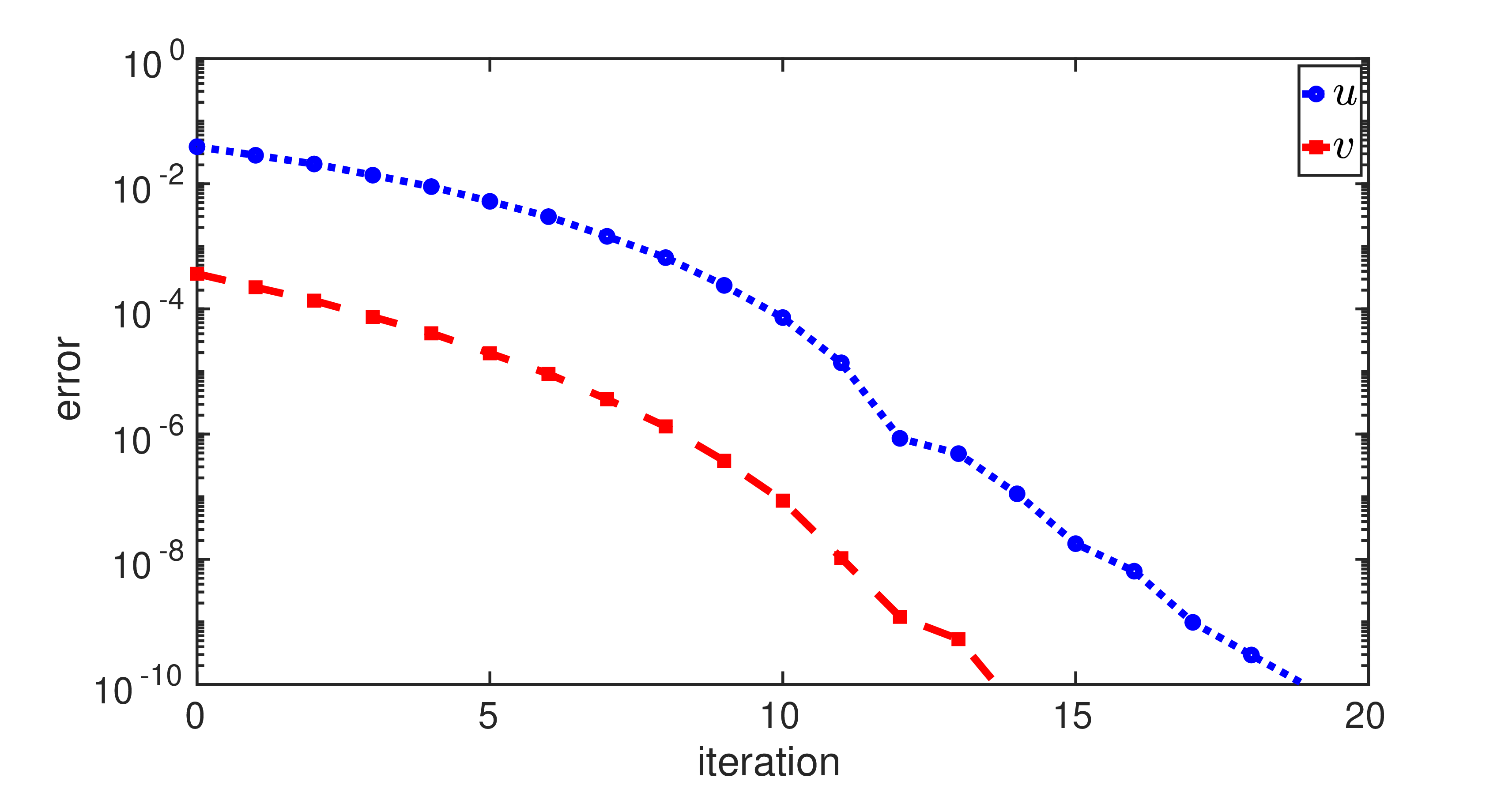} }}
    \subfloat{{\includegraphics[height=4cm,width=7cm]{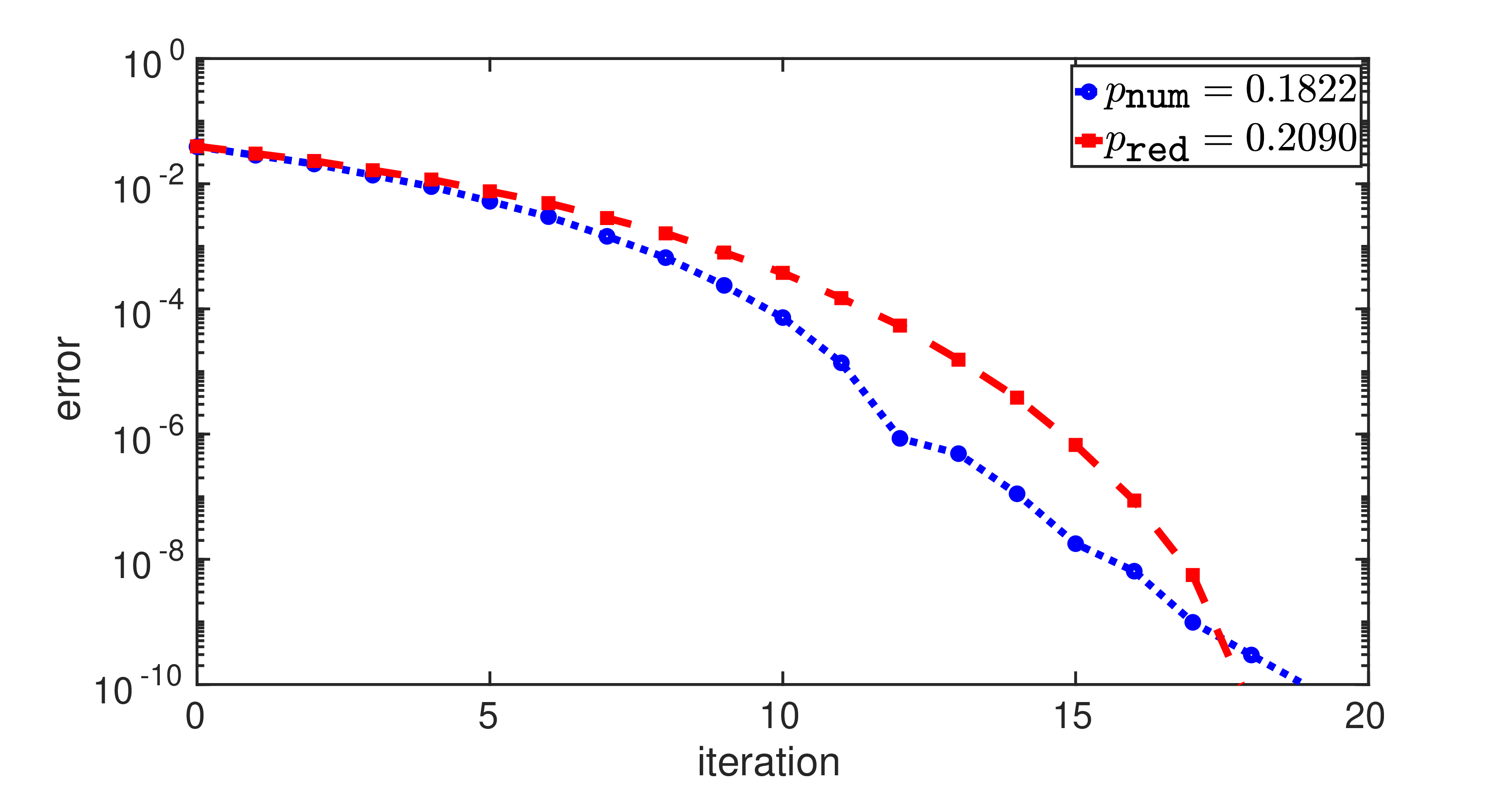} }}
    \caption{ On the left: convergence of OSWR with $h=1/64, \Delta t=0.05$ and $T=1$; On the right: convergence of OSWR with $h=1/64, \Delta t=0.05$ and $T=1$ for different $p$.}
    \label{oswr_p_2d}
\end{figure}
In the right plot of Figure \ref{oswr_p_2d}, we present the error curves for the variable $u$ corresponding to the parameters $p=p_{\texttt{num}}$ and $p=p_{\texttt{red}}$. From these curves, it is apparent that the Robin parameter $p_{\texttt{red}}$, which was derived from reduced dynamics, performs exceptionally well. This effectiveness highlights the robustness and efficiency of $p_{\texttt{red}}$ in minimizing the error, further validating the approach used to determine this parameter.

\subsection{Numerical Investigation of the Parareal-Incomplete OSWR Method} 
In this subsection, we present numerical results obtained using the Parareal-incomplete OSWR method (Algorithm \ref{alg:parareal}). We utilize equation \eqref{discrete_RM} with a coarse time step $\Delta T$ for the coarse solver. For the fine solver, we implement the OSWR method with a finer time step $\Delta t$. Specifically, we perform $K$ iterations of the OSWR method (i.e., before achieving complete OSWR convergence) and evaluate the performance across different values of $K$. For the initial Robin data, we set ${g}_{i,n}^{0,0} = {\mathtt{B}_i {U}_n^{0}}_{|\Gamma_T}$ for $i=1, 2$. The iteration process continues until the maximum error, defined as $\max\limits_{0\leq n\leq N}\parallel\mathcal{E}_{n}^l\parallel$, falls below the tolerance threshold of $10^{-6}$. The reference solution used for comparison is a discrete sequential solution using the linearized scheme \eqref{discrete_RM}.

\subsubsection{Experiments in 1D with two subdomains}
First, we decompose the domain $\Omega$ as $\Omega_1=(0, 1/2)$ and $\Omega_2=(1/2, 1)$. 
\begin{figure}[h!]
    \centering
    \subfloat{{\includegraphics[height=4cm,width=7cm]{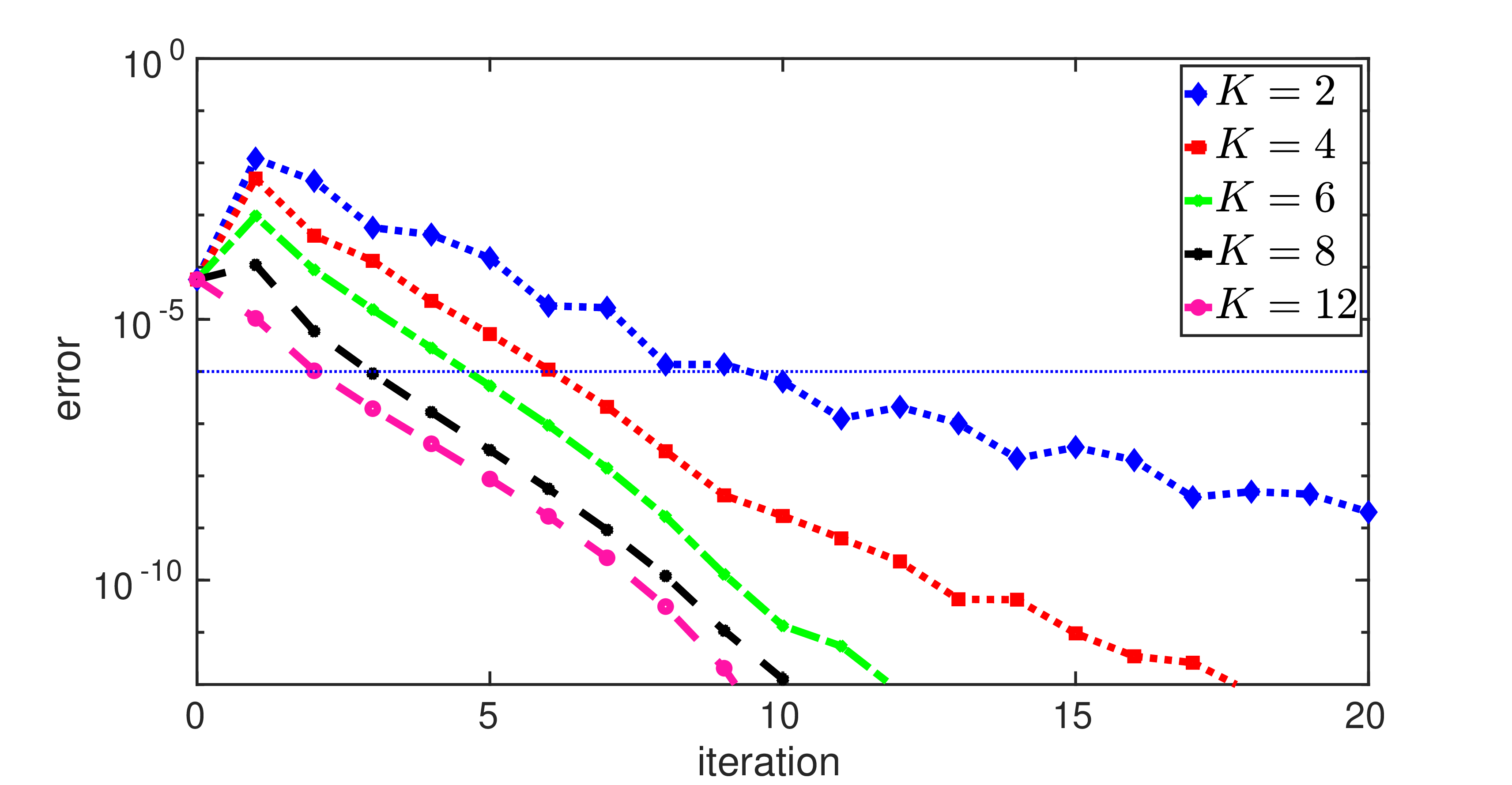} }}
    \subfloat{{\includegraphics[height=4cm,width=7cm]{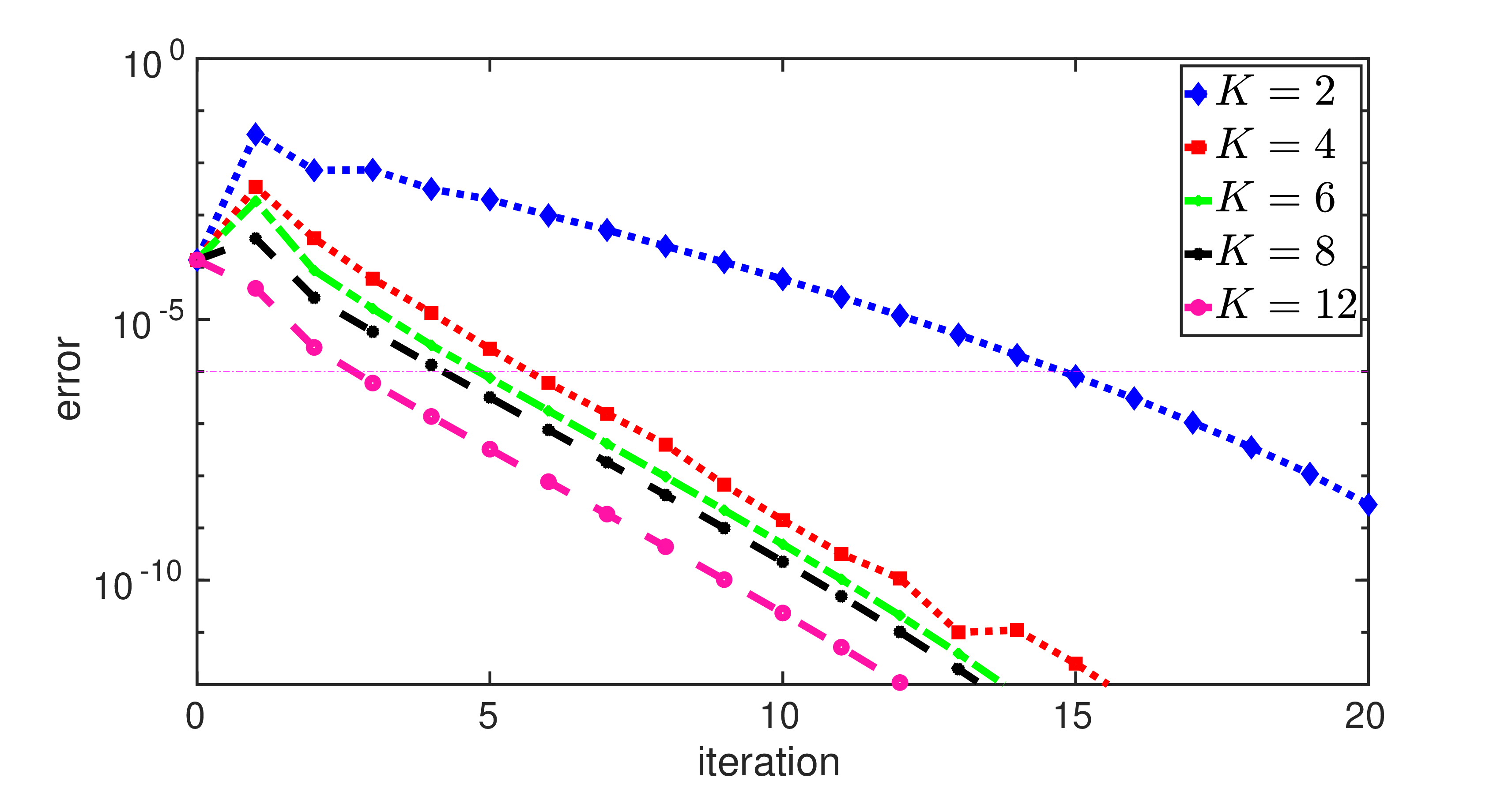} }}
    \caption{ Convergence of Parareal-OSWR for different inner iterations $K$. On the left: $T=1, N=10$; On the right: $T=2, N=20$.}
    \label{pa_oswr}
\end{figure}
In the left panel of Figure \ref{pa_oswr}, we present the error curves for the Parareal-OSWR method under various inner iteration counts $K$. The simulation parameters are set as $\Delta t=0.0067, \Delta T=0.1, h=1/256$, and the number of time slices $N=10$. Our results illustrate that increasing the number of inner iterations $K$ leads to a reduction in the number of required outer iterations. However, since inner iterations are computationally expensive, it is advisable to select $K$ from the set $\{4, 5, 6 \}$. This choice strikes a balance between computational cost and the convergence rate of the outer iterations. In the right panel of Figure \ref{pa_oswr}, we extend our analysis to a longer final time $T=2$, using $N=20$ time slices while keeping the other parameters consistent with the $T=1$ case. This comparison further highlights the performance of the Parareal-OSWR method under different temporal discretization and iteration configurations.

\begin{figure}[h!]
    \centering
    \subfloat{{\includegraphics[height=4cm,width=7cm]{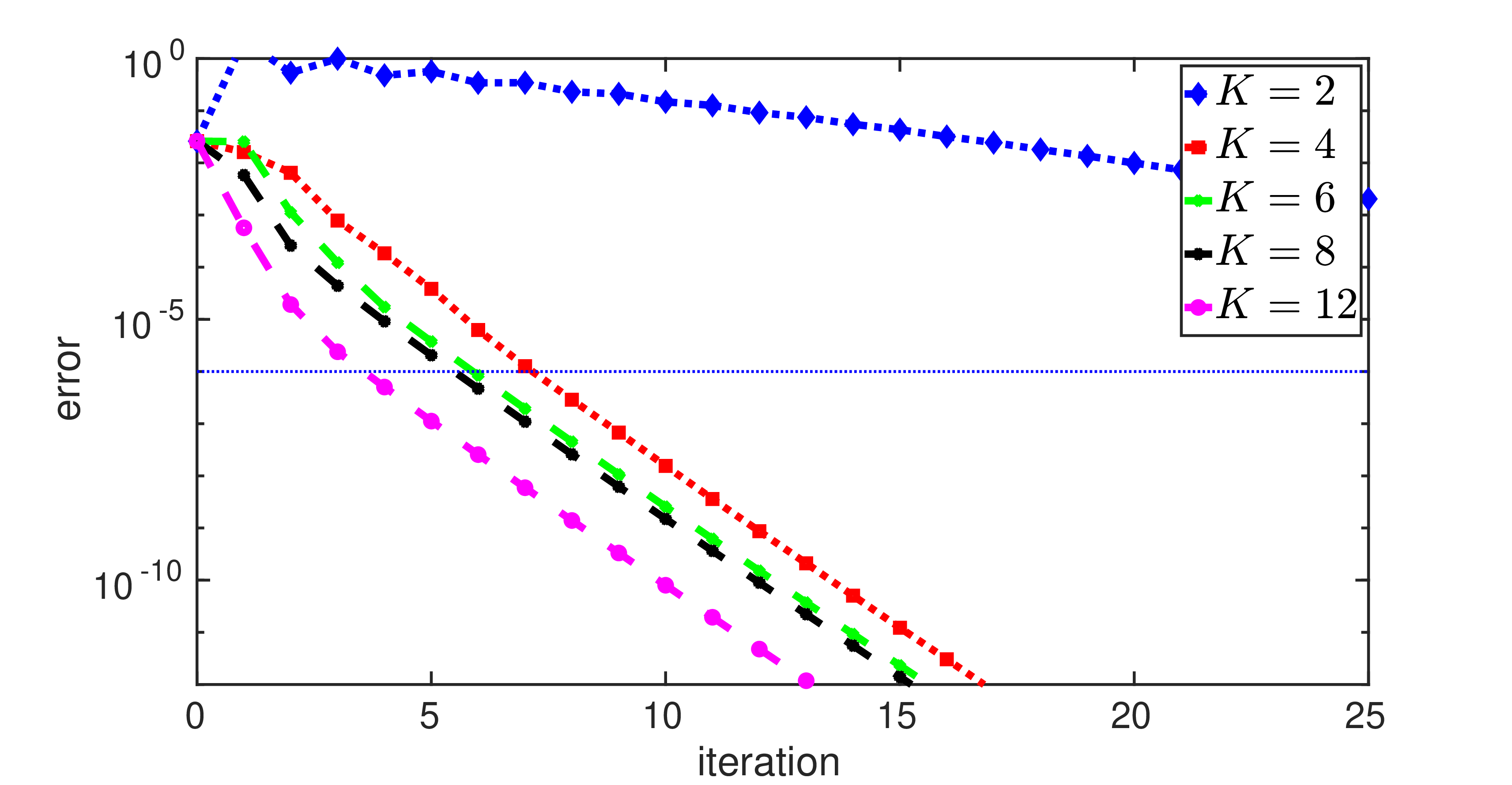} }}
    \subfloat{{\includegraphics[height=4cm,width=7cm]{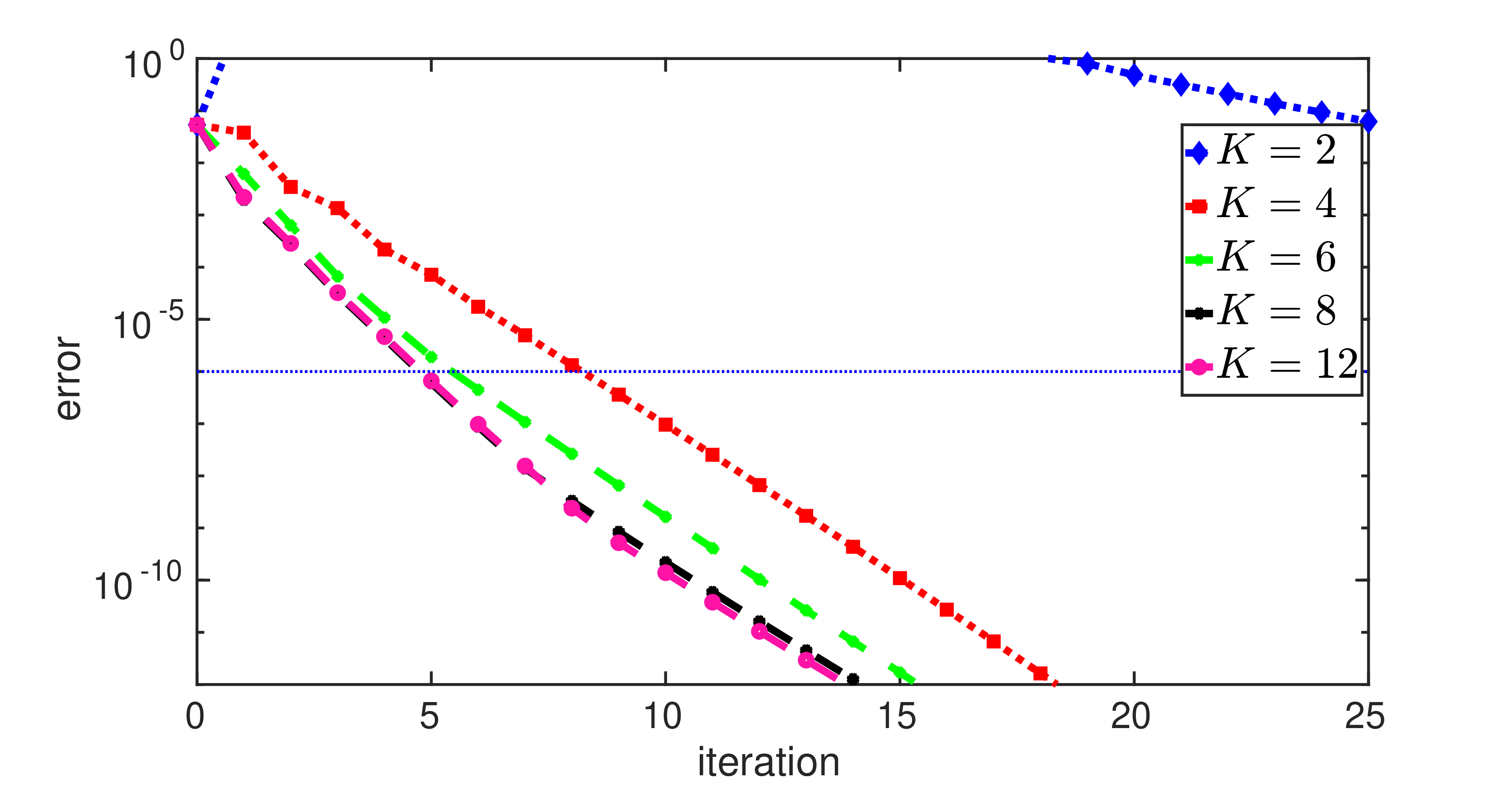} }}
    \caption{ Convergence of Parareal-OSWR for different $K$ with $h=1/256$. On the left: $T=100, \Delta T=0.25$; On the right: $T=100, \Delta T=1$.}
    \label{pa_oswr_diff_K_largeT}
\end{figure}
Next, in Figure \ref{pa_oswr_diff_K_largeT}, we plot the error curves of the Parareal-OSWR method for different values of inner iteration $K$ when applied to a large final time $T=100$. Our observations reveal that, except for $K=2$, which exhibits slow convergence, it is possible to use a larger time step $\Delta T$. Utilizing a larger $\Delta T$ results in a reduced number of time slices $N$, thereby decreasing the computational resources required to achieve the same level of accuracy. This demonstrates the efficiency of the method, especially for long-time simulations, as it allows for fewer processor cycles while maintaining convergence performance.

\subsubsection{Experiments in 1D with two subdomains for various initial data}
Now, we test the Parareal-OSWR with the following diverse initial data. Consider
\begin{equation}\label{diff_initial_data}
\begin{aligned}
u_0=0.1e^{-30x^2}, v_0=0\; \text{with smaller $L^2$ norm in}\; \Omega=(0, 1),\\
u_0=0.1\sin(4\pi x), v_0=0\; \text{having low frequency mode in}\; \Omega=(0, 1).\\
u_0=0.1\sin(30\pi x), v_0=0\; \text{having high frequency mode in}\; \Omega=(0, 1).
\end{aligned}
\end{equation} 
For the specified initial data collection, we generate error curves depicted in Figure \ref{pa_oswr_small_data} by systematically altering the inner iteration count $K$, while taking $h=1/256, \Delta T=0.1$ and $T=4$. The plots illustrate the method's performance across different inner iteration settings. Surprisingly, the method demonstrates robustness even with a minimal inner iteration count of $K=2$, as evidenced by the error curves. This suggests the efficacy of the Parareal-OSWR method in handling diverse initial data sets while maintaining convergence, even under relatively modest computational settings.
\begin{figure}[h!]
    \centering
    \subfloat{{\includegraphics[height=4cm,width=5cm]{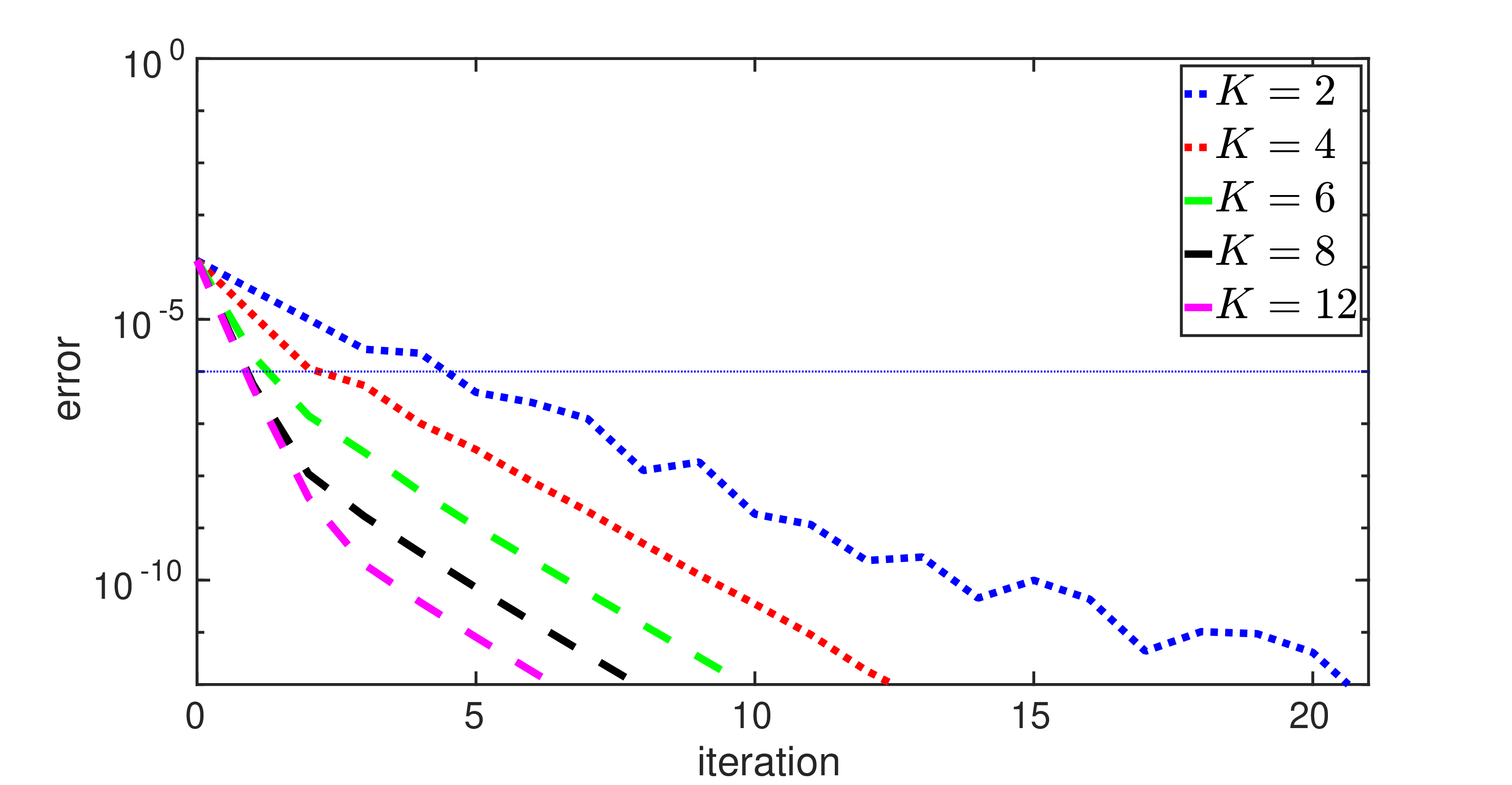} }}
    \subfloat{{\includegraphics[height=4cm,width=5cm]{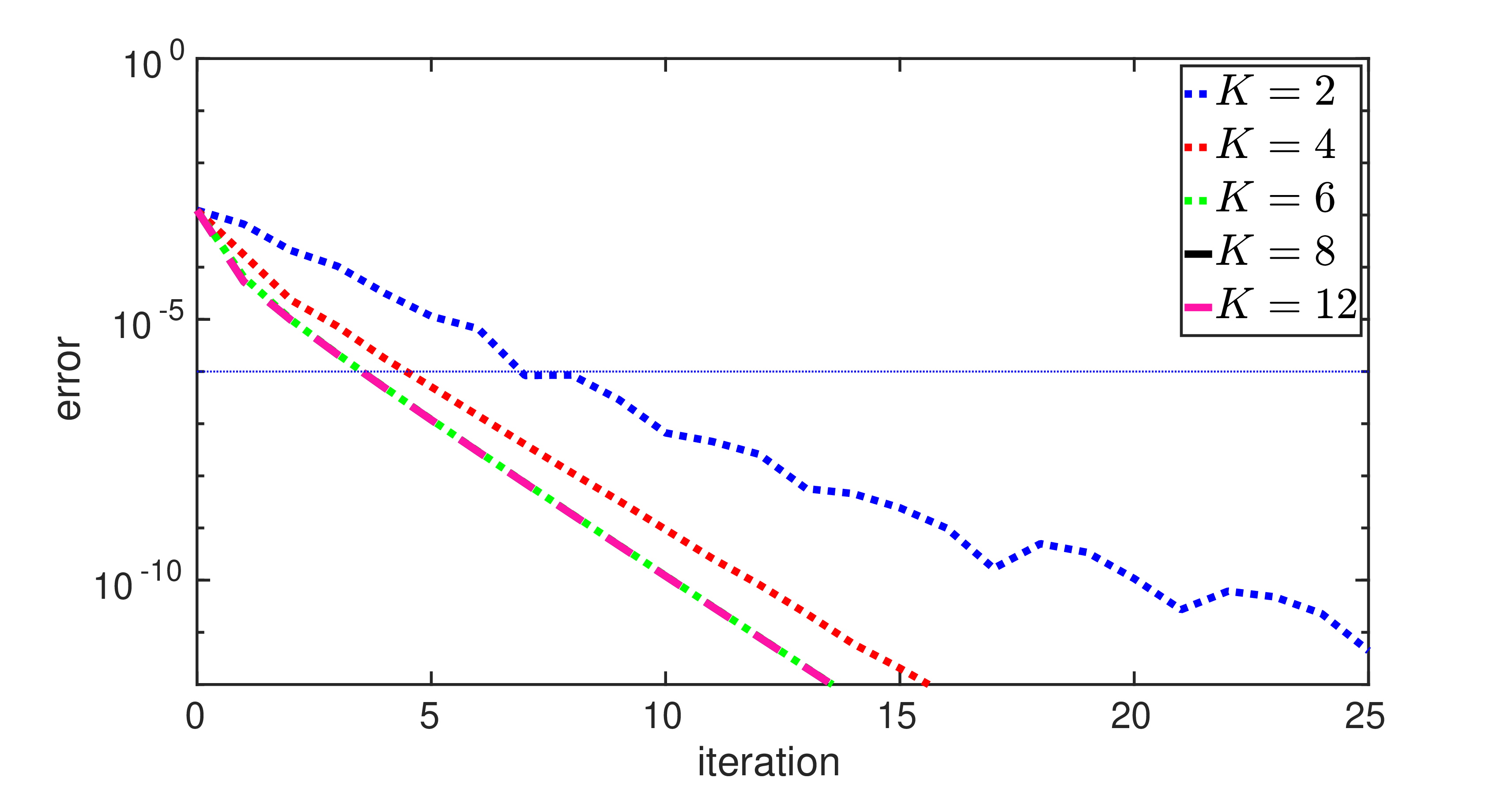} }}
    \subfloat{{\includegraphics[height=4cm,width=5cm]{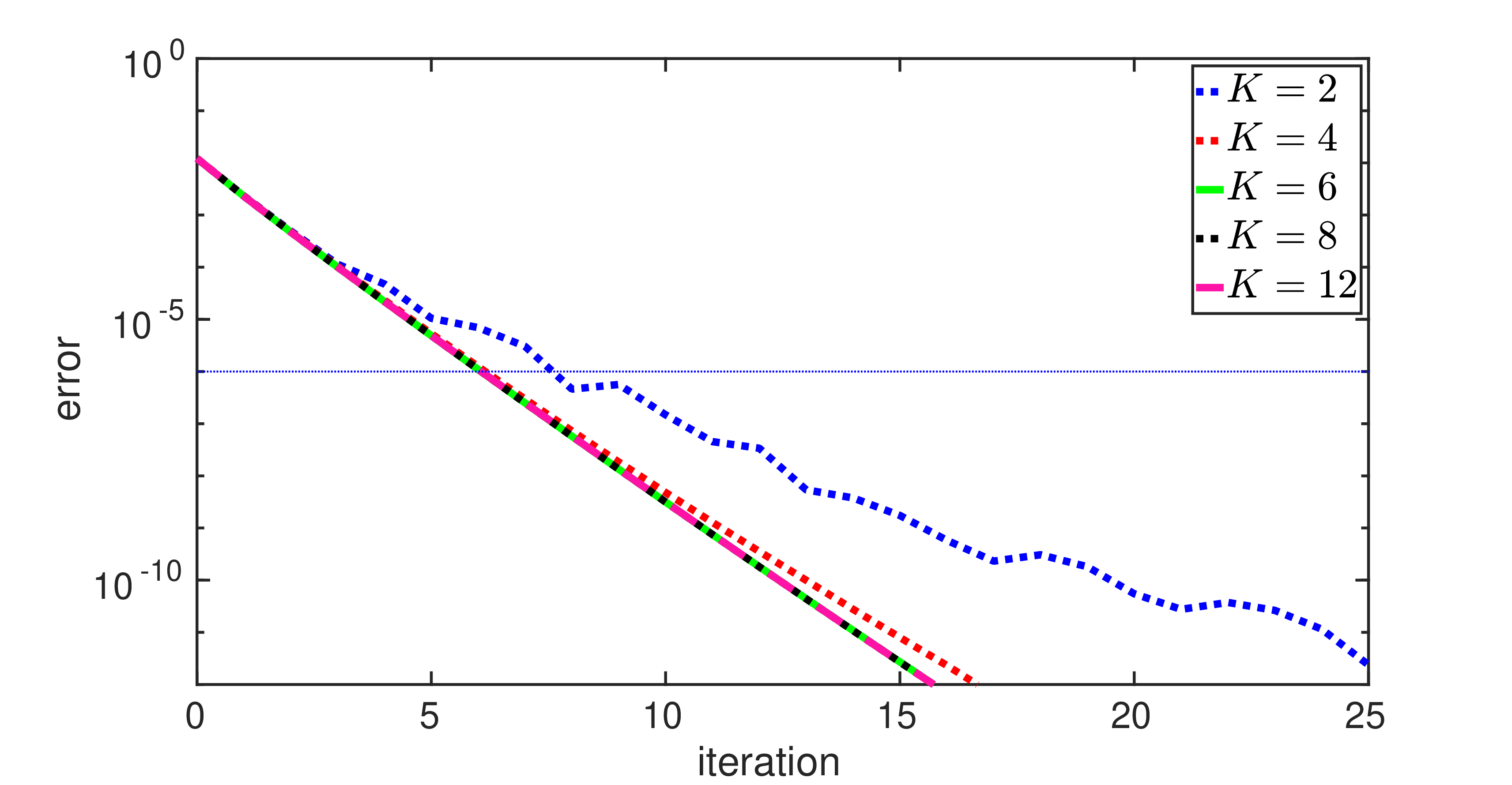} }}
    \caption{ Convergence of Parareal-OSWR for different initial data. On the left: \eqref{diff_initial_data}$_1$; On the middle:  \eqref{diff_initial_data}$_2$; On the right:  \eqref{diff_initial_data}$_3$.}
    \label{pa_oswr_small_data}
\end{figure}

\subsubsection{Experiments in 1D for more than two subdomain}
In this subsection, we explore numerical outcomes pertaining to spatial domain decompositions involving more than two subdomains. Let $N_s$ represent the number of spatial subdomains considered. We partition the spatial domain $\Omega=(0, 10)$ into $N_s$ subdomains of equal length. In the left panel of Figure \ref{pa_oswr_multi_sub}, we present error curves generated by the Parareal-OSWR method, with $N_s=64, N=10$, and varying inner OSWR iterations $K$. The plots illustrate that the convergence remains robust and mirrors the behavior observed in the case of two subdomain decompositions.
\begin{figure}[h!]
    \centering
    \subfloat{{\includegraphics[height=4cm,width=7cm]{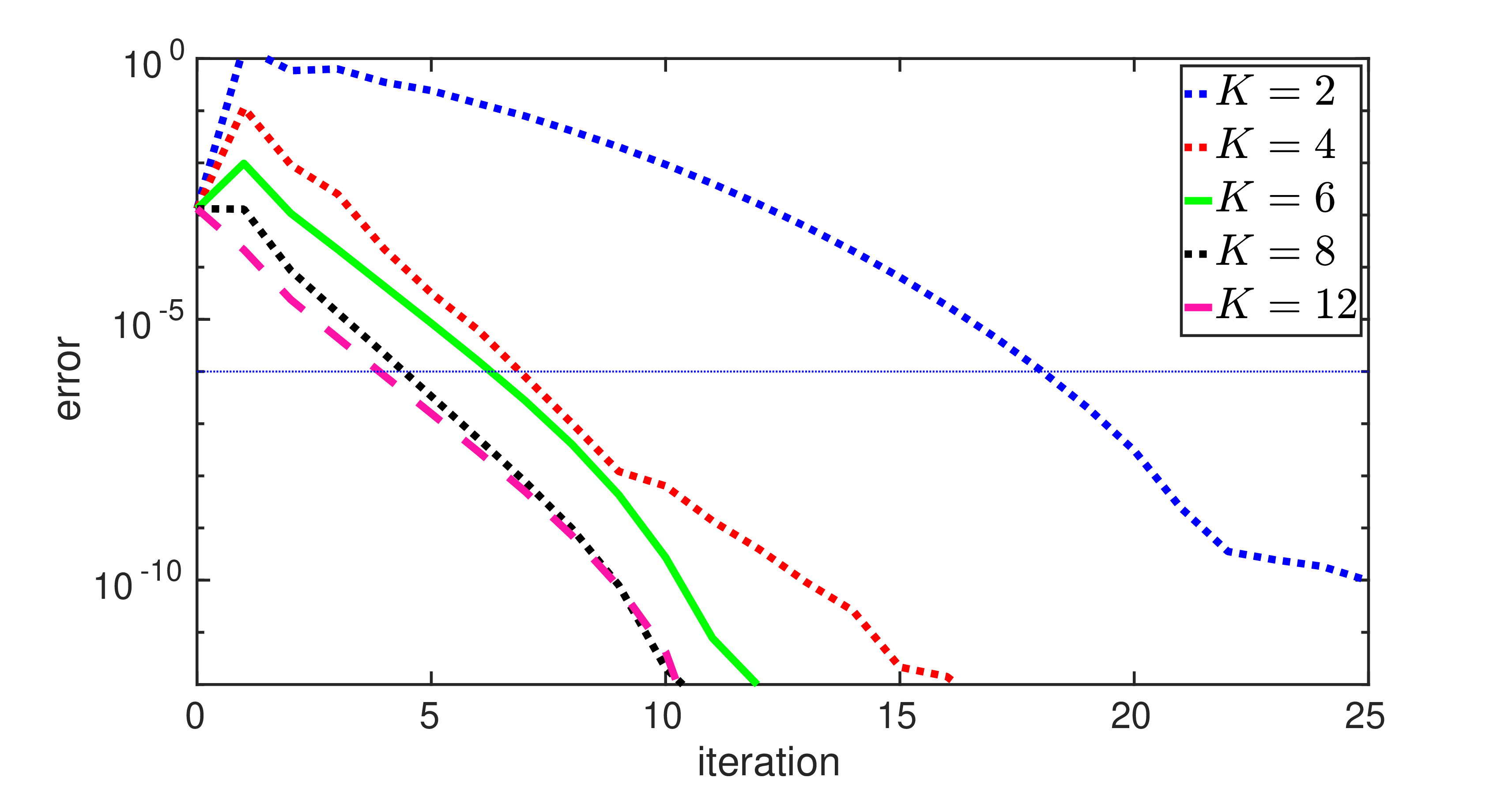} }}
    \subfloat{{\includegraphics[height=4cm,width=7cm]{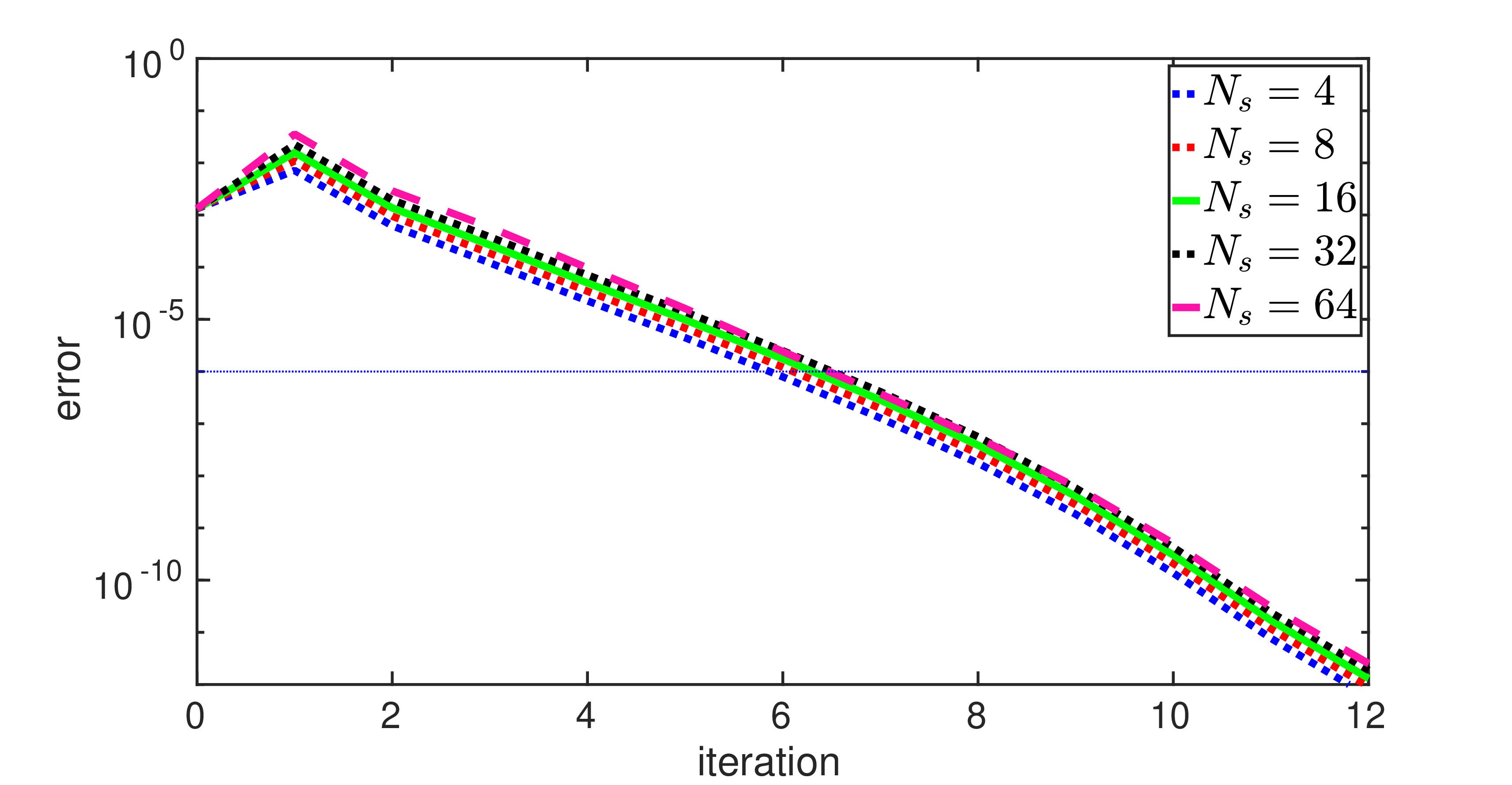} }}
    \caption{ Convergence of Parareal-OSWR with $N=10, T=5, h=0.0098$. On the left: $N_s=64$ and different $K$; On the right: $K=5$ and different $N_s$.}
    \label{pa_oswr_multi_sub}
\end{figure}
On the right panel of Figure \ref{pa_oswr_multi_sub}, we explore the effect of varying $N_s$ while keeping the inner OSWR iteration count fixed at $K=5$. Remarkably, the convergence remains consistent across different numbers of subdomains, indicating the method's stability and adaptability to various spatial decompositions.

\subsubsection{Experiments in 2D}
In our exploration into 2D experiments, we employ the initial data specified in Eq~\eqref{initial-2d}. This data provides the foundation for our investigation into the Parareal-OSWR method's performance in a 2D setting. In the left panel of Figure \ref{pa_oswr_smooth_data_2d}, we visualize the error curves resulting from the application of the Parareal-OSWR method, with $K$ being systematically varied. Notably, the observed convergence behavior bears similarity to that observed in the one-dimensional experiments. This consistency in convergence behavior across dimensions underscores the method's robustness and reliability.
\begin{figure}[h!]
    \centering
    \subfloat{{\includegraphics[height=4cm,width=7cm]{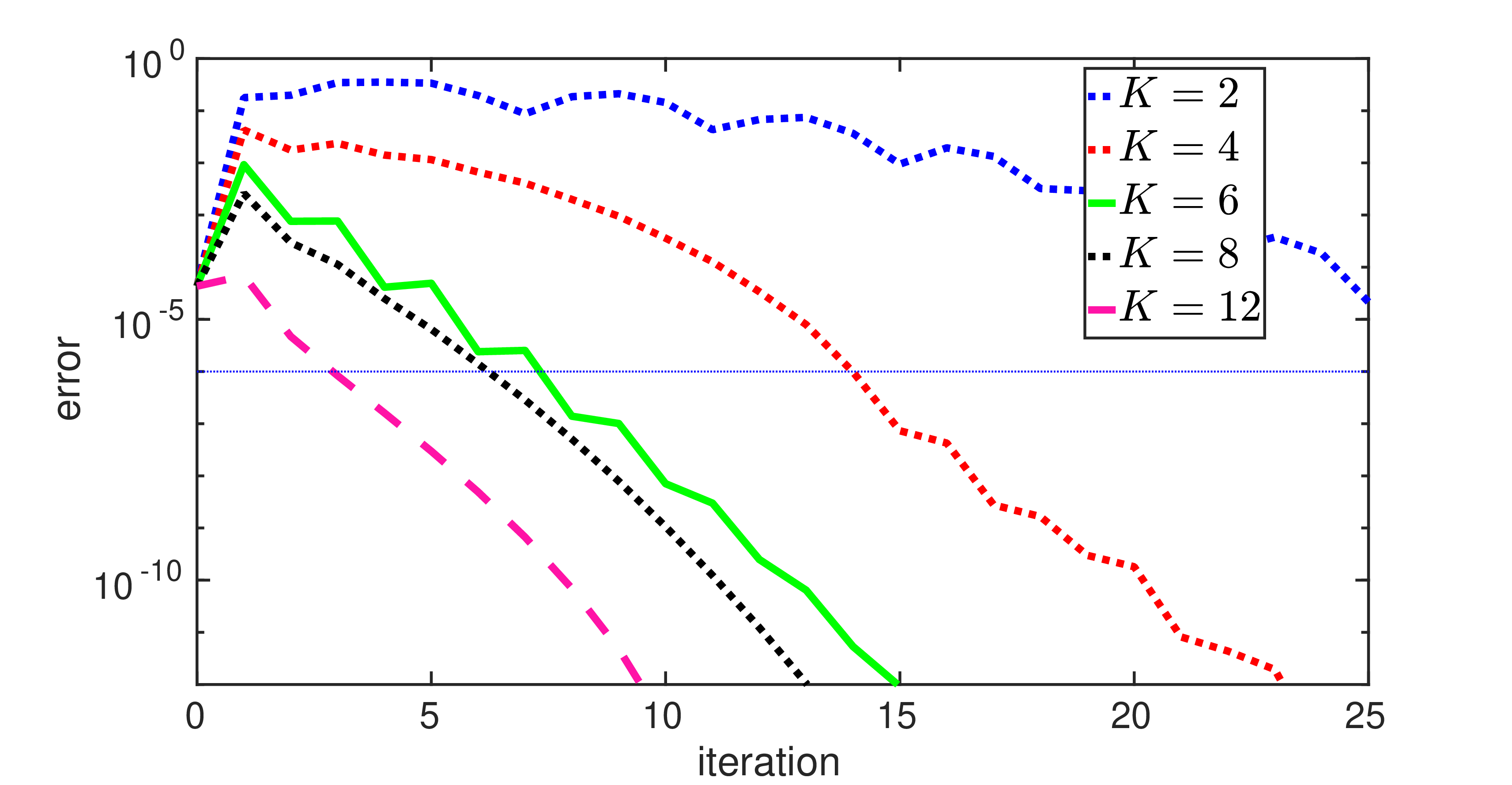} }}
     \subfloat{{\includegraphics[height=4cm,width=7cm]{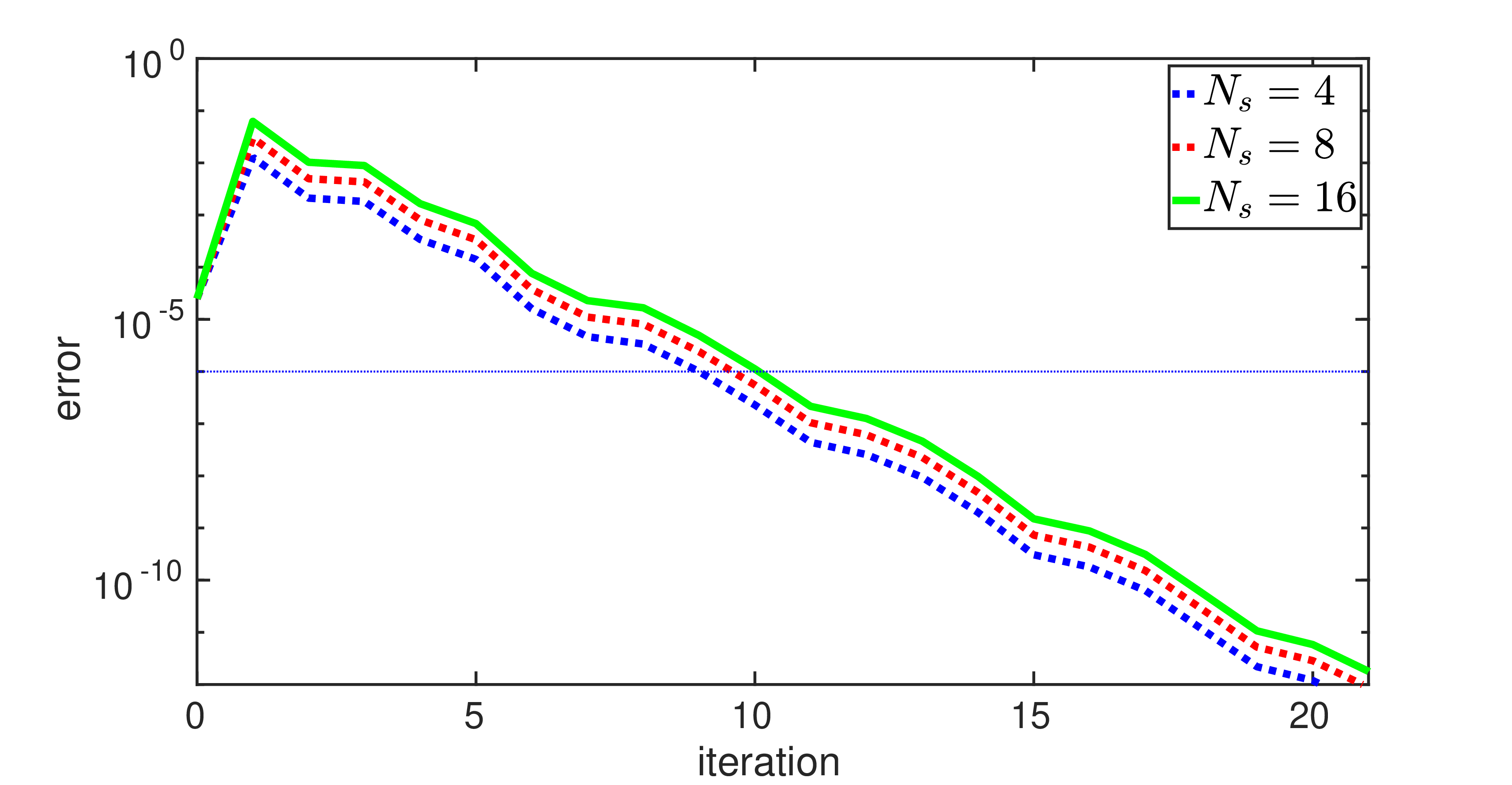}  }}
    \caption{ On the left: convergence of Parareal-OSWR in 2D with $T=1, N=10, \Delta t=0.0067, h=1/64$. $h=1/64$; On the right: Convergence of Parareal-OSWR in 2D for multiple subdomains with $T=1, N=10, h=0.03125, \Delta t=0.0067, K=6$.}
    \label{pa_oswr_smooth_data_2d}
\end{figure}
In our study of numerical simulations involving more than two subdomains in a two-dimensional setting, we define the spatial domain as $\Omega=(0, 4)^2$. Here, we decompose the domain into $N_s$ strip-type subdomains of equal width. In the right panel of Figure~\ref{pa_oswr_smooth_data_2d}, we present the error curves resulting from these experiments for various $N_s$ values, while keeping the inner iteration count fixed at $K=6$. Notably, the plots illustrate a robust convergence behavior across different numbers of subdomains. This consistency in convergence highlights the method's reliability and its ability to maintain effectiveness even in multi-subdomain configurations in two-dimensional simulations.
\subsubsection{Experiments with a non-smooth initial guess} 
For 1D, we use the following initial solution 
\begin{equation}\label{nonsmooth_initial}
u_0 (x) = \begin{cases}
0.3 &\text{for $x\in (0.4, 0.6),$}\\
0 &\text{elsewhere},
\end{cases}
\qquad
\text{ and } \quad v_0=0,
\end{equation}
in $\Omega=(0, 1)$, which corresponds to a sudden shift in the cardiac action potential and then returning to its resting state. We partition the domain at $x=1/2$. In the left and middle panel of Figure~\ref{pa_oswr_nonsmooth_data}, we present the error curves for the Parareal-OSWR method applied to non-smooth initial data as specified in Eq.~\eqref{nonsmooth_initial}, using $h=1/256$ and varying the inner iteration count $K$.
\begin{figure}[h!]
    \centering
    \subfloat{{\includegraphics[height=3.5cm,width=5cm]{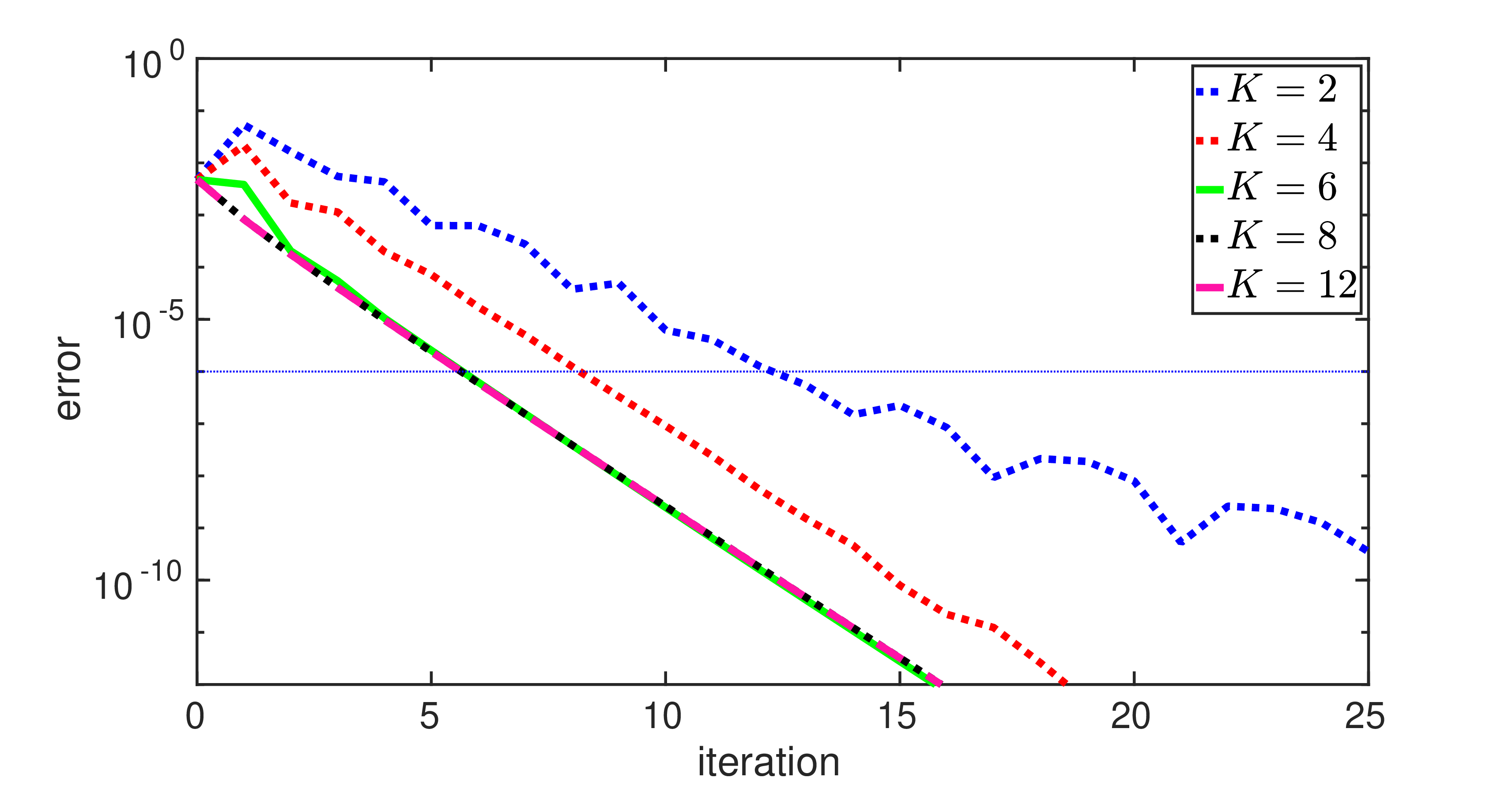} }}
    \subfloat{{\includegraphics[height=3.5cm,width=5cm]{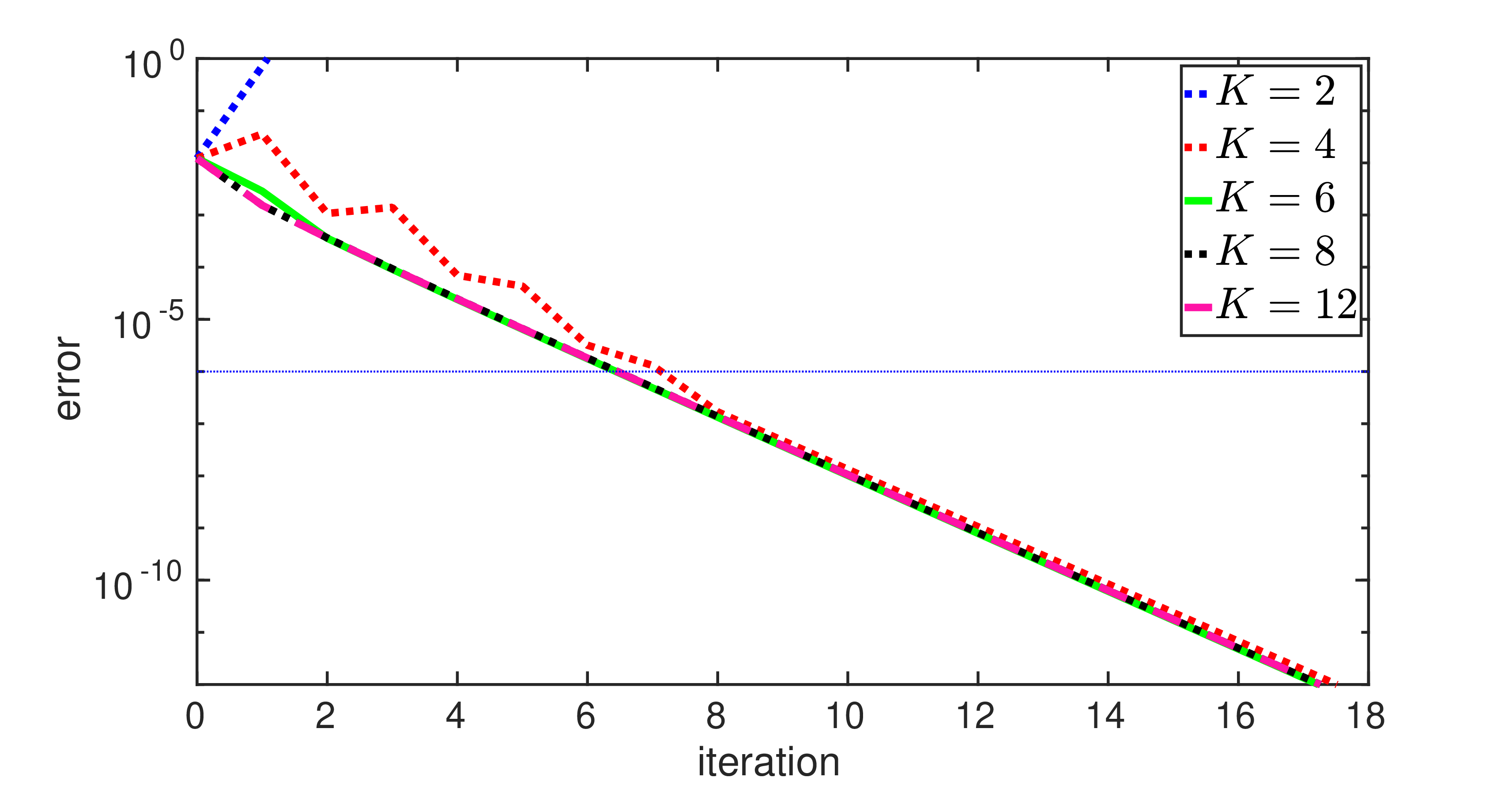} }}
    \subfloat{{\includegraphics[height=3.5cm,width=5cm]{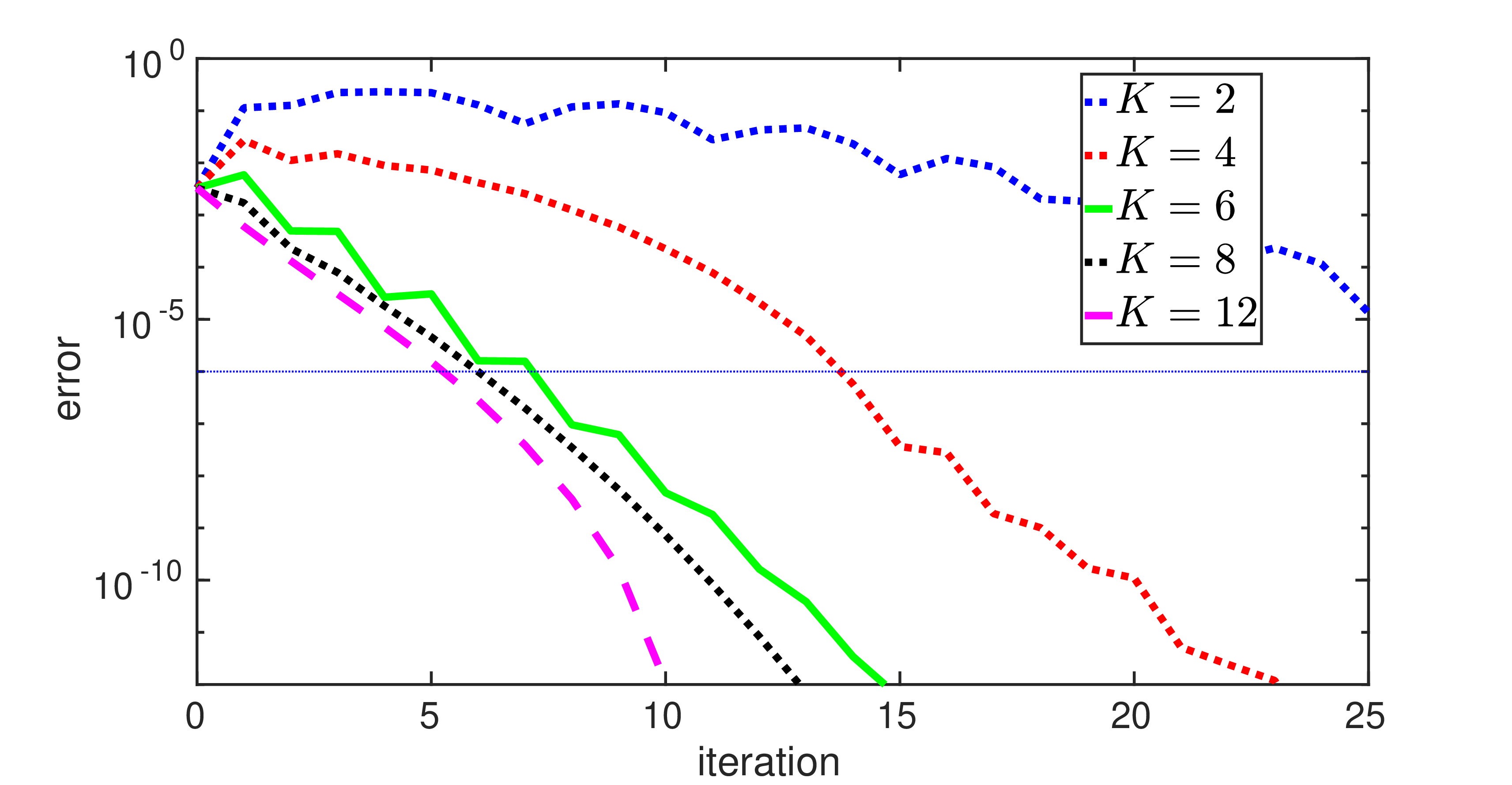} }}
    \caption{ Convergence of Parareal-OSWR. On the left: $T=4, N=40$ in 1D; On the middle: $T=600, N=600$ in 1D; On the right: $T=1, N=10, h=1/64$ in 2D.}
    \label{pa_oswr_nonsmooth_data}
\end{figure}
The results indicate that the method performs effectively over very large time windows, with the exception of the case where $K=2$. This demonstrates the robustness of the Parareal-OSWR method in handling non-smooth initial conditions, provided an adequate number of inner iterations are used.

\noindent For 2D, we consider the following initial guess 
\begin{equation}\label{nonsmooth_initial_2d}
u_0 = \begin{cases}
0.3 &\text{for $(x,y)\in (0.4, 0.6)^2,$}\\
0 &\text{elsewhere},
\end{cases}
\qquad
\text{ and } \quad v_0=0,
\end{equation}
in $\Omega=(0, 1)^2$, with $\Omega_1=(0, 1/2)\times(0, 1)$ and $\Omega_2=(1/2, 1)\times(0, 1)$. 
In the right panel of Figure \ref{pa_oswr_nonsmooth_data}, we illustrate the convergence behavior of the Parareal-OSWR method across different values of $K$. The displayed results indicate a clear trend that the method's performance improves significantly with an increased number of inner iterations $K$. 
\subsubsection{Verification of Solution}
In this section, we present the numerical solution of our model. Figure \ref{solution_profile} presents the numerical solution in 1D corresponding to the initial condition \eqref{nonsmooth_initial}, where the traveling wave behavior of the solution is clearly observed. The proposed method accurately captures this wave propagation.
\begin{figure}[h!]
    \centering
    \subfloat{{\includegraphics[height=3cm,width=7cm]{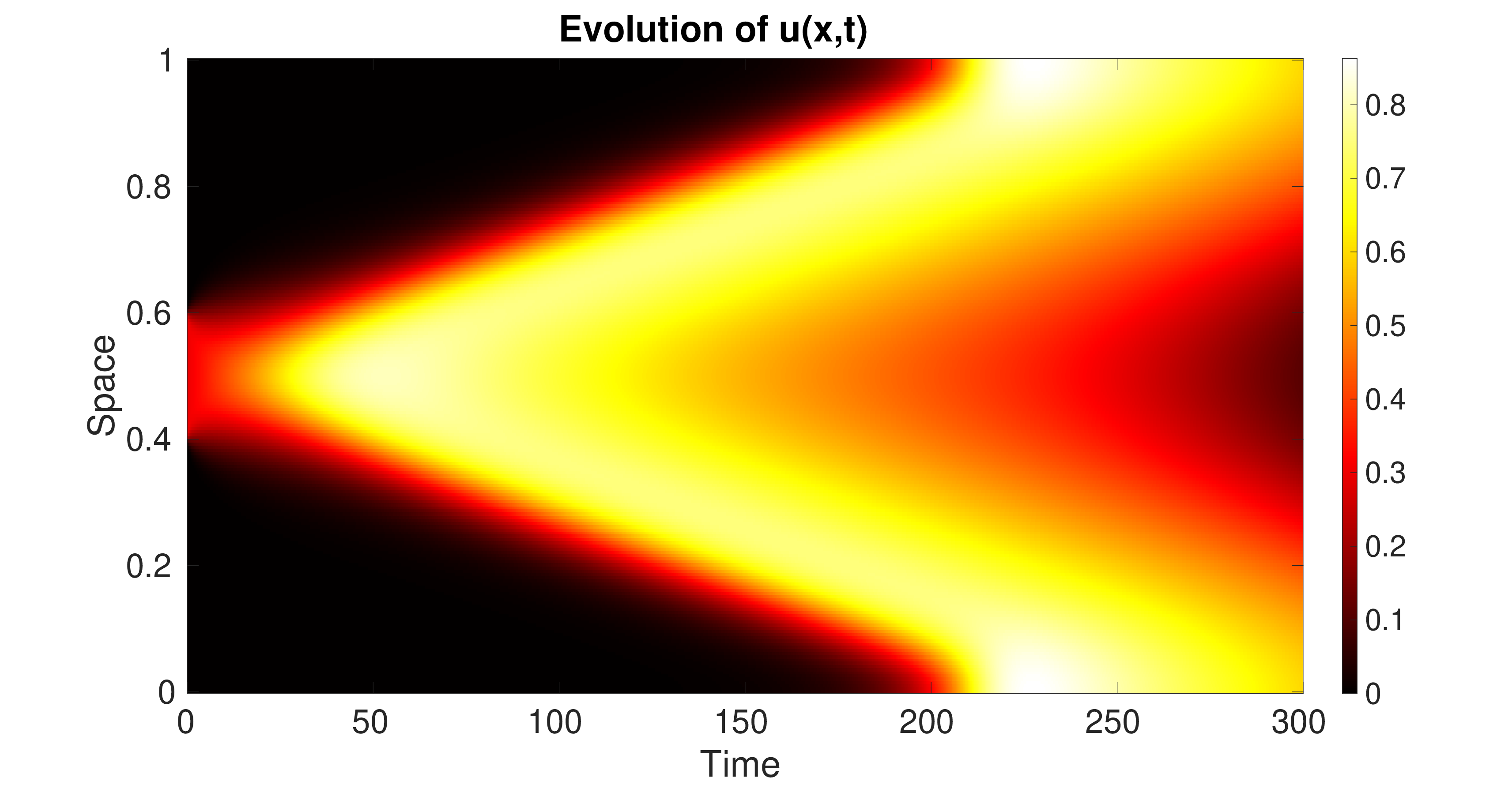} }}
    \subfloat{{\includegraphics[height=3cm,width=7cm]{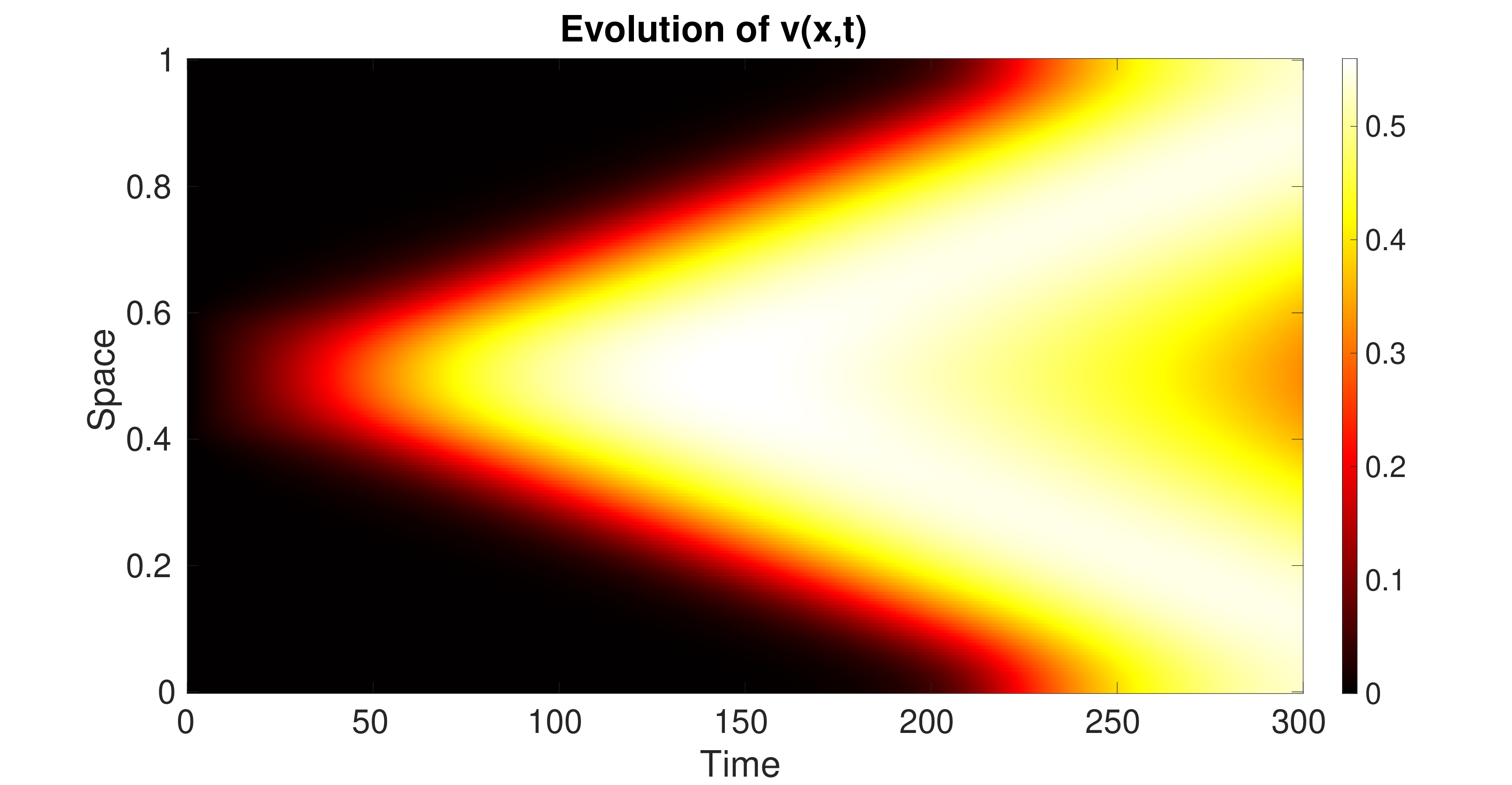} }}
    \caption{ On the left: Evolution of $u(x,t)$ up-to $T=300$ with $\gamma=0.0001$;  On the right: Evolution of $v(x,t)$ up-to $T=300$ with $\gamma=0.0001$.}
    \label{solution_profile}
\end{figure}
In 2D, we define the initial conditions in \(\Omega = (0, 10)^2\) using the distance functions  
\(r_1(x, y) = \sqrt{(x - 7.5)^2 + (y - 5)^2}\),  
\(r_2(x, y) = \sqrt{(x - 5)^2 + (y - 7.5)^2}\),  
as follows:
\[
\begin{aligned}
u_0  &= 
\begin{cases}
0.8, & r_1 \geq 5, \\
0.8\left(1 - \exp\left(0.04 - \dfrac{1}{(r_1 - 5)^2}\right)\right), & r_1 < 5,
\end{cases}\;\;
v_0 = 
\begin{cases}
1, & r_2 \geq 5, \\
1 - \exp\left(0.04 - \dfrac{1}{(r_2 - 5)^2}\right), & r_2 < 5.
\end{cases}
\end{aligned}
\]
Figure \ref{fig:all_plots} illustrates the formation of spiral waves or traveling wavefronts in the computed solution. This demonstrates that the proposed Parareal-OSWR approach accurately captures the complex spatiotemporal dynamics of the RM model.
\begin{figure}[htbp]
    \centering
   \includegraphics[scale=0.12]{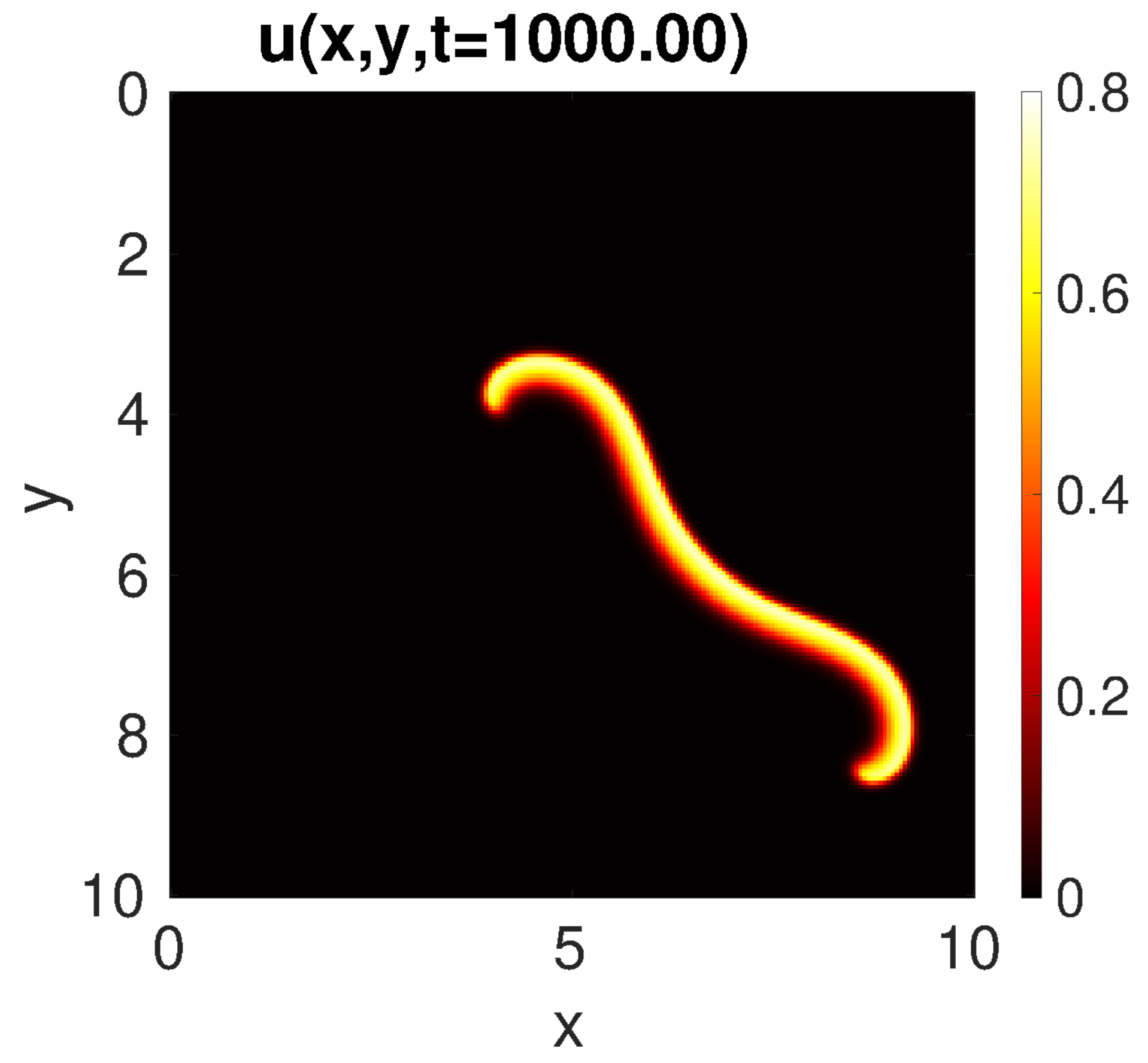}%
    \includegraphics[scale=0.12]{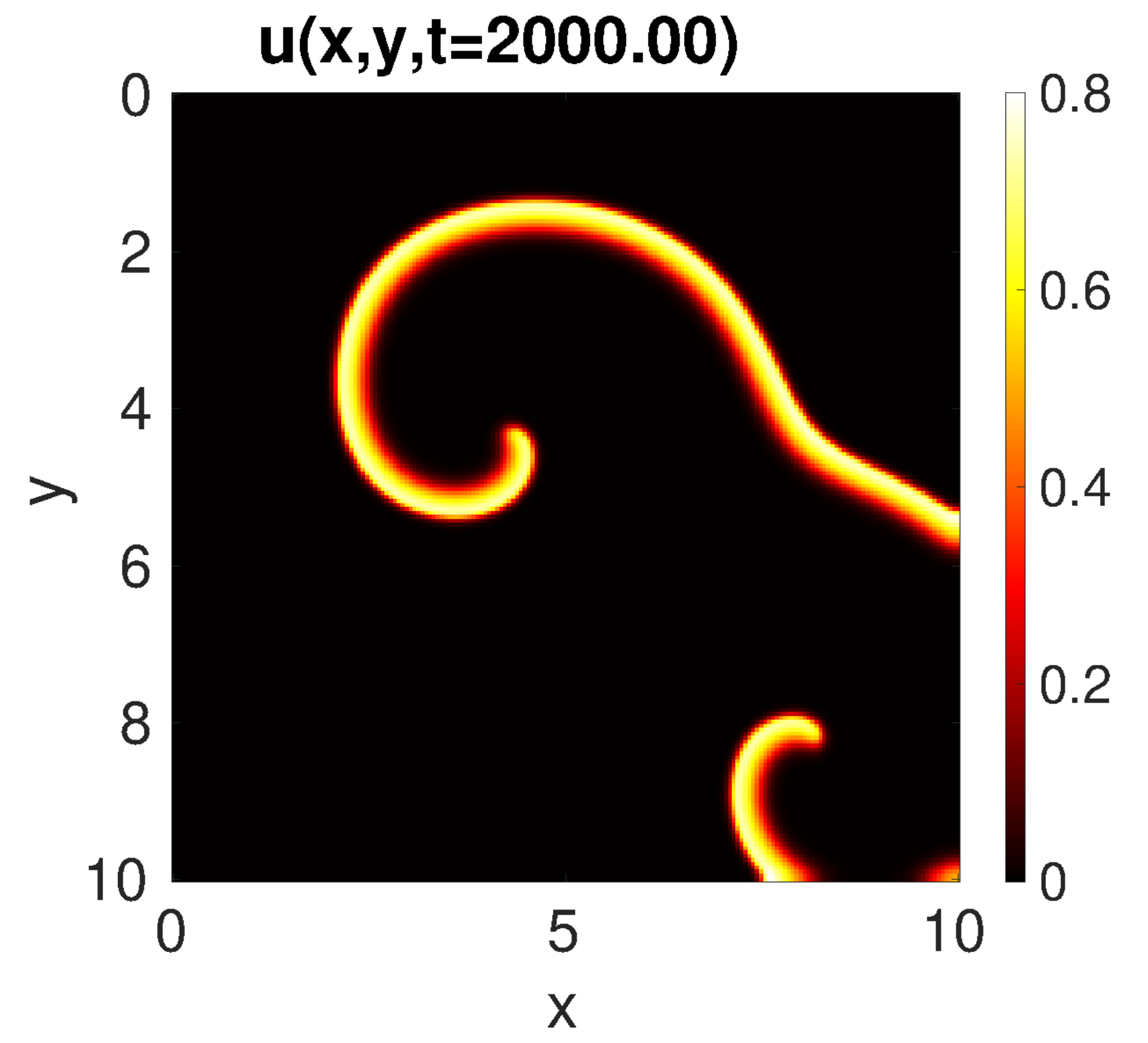}%
     \includegraphics[scale=0.12]{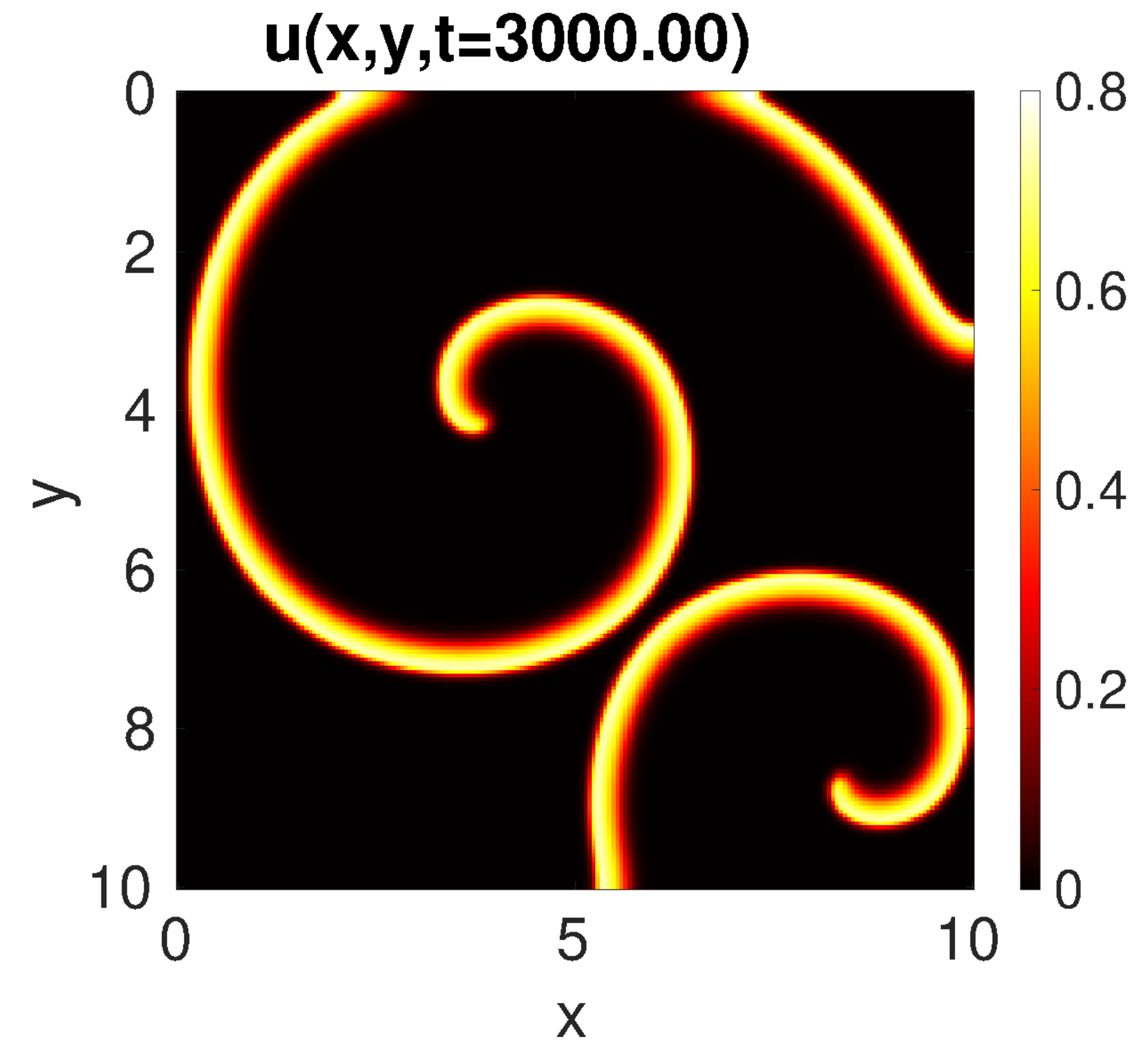}%
    \includegraphics[scale=0.12]{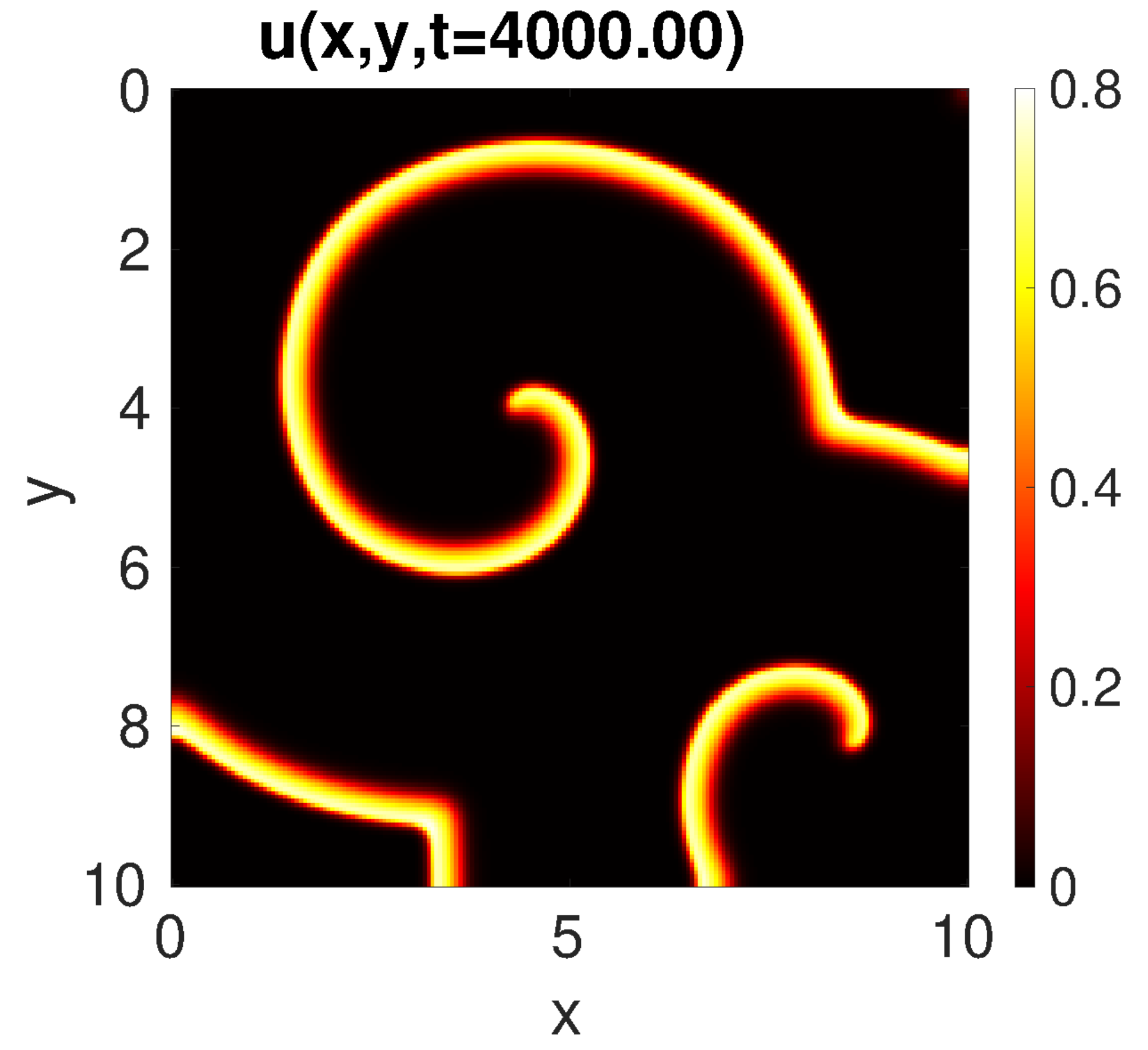}

    \includegraphics[scale=0.12]{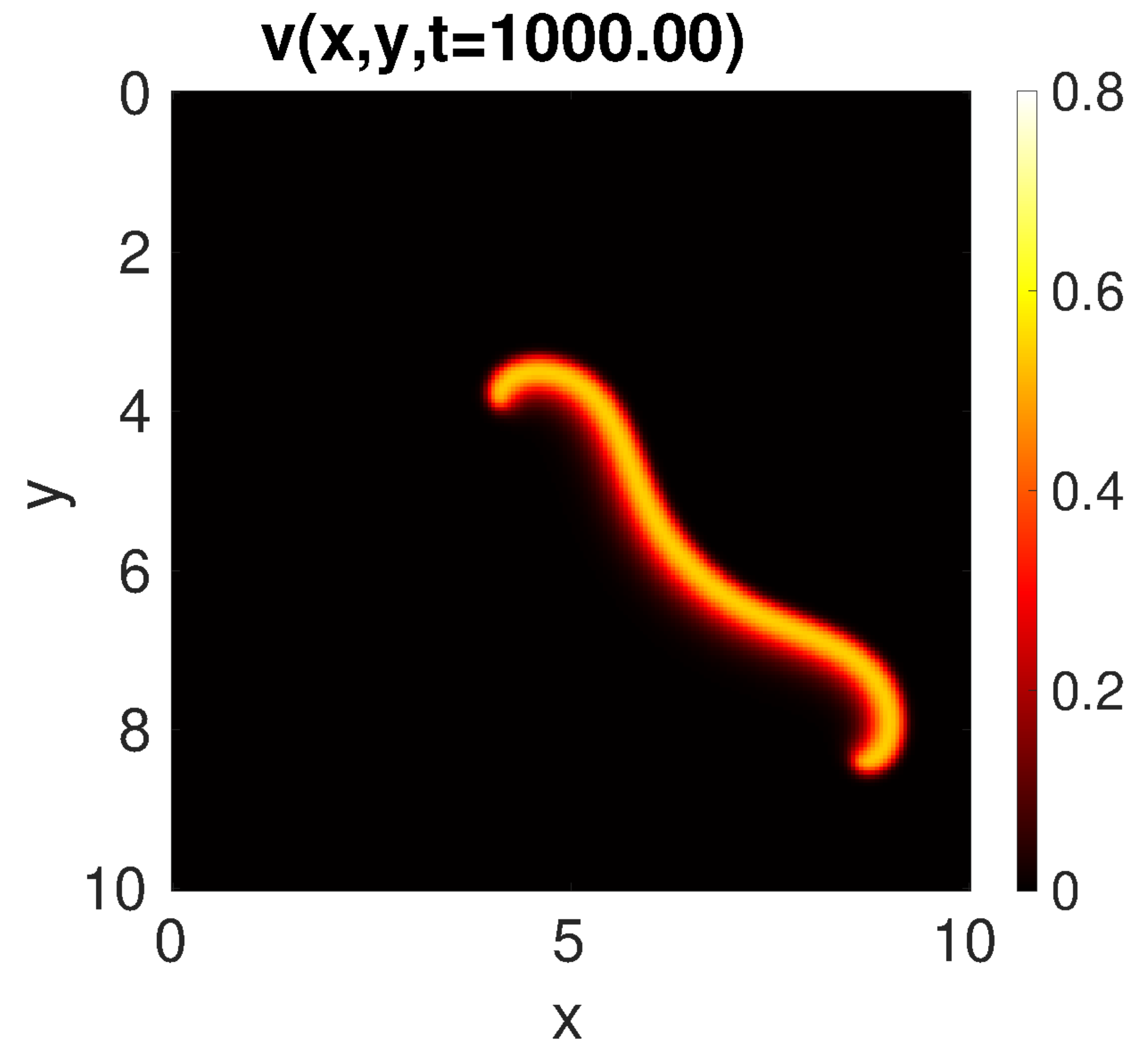}%
    \includegraphics[scale=0.12]{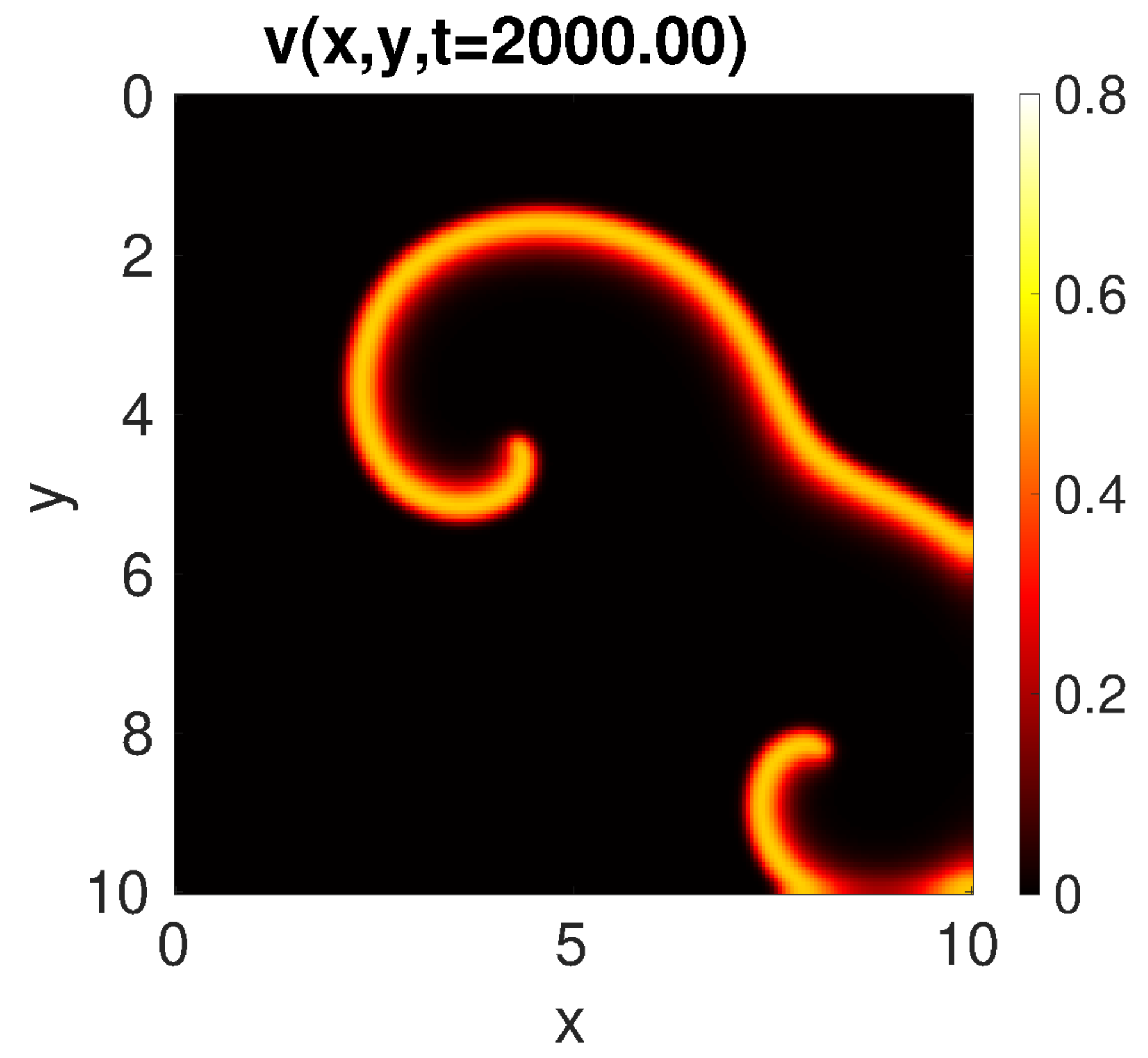}%
     \includegraphics[scale=0.12]{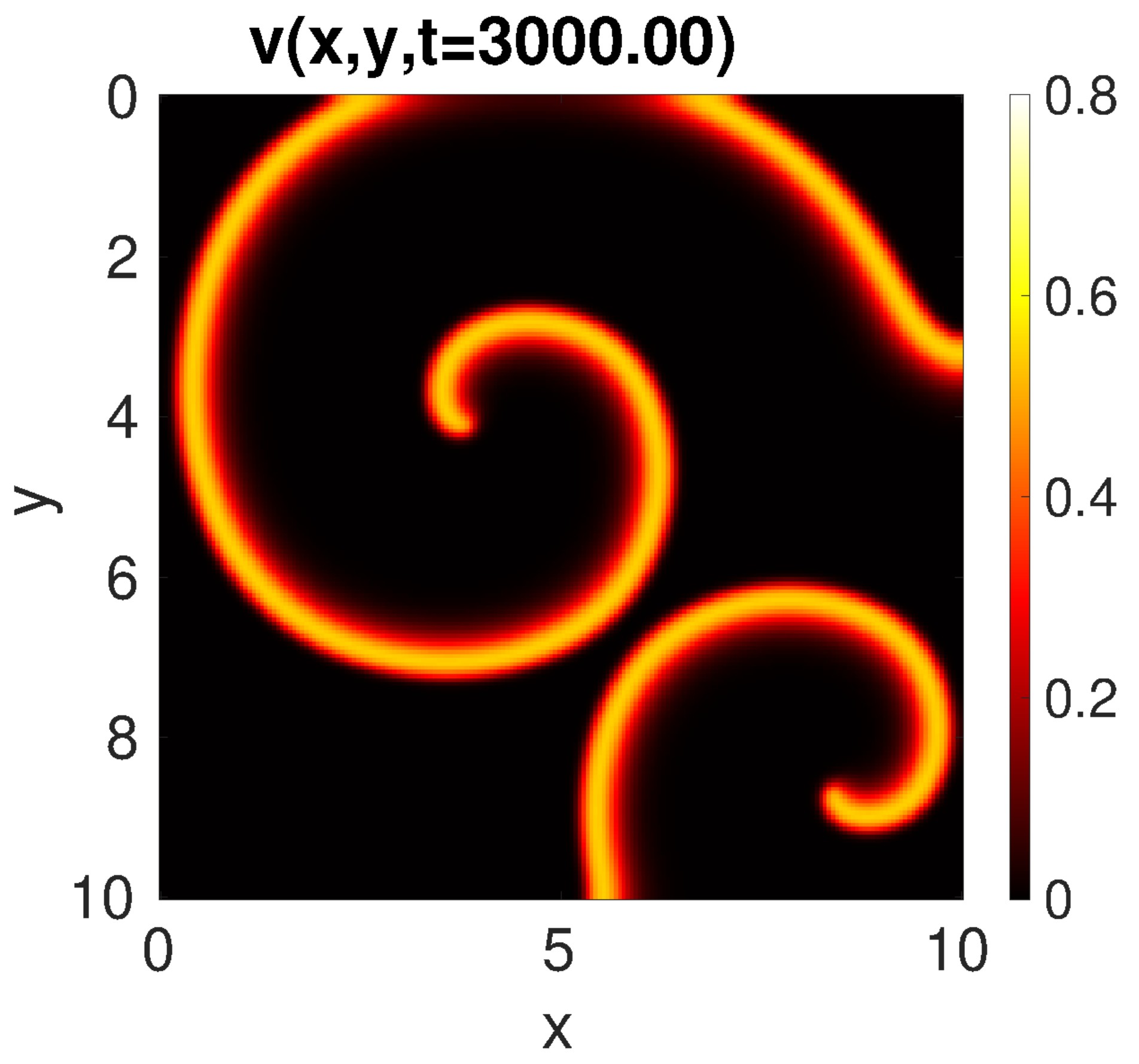}%
    \includegraphics[scale=0.12]{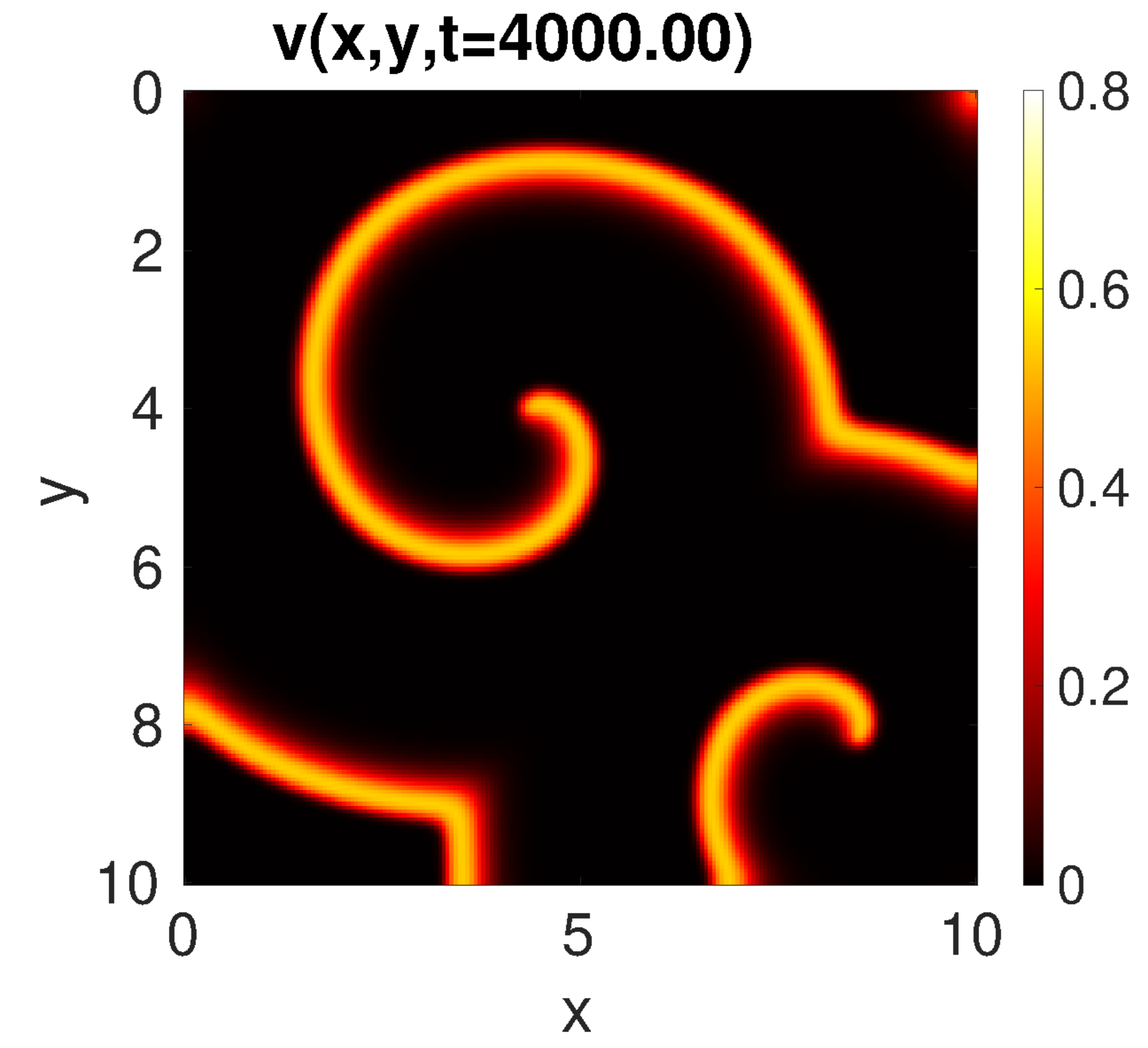}

    \caption{Row 1 show snapshots of $u$, and row 2 of $v$, at different times with $\gamma = 0.0001, h=1/20$.
}
    \label{fig:all_plots}
\end{figure}

\section{Conclusions}
In this work, we developed and analyzed the Parareal approach combined with the incomplete OSWR method for the nonlinear coupled reaction-diffusion type system, namely, the Rogers-McCulloch (RM) model. Initially, we formulated the OSWR technique specifically for the RM model and conducted a detailed convergence study. Additionally, we provided an analytical formulation for the Robin parameter used in the OSWR method.

We then introduced the nonlinear Parareal-incomplete OSWR method, a parallel space-time procedure designed to solve the RM model. The convergence results for this method were presented and validated through numerical experiments with various initial profiles. Notably, the numerical results demonstrate robust convergence behavior across different numbers of space-time subdomains, highlighting the method's reliability. This consistency underscores the effectiveness of the Parareal-incomplete OSWR method, even in complex multi-subdomain configurations within two-dimensional simulations.



\section*{Acknowledgement} 
We gratefully acknowledge the financial support under R\&D projects leading to HPC Applications (DST/NSM/R\&D\_HPC\_Applications/Extension/2023/33), National Supercomputing Mission, India. 
\subsection*{Competing interests}
 The authors declare no competing interests.

\bibliographystyle{plain}
\bibliography{PA_OSWR_bibfile}

\end{document}